\documentclass[10pt]{amsart}
\usepackage{tikz-cd}
\usepackage{amssymb}
\usepackage{mathrsfs}
\usepackage[shortlabels]{enumitem}
\usepackage[colorlinks,linkcolor=black,citecolor=black,urlcolor=black]{hyperref}
\usepackage[color = blue!20, bordercolor = black, textsize = tiny]{todonotes}
\usepackage[capitalise]{cleveref}

\numberwithin{equation}{section}
\newtheorem{thmintro}{Theorem}[section]
\newtheorem{thm}{Theorem}[subsection]

\newtheorem{lem}[thm]{Lemma}
\newtheorem{prop}[thm]{Proposition}

\theoremstyle{definition}
\newtheorem{dfintro}[thmintro]{Definition}
\newtheorem{df}[thm]{Definition}
\newtheorem{rmk}[thm]{Remark}
\newtheorem{exm}[thm]{Example}
\newtheorem{const}[thm]{Construction}

\newcommand{\A}{\mathbb{A}}

\newcommand{\C}{\mathbb{C}}

\newcommand{\G}{\mathbb{G}}

\newcommand{\N}{\mathbb{N}}
\renewcommand{\P}{\mathbb{P}}

\newcommand{\R}{\mathbb{R}}

\newcommand{\Z}{\mathbb{Z}}

\newcommand{\cO}{\mathcal{O}}

\DeclareMathOperator{\Spec}{Spec}

\newcommand{\colim}{\mathop{\mathrm{colim}}}

\newcommand{\id}{\mathrm{id}}
\newcommand{\ul}{\underline}
\newcommand{\ol}{\overline}

\newcommand{\iprod}{\mathbin{\lrcorner}}
\newcommand{\Sym}{\mathrm{Sym}}
\newcommand{\Cone}{\mathrm{Cone}}
\newcommand{\bd}{\mathbf{d}}

\newcommand{\Bl}{\mathrm{Bl}}

\newcommand{\sd}{\mathrm{sd}}
\newcommand{\ess}{\mathrm{ess}}
\newcommand{\Div}{\mathrm{Div}}
\newcommand{\PsDiv}{\mathrm{PsDiv}}

\newcommand{\bary}{\mathrm{bary}}

\DeclareMathOperator{\im}{im}

\newcommand{\op}{\mathrm{op}}

\newcommand{\res}{\mathrm{res}}

\newcommand{\CH}{\mathrm{CH}}
\newcommand{\sta}{\mathrm{sta}}
\newcommand{\Sta}{\mathrm{Sta}}
\newcommand{\BM}{\mathrm{BM}}

\usepackage[bbgreekl]{mathbbol}
\usepackage{relsize}

\DeclareSymbolFontAlphabet{\mathbb}{AMSb} 
\DeclareSymbolFontAlphabet{\mathbbl}{bbold}

\begin{document}
\title[Construction of logarithmic cohomology theories II]{Construction of logarithmic cohomology theories II: On Chow groups}
\author{Doosung Park}
\address{Department of Mathematics and Informatics, University of Wuppertal, Germany}
\email{dpark@uni-wuppertal.de}
\date{\today}
\begin{abstract}
The purpose of this second part of the series is to show a technical result on Chow groups of toric varieties.
This is a crucial ingredient for the first part.
\end{abstract}
\maketitle

\section{Introduction}

In the first part \cite{logSHF1} of the series,
we introduced a method of extending cohomology theories of schemes to fs log schemes.
As an application,
we showed that the K-theory of schemes is representable in the logarithmic motivic homotopy category \cite[Theorem C]{logSHF1},
and we construct a logarithmic cyclotomic trace \cite[Theorem F]{logSHF1},
assuming the following result:

\begin{thmintro}\label{boundarify.9}
We have $\CH^q(\C)\cong L_{C\partial}\CH^q(\square_{\C}^r)$ for all integers $q,r\geq 0$.
\end{thmintro}
This second part is devoted to proving this theorem.
In the introduction,
we will explain how to write the right-hand side purely in terms of Chow groups of toric varieties.
After that,
we will not need \cite{logSHF1} anymore,
i.e., the readers can read the proof of Theorem \ref{intro.2} below without reading \cite{logSHF1} first.

We need the following definitions for this purpose:

\begin{dfintro}
\label{intro.1}
Let $0\leq r\leq n$ be integers.
An \emph{$r$-standard subdivision of the fan $(\P^1)^n$} is a smooth subdivision $\Sigma$ of $(\P^1)^n$ satisfying the following conditions:
\begin{enumerate}
\item[(i)]
For $1\leq i\leq r$,
$e_i$ is a unique ray of $\Sigma$ whose $i$th coordinate is positive.
\item[(ii)]
$\Cone(e_{r+1},\ldots,e_n)\in \Sigma$.
\end{enumerate}
Here,
$e_1,\ldots,e_n$ denote the standard coordinates in $\Z^n$.
Let $\Sta_{n,r}$ denote the category of $r$-standard subdivisions of $(\P^1)^n$.
We refer to Proposition \ref{ordering.26} for a geometric meaning of $r$-standard subdivisions.
\end{dfintro}

\begin{dfintro}
\label{intro.3}
For a fan $\Sigma$ in a lattice $N$ and cone $\sigma$ of $\Sigma$,
let $N_\sigma$ be the sublattice of $N$ generated by $\sigma$.
Consider the projection
\[
\overline{(-)}\colon N\to N/N_\sigma.
\]
Let $V(\sigma)$ be the fan in $N/N_\sigma$ consisting of the cones $\ol{\tau}$ for $\tau\in \Sigma$.
\end{dfintro}

\begin{dfintro}
\label{intro.4}
Let $0\leq r\leq  n$ be integers.
For an $r$-standard subdivision $\Sigma$ of $(\P^1)^n$
and integer $r+1\leq i\leq n$,
let $D_{i.0}(\Sigma):=V(e_i)$,
and let $D_{i,1}(\Sigma)$ be the restriction of $\Sigma$ to the lattice $\Z^{i-1}\times 0 \times \Z^{n-i}$.
Observe that $D_{i,0}(\Sigma)$ and $D_{i,1}(\Sigma)$ are $r$-standard subdivisions of $(\P^1)^{n-1}$.
\end{dfintro}

For example,
we have $D_{i,0}((\P^1)^n)\cong D_{i,1}((\P^1)^n)\cong (\P^1)^{n-1}$.

\begin{dfintro}
\label{intro.5}
For integers $0\leq r\leq n$,
we define
\[
\CH_\sta^*(n,r)
:=
\colim_{\Sigma\in \Sta_{n,r}^\op}\CH^*(\Sigma_{\C}),
\]
where $\Sigma_\C$ denotes the toric variety over $\C$ associated with $\Sigma$.
For $r+1\leq i\leq n$,
let
\[
\delta_{i,0}^*,\delta_{i,1}^*\colon
\CH_\sta^*(n,r)\to \CH_\sta^*(n-1,r)
\]
be the maps induced by the closed immersions $D_{i,0}(\Sigma)_{\C},D_{i,1}(\Sigma)_{\C}\to \Sigma_\C$ for all $\Sigma\in \Sta_{n,r}$.
We define
\begin{gather*}
\CH_\sta^{*,\flat}(n,r)
:=
\bigcap_{i=r+1}^n
\ker \delta_{i,0}^*,
\\
\delta^*:=\sum_{i=r+1}^n (-1)^{i-r}\delta_{i,1}^*\colon \CH_\sta^{*,\flat}(n,r)
\to
\CH_\sta^{*,\flat}(n,r-1).
\end{gather*}
Recall from \cite[Remark 7.8]{logSHF1} (after adjusting $n$ and $r$) that we have maps
\begin{gather*}
p_i^*\colon \CH_\sta^*(n,r)\to \CH_\sta^*(n+1,r) \text{ for } r+1\leq i\leq n+1,
\\
\mu_i^*\colon \CH_\sta^*(n,r)\to \CH_\sta^*(n+1,r) \text{ for }r+1\leq i\leq n,
\end{gather*}
and there are relations
\begin{gather}
\label{intro.5.1}
\delta_{i,\epsilon}^*\delta_{j,\epsilon'}^*
=
\delta_{j,\epsilon'}^*\delta_{i+1,\epsilon}^*
\text{ if }j\leq i,
\\
\label{intro.5.2}
\delta_{i,\epsilon}^*p_j^*
=
\left\{
\begin{array}{ll}
p_j^*\delta_{i-1,\epsilon}^*& \text{if $j<i$},
\\
\id& \text{if $j=i$},
\\
p_{j-1}^*\delta_{i,\epsilon}^*& \text{if $j>i$},
\end{array}
\right.
\\
\label{intro.5.3}
\delta_{i,\epsilon}^*\mu_j^* 
=
\left\{
\begin{array}{ll}
\mu_j^*\delta_{i-1,\epsilon}^*& \text{if $j<i-1$},
\\
p_j^*\delta_{j,0}^*& \text{if $j=i-1,i$ and $\epsilon=0$},
\\
\id& \text{if $j=i-1,i$ and $\epsilon=1$},
\\
\mu_{j-1}^*\delta_{i,\epsilon}^*& \text{if $j>i$}.
\end{array}
\right.
\end{gather}
We will not need the explicit descriptions of $p_i^*$ and $\mu_i^*$.
\end{dfintro}

With all these definitions,
according to \cite[Remark 7.8]{logSHF1} (after adjusting $n$ and $r$),
Theorem \ref{boundarify.9} is equivalent to the following:

\begin{thmintro}
\label{intro.2}
For all integers $q,r\geq 0$,
the complex
\[
\cdots \xrightarrow{\delta^*}
\CH_\sta^{q,\flat}(r+1,r)
\xrightarrow{\delta^*}
\CH_\sta^{q,\flat}(r,r)
\to
0
\]
is quasi-isomorphic to 
$\CH^q(\C)$,
where $\CH_\sta^{q,\flat}(r,r)$ sits in degree $0$.
\end{thmintro}

The proof is finished at the end of this paper.
Let us explain the outline of the proof.
For $x\in \CH_\sta^{q,\flat}(n,r)$ with an integer $n\geq r$ such that $\delta^*x=0$,
we need to show that $x=\delta^*y$ for some $y\in \CH_\sta^{q,\flat}(n+1,r)$.
To prove this claim,
we may assume that $x$ comes from an element of $\CH^q(\Sigma):=\CH^q(\Sigma_\C)$ for some $r$-standard subdivision $\Sigma$ of $(\P^1)^n$.
We divide the proof into three parts.

Part I. \emph{Construction of $\Theta_{n,r,\bd}$}.
In \S \ref{dist} and \ref{subdivision},
we construct a specific $r$-standard subdivision $\Theta_{n,r,\bd}$ of $(\P^1)^n$,
where $\bd$ is a finite sequence of positive integers.
This is sufficiently fine in the following sense:
For every $r$-standard subdivision $\Sigma$ of $(\P^1)^n$,
there exists a sequence of star subdivisions 
\begin{equation}
\label{intro.0.1}
\Sigma_m\to \cdots \to \Sigma_0=\Theta_{n,r,\bd}
\end{equation}
for some nonempty finite sequence $\bd$ of positive integers such that $\Sigma_m$ is a subdivision of $\Sigma$ and for each $i$,
$\Sigma_{i+1}$ is the star subdivision of $\Sigma_i$ relative to a $2$-dimensional cone $\sigma_i$ satisfying $e_1,\ldots,e_n\notin \sigma_i$.

In  \S \ref{subdivision},
we show that $\Theta_{n,r,\bd}$ is a very $r$-standard subdivision,
which is a technical property used in \S \ref{induction}.
An $r$-standard subdivision $\Sigma$ of $(\P^1)^n$ is called very $r$-standard if $I:=\{i_1,\ldots,i_s\}$ is a subset of $\{r+1,\ldots,n\}$ such that
\[
\Cone( \sigma,e_{i_1}),\ldots,\Cone( \sigma,e_{i_s}) \in \Sigma,
\]
then $\Cone( \sigma,e_{i_1},\ldots,e_{i_s}) \in \Sigma$.
For example,
if $\Sigma$ is a $0$-standard subdivision of $(\P^1)^3$ containing the left subfan in the following figure with some rays $\alpha$ and $\beta$,
then $\Sigma$ is not a very $0$-standard subdivision since we have $\Cone(e_1,\beta),\Cone(e_2,\beta)\in \Sigma$ but $\Cone(e_1,e_2,\beta)\notin \Sigma$.
We further need to subdivide as the right subfan for the condition of very $0$-standard subdivisions.
\[
\begin{tikzpicture}[scale = 0.8]
\draw (2,0)--(0,3.5)--(-2,0)--(2,0);
\draw (2,0) node[below right] {$\beta$};
\draw (-2,0) node[below left] {$e_2$};
\draw (0,3.5) node[above] {$e_1$};
\draw (2,0)--(0,3.5/3);
\draw (-2,0)--(0,3.5/3);
\draw (0,3.5)--(0,3.5/3);
\draw (0,3.2/3) node[above left] {$\alpha$};
\begin{scope}[shift={(6,0)}]
\draw (2,0)--(0,3.5)--(-2,0)--(2,0);
\draw (2,0) node[below right] {$\beta$};
\draw (-2,0) node[below left] {$e_2$};
\draw (0,3.5) node[above] {$e_1$};
\draw (-2,0)--(1,1.75);
\draw (0,0)--(0,3.5);
\draw (0,3.5/3)--(2,0);
\draw (0,0) node[below] {$e_2+\beta$};
\draw (1,1.7) node[above right] {$e_1+\beta$};
\draw (0,3.2/3) node[above left] {$\alpha$};
\end{scope}
\end{tikzpicture}
\]

Part II. \emph{Proof for $\Theta_{n,r,\bd}$}.
The purpose of \S \ref{ordering} and \ref{resolution} is to prove the main claim for this specific $\Theta_{n,r,\bd}$.
We need two different computational methods for Chow groups of toric varieties: finding a basis of the Chow group given a nice ordering of the maximal cones \cite[Theorem in p.\ 102]{Fulton:1436535}, and a spectral sequence converging to the Chow group \cite[\S 5]{MR3264256}.
For an $r$-standard subdivision $\Sigma$ of $(\P^1)^n$ and integer $p$,
let
\begin{gather*}
\Sigma^\circ
:=
\{\sigma\in \Sigma:\Cone(\sigma,e_i)\in \Sigma \text{ for some }r+1\leq i\leq n\},
\\
\CH^\flat_p(\Sigma):=\CH_p(\Sigma_{\C}-\Sigma_{\C}^\circ).
\end{gather*}
In Proposition \ref{ordering.22},
we show that the maximal cones of $\Theta_{n,r,\bd}$ admit an ordering such that \cite[Theorem in p.\ 102]{Fulton:1436535} is applicable to $\Theta_{n,r,\bd}$ and $\Theta_{n,r,\bd}^\circ$, and hence we obtain explicit bases of $\CH_p(\Theta_{n,r,\bd})$ and $\CH_p^{\flat}(\Theta_{n,r,\bd})$.
After that,
we use  \cite[\S 5]{MR3264256} to obtain an explicit free resolution
\[
\cdots \xrightarrow{d} Z_{p,1}^\flat(\Theta_{n,r,\bd}) \xrightarrow{d} Z_{p,0}^\flat(\Theta_{n,r,\bd})
\to \CH_p^\flat(\Theta_{n,r,\bd})\to 0
\]
in Proposition \ref{ordering.16}.
We also construct maps $\delta_{i,1}^*$, $\rho_i^*$, and $\nu_i^*$ for $Z_{p,q}^\flat(\Theta_{n,r,\bd})$ enjoying cubical identities.
Using these identities,
we complete the proof of the claim for $\Theta_{n,r,\bd}$ in Proposition \ref{ordering.15}.

Part III. \emph{Proof for general $\Sigma$.}
Due to \eqref{intro.0.1} and Part II,
by induction,
it suffices to show that the claim for $\Sigma_i$ with $0\leq i<m$ implies the claim for $\Sigma_{i+1}$.
There exists a smooth blow-up square of toric schemes
\begin{equation}
\label{intro.0.2}
\begin{tikzcd}
(\Delta_i')_\C\ar[d]\ar[r]&
(\Sigma_{i+1})_\C\ar[d]
\\
(\Delta_i)_\C\ar[r]&
(\Sigma_i)_\C
\end{tikzcd}
\end{equation}
for some fans $\Delta_i$ and $\Delta_i'$,
and then we can use the blow-up formula for Chow groups.
Since $\Delta_i$ is no longer an $r$-standard subdivision,
the methods in the earlier sections do not apply to $\Delta_i$.
In \S \ref{admissible},
we introduce the notion of admissible subdivisions,
and $\Delta_i$ is an admissible subdivision as shown in Construction \ref{subdivision.27}.
Here,
we need that $\Sigma_i$ is a very $r$-standard subdivision,
which is a consequence of Part I.
After establishing cubical identities in Lemma \ref{subdivision.18} and constructing classes that behave like homotopies in Lemmas \ref{subdivision.21} and \ref{subdivision.22} for admissible subdivisions and hence for $\Delta_i$ too,
we finish the proof of Theorem \ref{intro.2} in \S \ref{induction} by investigating the Chow groups of the varieties appearing in \eqref{intro.0.2}.

\subsection*{Notation and convention}

We refer to \cite{MR495499}, \cite{Fulton:1436535}, and \cite{CLStoric} for toric geometry.
For a fan $\Sigma$ and its cone $\sigma$,
we employ the following notation throughout the paper.

\begin{tabular}{l|l}
$N(\Sigma)$ & lattice of $\Sigma$
\\
$M(\Sigma)$ & dual lattice of $\Sigma$
\\
$N_\sigma$ & sublattice of $N(\Sigma)$ generated by $\sigma$
\\
$\sigma^\vee$ & dual cone of $\sigma$
\\
$\sigma^\bot$ & set of $y\in M(\Sigma)$ such that $\langle x,y\rangle =0$ for all $x\in \sigma$
\\
$M(\sigma)$ & lattice $\sigma^\bot \cap M(\Sigma)$
\\
$\Sigma(d)$ & set of $d$-dimensional cones of $\Sigma$
\\
$\Sigma_{\max}$ & set of maximal cones of $\Sigma$
\\
$\lvert \sigma \rvert$ & support of $\sigma$
\\
$\lvert \Sigma \rvert$ & support of $\Sigma$
\\
$\Sigma^*(\sigma)$ & star subdivision of $\Sigma$ relative to $\sigma$
\\
$e_1,\ldots,e_n$ & standard coordinates in $\Z^n$
\\
$\Cone(x_1,\ldots,x_r)$ & cone generated by $x_1,\ldots,x_r$ in a lattice.
\\
$\Cone(\sigma_1,\ldots,\sigma_r)$ &smallest cone containing cones $\sigma_1,\ldots,\sigma_r$
\\
$\tau\prec \sigma$ & $\tau$ is a face of $\sigma$
\\
$\tau\prec_1 \sigma$ & $\tau$ is a facet of $\sigma$, i.e., a codimension $1$ face.
\end{tabular}

\

We often write a ray (i.e., a $1$-dimensional cone) as an element of $\Z^n$ such that the $\gcd$ of the coordinates is $1$.
A ray of a cone means a $1$-dimensional face.
We have the fan $\A^1$ (resp.\ $\P^1-0$, resp.\ $\G_m$) in $\Z$ whose single maximal cone is $\Cone(e_1)$ (resp.\ $\Cone(-e_1)$, resp.\ the zero cone).
We have the fan $\P^n$ in $\Z^n$ for every integer $n\geq 0$ whose maximal cones are $\Cone(e_1,\ldots,e_n)$ and
\[
\Cone(e_1,\ldots,e_{i-1},e_{i+1},\ldots,e_n,-e_1-\cdots-e_n)
\]
for all integers $1\leq i\leq n$.

\section{Sufficiently fine subdivisions}
\label{dist}

In algebraic topology,
one can repeat barycentric subdivisions to a simplex to yield a sufficiently fine simplicial complex whose simplices have arbitrarily small diameters.
However, that does not work in toric geometry,
see Remark \ref{dist.14} below.
Instead,
we will repeat barycentric subdivisions relative to suitable subfans (the usual barycentric subdivision is the barycentric subdivision relative to itself).
See Proposition \ref{dist.10} below for the main result of this section.
The reader who accepts this result can skip the technical details of \S \ref{repeat}.

\subsection{\texorpdfstring{$\eta$}{eta}-excluded barycentric subdivision}

We will need a variant of the barycentric subdivision as follows.

\begin{df}
\label{subdivision.29}
Consider $\Sigma:=\A^n$ and its cones $\eta:=\Cone(e_1,\ldots,e_t)$ and $\tau:=\Cone(e_1,\ldots,e_m)$,
where $0\leq t\leq m\leq n$ are integers.
Let $\sigma_1,\ldots,\sigma_p$ be a list of the faces of $\tau$ not contained in $\eta$
such that $2\leq \dim \sigma_1\leq \cdots \leq \dim \sigma_p$.
Then the \emph{$\eta$-excluded barycentric subdivision of $\Sigma$ relative to $\tau$} is
\[
\Sigma_\eta^\bary(\tau)
:=
(\cdots (\Sigma^*(\sigma_p))^*(\sigma_{m-1})\cdots )^*(\sigma_1).
\]
Note that it is the usual barycentric subdivision of $\Sigma$ \cite[Exercise 11.1.10]{CLStoric} if $\eta=0$ and $\tau=\Cone(e_1,\ldots,e_n)$.

The description in Proposition \ref{subdivision.32} shows that the $\eta$-excluded barycentric subdivision of $\A^n$ is independent of the ordering on the list $\sigma_1,\ldots,\sigma_p$ as long as the inequalities $\dim \sigma_1\leq \cdots \leq \dim \sigma_p$ are satisfied.
\end{df}

\begin{df}
\label{subdivision.31}
Let $0\leq t\leq m$ be integers.
A \emph{$t$-admissible permutation of $\{1,\ldots,m\}$} is a permutation $\alpha$ of $\{1,\ldots,m\}$ satisfying the following condition:
If $t<m$ and $c$ is the smallest integer satisfying $\alpha(c+1)>t$,
then $\alpha(1)<\cdots<\alpha(c)$.
If $t=m$, then $\alpha=\id$.
\end{df}

The following combinatorial description of the maximal cones of $\Sigma_{\eta}^\bary(\tau)$ will be useful later especially in \S \ref{repeat}.

\begin{prop}
\label{subdivision.32}
Consider $\Sigma:=\A^n$ and its cones $\eta:=\Cone(e_1,\ldots,e_t)$ and $\tau:=\Cone(e_1,\ldots,e_m)$,
where $0\leq t\leq m\leq n$ are integers.
Then for every maximal cone $\sigma$ of $\Sigma_\eta^\bary(\tau)$,
there exists a unique $t$-admissible permutation $\alpha$ of $\{1,\ldots,n\}$ such that $\sigma$ is equal to
\[
\sigma_\alpha
:=
\Cone(e_{\alpha(1)},\ldots,e_{\alpha(c)};e_{\alpha(1)}+\cdots +e_{\alpha(c+1)},\ldots,e_{\alpha(1)}+\cdots + e_{\alpha(m)} ;e_{m+1},\ldots,e_n),
\]
where $c$ is the smallest integer satisfying $\alpha(c+1)>t$.
\end{prop}
\begin{proof}
For an integer $1\leq d\leq m$,
let $\sigma_{p_d+1},\ldots,\sigma_p$ be the faces of $\tau$ not contained in $\eta$ such that $d+1\leq \dim \sigma_{p_d+1}\leq \cdots \leq \dim \sigma_p$.
Consider
\[
\Sigma_d
:=
(\cdots (\Sigma^*(\sigma_p))^*(\sigma_{p-1})\cdots )^*(\sigma_{p_d+1}).
\]
For a permutation $\alpha$ on $\{1,\ldots,m\}$ and integer $1\leq c\leq m$,
let $\sigma_{\alpha,c}$ denote
\[
\Cone(e_{\alpha(1)},\ldots,e_{\alpha(c)};e_{\alpha(1)}+\cdots + e_{\alpha(c+1)},\ldots,e_{\alpha(1)}+\cdots+e_{\alpha(m)} ;e_{m+1},\ldots,e_n).
\]
We claim that the maximal cones of $\Sigma_d$ are of the form $\sigma_{\alpha,c}$ for $\alpha$ and $c$ satisfying one of the following conditions:
\begin{enumerate}
\item[(i)]
$c\geq d$, $\alpha(1),\ldots,\alpha(c)\leq t$, and $\alpha(c+1)>t$.
\item[(ii)]
$c=d$ and $\alpha(i)>t$ for some integer $1\leq i\leq d$.
\end{enumerate}

We proceed by decreasing induction on $d$.
The claim is obvious if $d=m$.
Assume that the claim holds for $d+1$.
Consider a maximal cone $\sigma_{\alpha,c}$ of $\Sigma_{d+1}$ such that $\alpha$ and $c$ satisfy one of the conditions (i) and (ii) for $d+1$.

Assume that $\alpha$ and $c$ satisfy the condition (i) for $d+1$.
Then we have $\sigma_{\alpha,c}\in \Sigma_d$, i.e., $\sigma_{\alpha,c}$ is not divided in $\Sigma_d$.

Assume that $\alpha$ and $c$ satisfy the condition (ii) for $d+1$.
Then we have $c=d+1$,
and $\sigma_{\alpha,c}$ is divided into $d$ cones obtained by replacing one of $e_{\alpha(1)},\ldots,e_{\alpha(c)}$, say $e_{\alpha(j)}$, by $e_{\alpha(1)}+\cdots+e_{\alpha(c)}$.
The resulting cone is
\begin{align*}
\Cone(& e_{\alpha(1)},\ldots,e_{\alpha(j-1)},e_{\alpha(j+1)},\cdots,e_{\alpha(c)};
\\
& e_{\alpha(1)}+\cdots+e_{\alpha(c)},\ldots,e_{\alpha(1)}+\cdots+e_{\alpha(n)};e_{m+1},\ldots,e_n),
\end{align*}
which is $\sigma_{\beta,c-1}$ for the permutation $\beta$ of $\{1,\ldots,m\}$ given by
\[
\beta(i)
:=
\left\{
\begin{array}{ll}
\alpha(i) & \text{if $i\leq j-1$ or $i\geq c+1$},
\\
\alpha(i+1) & \text{if $j\leq i\leq c-1$},
\\
\alpha(j) & \text{if $i=c$}.
\end{array}
\right.
\]
If $\alpha(j)>t$ and $\alpha(i)\leq t$ for all $i\in \{1,\ldots,c\}-\{j\}$, then $\sigma_{\beta,c-1}$ satisfies the condition (i) for $d$.
Otherwise,
$\sigma_{\beta,c-1}$ satisfies the condition (ii) for $d$.

To finish the induction argument,
observe that every pair of $\beta$ and $c-1$ satisfying one of the conditions (i) and (ii) for $d$ can be obtained by a pair of $\alpha$ and $c$ satisfying one of the conditions (i) and (ii) for $d+1$.

Now, assume $d=1$,
Then we have $\Sigma_1=\Sigma_\eta^\bary(\tau)$.
For a pair of $\alpha$ and $c$ satisfying on the conditions (i) and (ii) for $d=1$,
there exists a unique permutation $\gamma$ on $\{1,\ldots,m\}$ such that $\{\alpha(1),\ldots,\alpha(c)\}=\{\gamma(1),\ldots,\gamma(c)\}$ and $\gamma(1)<\cdots < \gamma(c)$.
The conditions (i) and (ii) imply that $\gamma$ is a $t$-admissible permutation.
To conclude,
observe that we have $\sigma_{\alpha,c}=\sigma_\gamma$.
\end{proof}

\begin{exm}
\label{dist.13}
We illustrate the $\eta$-excluded barycentric subdivision of $\A^3$ with $\eta:=\Cone(e_1,e_2)$ as follows:
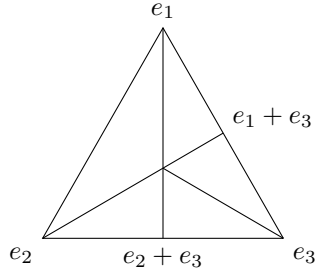
\[
\begin{tikzpicture}[scale = 0.8]
\draw (2,0)--(0,3.5)--(-2,0)--(2,0);
\draw (2,0) node[below right] {$e_3$};
\draw (-2,0) node[below left] {$e_2$};
\draw (0,3.5) node[above] {$e_1$};
\draw (-2,0)--(1,1.75);
\draw (0,0)--(0,3.5);
\draw (0,3.5/3)--(2,0);
\draw (0,0) node[below] {$e_2+e_3$};
\draw (1,1.7) node[above right] {$e_1+e_3$};
\end{tikzpicture}
\]
The $2$-admissible permutations of $\{1,2,3\}$ are
\[
(\alpha(1),\alpha(2),\alpha(3))
=
(1,2,3),(1,3,2),(2,3,1),(3,1,2),(3,2,1).
\]
These correspond to the maximal cones
\begin{gather*}
\Cone(e_1,e_2,e_1+e_2+e_3),
\text{ }
\Cone(e_1,e_1+e_3,e_1+e_2+e_3),
\\
\Cone(e_2,e_2+e_3,e_1+e_2+e_3),
\text{ }
\Cone(e_3,e_1+e_3,e_1+e_2+e_3),
\\
\Cone(e_3,e_2+e_3,e_1+e_2+e_3).
\end{gather*}
\end{exm}

The above $\Sigma_\eta^\bary(\tau)$ was only defined for $\Sigma=\A^n$,
and now we provide a definition for general $\Sigma$.

\begin{df}
\label{dist.12}
Let $\Sigma$ be a fan,
let $\eta$ be its cone,
and let $A$ be a subfan of $\Sigma$.
Assume that for every $\sigma\in \Sigma$,
there exists $a(\sigma)\in A$ containing every other cone $\tau$ of $A$ such that $\tau\subset \sigma$.
The \emph{$\eta$-excluded barycentric subdivision of $\Sigma$ relative to $A$}, denoted $\Sigma_\eta^\bary(A)$,
is obtained by replacing every cone $\sigma\in \Sigma$ by its $\eta\cap \sigma$-excluded barycentric subdivision relative to $a(\sigma)$.
\end{df}

\begin{exm}
Consider the fan in Example \ref{dist.13}.
Note that its subfan $A$ with the maximal cones
\[
\Cone(e_1,e_1+e_3),\Cone(e_1+e_3,e_3),\Cone(e_2,e_2+e_3),\Cone(e_2+e_3,e_3)
\]
satisfies the condition in Definition \ref{dist.12}.
\end{exm}

\subsection{Repetition of \texorpdfstring{$\eta$}{eta}-excluded barycentric subdivisions}
\label{repeat}
The barycentric subdivision in algebraic topology provides a sufficiently fine subdivision in the following sense: Given $\epsilon>0$, one can repeat barycentric subdivisions to an $n$-simplex such that every small simplex is contained in a ball with radius $\epsilon$.

In this subsection,
we will do a similar task for fans.
The repetition of barycentric subdivisions in the usual sense does not work as explained in Remark \ref{dist.14}.
Instead, we introduce $\sd_{\eta,u,d}^s$ in Definition \ref{dist.1},
which yields a sufficiently fine subdivision after repetition as shown in Proposition \ref{dist.10}.
The notion of a ball with radius $\epsilon$ is replaced with $H_\epsilon^{\geq 0}$ in Definition \ref{dist.15}. 

All we need to remember from this subsection are basically Definition \ref{dist.1} and Proposition \ref{dist.10}.

\begin{df}
\label{dist.1}
Let $\Sigma$ be an $n$-dimensional fan with a cone $\eta$,
where $n\in \N$.
Consider an integer $2\leq u\leq n$.
Let $\sd_{\eta,u,1}(\Sigma)$ be the $\eta$-excluded barycentric subdivision of $\Sigma$ relative to itself.
For an integer $d\geq 1$,
we inductively define
\[
\sd_{\eta,u,d+1}(\Sigma)
:=
(\sd_{\eta,u,d}(\Sigma))_\eta^\bary(A_{u,d}),
\]
where $A_{u,d}$ is the subfan of $\sd_{\eta,u,d}(\Sigma)$ generated by the cones of $\sd_{\eta,u,d}(\Sigma)$ contained in some $u$-dimensional cone of $\Sigma$.
Note that $A_{u,d}$ satisfies the condition in Definition \ref{dist.12}.
To show this,
note that for every cone $\sigma\in \sd_{\eta,u,d}(\Sigma)$,
\[
a(\sigma)
:=
\lvert \sigma \rvert \cap \bigcup_{\tau \in \Sigma(u)} \lvert \tau \rvert
\]
is a face of $\sigma$ and satisfies the condition in Definition \ref{dist.12}.

For an integer $s\geq 0$,
let
\[
\sd_{\eta,u,d}^s
:=(\sd_{\eta,u,d})^{s}
\]
denote the $s$th iterate. of $\sd_{\eta,u,d}$.
By convention,
$\sd_{\eta,u,d}^0(\Sigma):=\Sigma$.
\end{df}

\begin{exm}
We illustrate $\sd_{\eta,2,2}(\A^3)$ with $\eta:=\Cone(e_1,e_2)$ as follows:
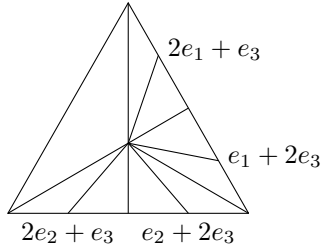
\[
\begin{tikzpicture}[scale = 0.8]
\draw (2,0)--(0,3.5)--(-2,0)--(2,0);
\draw (-2,0)--(1,1.75);
\draw (0,0)--(0,3.5);
\draw (0,3.5/3)--(2,0);
\draw (-1,0) node[below] {$2e_2+e_3$};
\draw (1,0) node[below] {$e_2+2e_3$};
\draw (0.5,2.4) node[above right] {$2e_1+e_3$};
\draw (1.5,0.95) node[right] {$e_1+2e_3$};
\draw (0,3.5/3)--(-1,0);
\draw (0,3.5/3)--(1,0);
\draw (0,3.5/3)--(0.5,10.5/4);
\draw (0,3.5/3)--(1.5,3.5/4);
\end{tikzpicture}
\]
\end{exm}

For the convenience of writing, we introduce the following.

\begin{df}
Let $I=\{i_1,\ldots,i_s\}$ be a subset of $\{1,\ldots,n\}$,
where $n\in \N$.
A ray $f$ in $\N^n$ is \emph{$I$-rigid} if for all $i,j\in I$,
the $i$th and $j$th coordinates of $f$ are equal.
A cone $\sigma$ in $\N^n$ is \emph{$I$-rigid} if
\[
\sigma=\Cone(e_{i_1},\ldots,e_{i_s},f_{s+1},\ldots,f_n)
\]
for some $I$-rigid rays $f_{s+1},\ldots,f_n\in \N^n$.
\end{df}

We have the following result on the structure of $\sd_{\eta,u,d}(\A^n)$.

\begin{prop}
\label{dist.8}
Consider the cone $\eta:=\Cone(e_1,\ldots,e_t)$ of $\Sigma:=\A^n$,
where $0\leq t\leq n$ are integers.
Let $2\leq u\leq n$ and $d\geq 1$ be integers.
Then every maximal cone $\sigma$ of $\sd_{\eta,u,d}(\Sigma)$ is of the form
\begin{equation}
\label{dist.8.1}
\sigma=\Cone(e_{i_1},\ldots,e_{i_s},f_{s+1},\ldots,f_n)
\end{equation}
such that for every integer $s+1\leq i\leq n$,
$f_i$ is not contained in any $(i-1)$-dimensional cone of $\A^n$ and
\[
\Cone(e_{i_1},\ldots,e_{i_s},f_{s+1},\ldots,f_i)
\]
is contained in an $i$-dimensional cone of $\A^n$.
Furthermore, $\sigma$ is $\{i_1,\ldots,i_s\}$-rigid.
\end{prop}
\begin{proof}
We proceed by induction on $d$.
The claim follows from Proposition \ref{subdivision.32} if $d=1$.
Assume that the claim holds for $d$.
Let $\tau$ be a maximal cone of $\sd_{\eta,u,d+1}(\Sigma)$,
and let $\sigma$ be the unique maximal cone of $\sd_{\eta,u,d}(\Sigma)$ containing $\tau$.
Without loss of generality,
we may assume that $\sigma$ is of the form \eqref{dist.8.1} with $i_1=1$, $\ldots$, $i_s=s$.
We also set $f_i:=e_i$ for $1\leq i\leq s$.

Using the assumption on $f_i$,
to determine $\tau$,
we can apply Proposition \ref{subdivision.32} to $\Cone(f_1,\ldots,f_u)$ and add the part $\Cone(f_{u+1},\ldots,f_n)$.
Hence there exists an $s$-admissible permutation $\alpha$ of $\{1,\ldots,u\}$ such that $\tau$ is of the form
\begin{equation}
\label{dist.8.2}
\Cone(f_{\alpha(1)},\ldots,f_{\alpha(c)} ;f_{\alpha(1)}+\cdots+f_{\alpha(c+1)},\ldots,f_{\alpha(1)}+\cdots+f_{\alpha(u)} ;f_{u+1},\ldots,f_n),
\end{equation}
where $c$ is the smallest integer satisfying $\alpha(c+1)>s$.
Let us write this cone as $\Cone(g_1,\ldots,g_n)$.

We set $J:=\{\alpha(1),\ldots,\alpha(c)\}$.
By the inductive hypothesis, $f_{s+1},\ldots,f_n$ are $\{1,\ldots,s\}$-rigid and hence $J$-rigid.
For $i\in \{1,\ldots,s\}-J$,
$f_i=e_i$ is $J$-rigid.
Since $f_{\alpha(1)}+\cdots+f_{\alpha(c)}$ is $J$-rigid too,
we see that $\tau$ is $J$-rigid.

For an integer $c+1\leq i\leq n$,
$\Cone(g_1,\ldots,g_i)$ is contained in an $i$-dimensional cone of $\A^n$ since it is contained in $\Cone(f_1,\ldots,f_s)$.
If $i\leq u$,
then there exists an integer $c+1\leq j\leq i$ such that $\alpha(j)\geq i$,
so $g_i$ is not contained in any $(i-1)$-dimensional cone of $\A^n$ since $f_j$ is not contained in any $(j-1)$-dimensional cone of $\A^n$.
If $i>u$,
then $g_i=f_i$ is not contained in any $(i-1)$-dimensional cone of $\A^n$.
Hence the claim holds for $d+1$.
\end{proof}

Next, we introduce $H_\epsilon^{\geq 0}(\eta,\Sigma)$, which behaves like a ball with radius $\epsilon$.

\begin{df}
\label{dist.15}
Consider the cone $\eta:=\Cone(e_1,\ldots,e_t)$ in $\Sigma:=\A^n$,
where $1\leq t<n$ are integers.
For a real number $\epsilon>0$,
we set
\begin{gather*}
H_\epsilon^{\geq 0}(\eta,\Sigma)
:=
\{(a_1,\ldots,a_n)\in \R_{\geq 0}^n
:
\epsilon(a_1+\cdots+a_t)
\geq
a_{t+1}+\cdots+a_n\},
\\
H_\epsilon^{\leq 0}(\eta,\Sigma)
:=
\{(a_1,\ldots,a_n)\in \R_{\geq 0}^n
:
\epsilon(a_1+\cdots+a_t)
\leq
a_{t+1}+\cdots+a_n\},
\end{gather*}
and $H_\epsilon^0(\eta,\Sigma):=H_\epsilon^{\geq 0}(\eta,\Sigma)\cap H_\epsilon^{\leq 0}(\eta,\Sigma)$.
\end{df}

In algebraic topology, every simplex $\Delta'$ of the barycentric subdivision of an $n$-simplex $\Delta$ enjoys the inequality of diameters
\[
\mathrm{diam}(\Delta')
\leq 
\frac{n}{n+1}\,\mathrm{diam}(\Delta).
\]
We provides an analogous but more sophisticated inequality for the case of fans as follows.

\begin{lem}
\label{dist.2}
Consider the cone $\eta:=\Cone(e_1,\ldots,e_t)$ in $\A^n$,
where $1\leq t< n$ are integers.
Then for every integer $d\geq 1$ and maximal cone $\sigma\in \sd_{\eta,t+1,d}^{n-t}(\A^n)$ such that $\eta\subset \sigma$,
we have
\[
\lvert \sigma \rvert
\subset
H_{(n-t)^{n-t}/dt}^{\geq 0}(\eta,\A^n).
\]
\end{lem}
\begin{proof}
Consider the functions
\[
\varphi_1(a_1,\ldots,a_n):=a_1+\cdots+a_t,
\text{ }
\varphi_2(a_1,\ldots,a_n):=a_{t+1}+\cdots+a_n.
\]
Consider a maximal cone $\sigma \in \sd_{\eta,t+1,1}(\A^n)$ such that $\eta\subset \sigma$.
Then by Proposition \ref{subdivision.32},
$\sigma$ is of the form
\[
\Cone(e_1,\ldots,e_t,f_{t+1},\ldots,f_n)
\]
such that $\varphi_1(f_i)=t$ and $\varphi_2(f_i)=i-t$ for $t+1\leq i\leq n$.
Furthermore,
$\Cone(e_1,\ldots,e_t,f_{t+1})$ is contained in an $(t+1)$-dimensional cone of $\A^n$,
and $f_i$ is not contained in any $(t+1)$-dimensional cone of $\A^n$ for $t+2\leq i\leq n$.
Hence there exists a unique cone $\sigma_1\in \sd_{\eta,t+1,d}(\Sigma)$ such that $\eta\subset \sigma_1\subset \sigma$,
and we have
\[
\sigma_1
=
\Cone(e_1,\ldots,e_t,f_{t+1,1},\ldots,f_{n,1}),
\]
with $f_{t+1,1}:=(d-1)(e_1+\cdots+e_t)+f_{t+1}$ and $f_{i,1}=f_i$ for $t+2\leq i\leq n$.
Note that we have $\varphi_1(f_{t+1,1})=dt$.
For $2\leq j\leq n-t$,
consider a cone
\[
\sigma_j=\Cone(e_1,\ldots,e_t,f_{t+1,j},\ldots,f_{n,j})
\in \sd_{\eta,t+1,d}^j(\A^n)
\]
such that $\eta\subset \sigma_j \subset \sigma_{j-1}$.
For $1\leq j\leq n-t$,
we claim:
\begin{enumerate}
\item[(i)]
$\lvert \{ t+1\leq i\leq n:\varphi_1(f_{i,j})\geq dt\}\rvert \geq j$,
\item[(ii)]
$\varphi_2(f_{i,j})\leq (n-t)^j$ for $t+1\leq i\leq n$.
\end{enumerate}
The claims are clear if $j=1$.
Assume that the claims hold for $j-1$.
We may assume that there exists a permutation $\alpha$ of $\{t+1,\ldots,n\}$ such that
\[
f_{t+1,j}
=
d(e_1+\cdots+e_t)+f_{\alpha(t+1),j-1}
\]
and
\[
f_{i,j}
=
e_1+\cdots+e_t+f_{\alpha(t+1),j-1}+\cdots+ f_{\alpha(i),j-1}
\]
for $t+2\leq i\leq n$.
We deduce the claim (ii) from these two.

Observe that we have $\varphi_1(f_{t+1,j})\geq dt$.
If $i\geq n-j+2$,
then there exists an integer $t+1\leq k\leq i$ such that $\varphi_1(f_{\alpha(k),j-1})\geq dt$ by the claim (i) for $j-1$.
Then we have $\varphi_1(f_{i,j})\geq dt$ too.
Since $\lvert \{t+1\}\cup \{n-j+2,\ldots,n\} \rvert = j$,
we deduce the claim (i).

Now, with $j=n-t$,
the claim (i) implies that $\varphi_1(f_{i,n-t})\geq dt$ for every $t+1\leq i\leq n$.
Hence we have
\[
(n-t)^{n-t} \varphi_1(f_i)
\geq 
dt \varphi_2(f_i)
\]
for every integer $t+1\leq i\leq n$.
The same inequality holds for $1\leq i\leq t$ too, so we have $\lvert \sigma_{n-t} \rvert
\subset
H_{(n-t)^{n-t}/dt}^{\geq 0}(\eta,\A^n)$.
To conclude,
observe that every cone $\tau\in \sd_{\eta,n+1,d}^{n-t}(\Sigma)$ such that $\eta\subset \tau$ is of the form $\sigma_{n-t}$.
\end{proof}

We need the following technical result, which will be used in Lemma \ref{dist.4} below.

\begin{lem}
\label{dist.9}
Consider the cone $\eta:=\Cone(e_1,\ldots,e_t)$ in $\A^n$,
where $1\leq t< n$ are integers.
Let $\tau$ be an $n$-dimensional cone contained in $\N^n$ such that $\tau\cap \eta$ is a face of $\eta$.
Consider the fan $\Sigma$ with a single maximal cone $\tau$.
Then there exists $C>0$ such that
\[
H_\epsilon^{\geq 0}(\eta\cap \tau,\Sigma)
\subset
H_{C\epsilon}^{\geq 0}(\eta,\A^n)
\]
for every $0<\epsilon<1$.
\end{lem}
\begin{proof}
Without loss of generality,
we may assume $\tau\cap \eta=\Cone(e_1,\ldots,e_s)$ with an integer $0\leq s\leq t$.
Write $\tau$ as
\[
\Cone(e_1,\ldots,e_s,f_{s+1},\ldots,f_n)
\]
such that $f_{s+1},\ldots,f_n\notin \Cone(e_1,\ldots,e_s)$.
For integers $s+1\leq i\leq n$ and $1\leq j\leq n$,
let $f_{i,j}$ be the $j$th coordinate of $f_i$.
We have
\[
H_\epsilon^{\geq 0}(\eta\cap \tau,\Sigma)
=
\{(b_1,\ldots,b_n)\in \R_{\geq 0}^n:\epsilon(b_1+\cdots+b_s)\geq b_{s+1}+\cdots+b_n\}
\]
with the coordinates $b_1,\ldots,b_n$ for $\lvert \Sigma\rvert$.
Let $(a_1,\ldots,a_n)$ be the coordinates for $\R_{\geq 0}^n=\lvert \A^n \rvert$.
Consider the $n\times n$ matrix such that its $j$th column agrees with $e_j$ for $1\leq j\leq s$ and $f_j$ for $s+1\leq j\leq n$,
and let $M=(m_{ij})$ be its inverse matrix.
Note that we have $m_{ij}=0$ if $s+1\leq i\leq n$ and $1\leq j\leq s$ or if $1\leq i,j\leq s$ and $i\neq j$.
An element $(a_1,\ldots,a_n)$ using the coordinates for $\lvert \A^n\rvert$ corresponds to $(b_1,\ldots,b_n):=M(a_1,\ldots,a_n)$ using the coordinates for $\lvert \Sigma \rvert$.
We have the formula
\[
b_i
=
\left\{
\begin{array}{ll}
a_i+\sum_{j=s+1}^n m_{ij}a_j & \text{if $1\leq i\leq s$},
\\
\sum_{j=s+1}^n m_{ij}a_j & \text{otherwise}.
\end{array}
\right.
\]
Plug these in the inequality
\[
\epsilon(b_1+\cdots+b_s)\geq b_{s+1}+\cdots+b_s
\]
and use the condition $0<\epsilon<1$ to deduce
\[
H_\epsilon^{\geq 0}(\eta\cap \tau,\Sigma)
\subset
\{(a_1,\ldots,a_n)\in \R_{\geq 0}^n:C\epsilon(a_1+\cdots+a_s)\geq a_{s+1}+\cdots+a_n\}
\]
with $1/C:=\sum_{i=1}^n \sum_{j=s+1}^n \lvert m_{ij}\rvert$.
Compare this with
\[
H_{C\epsilon}^{\geq 0}(\eta,\A^n)
\subset
\{(a_1,\ldots,a_n)\in \R_{\geq 0}^n:C\epsilon(a_1+\cdots+a_t)\geq a_{t+1}+\cdots+a_n\}
\]
to conclude.
\end{proof}

Here is another technical result that will be used in Lemma \ref{dist.5} below.

\begin{lem}
\label{dist.11}
Consider the cone $\eta:=\Cone(e_1,\ldots,e_t)$ in $\A^n$,
where $1\leq t< n$ are integers.
Let $\tau$ be a smooth cone contained in $\N^n$ such that $\tau\cap \eta$ is a face of $\eta$.
Then there exists $\epsilon>0$ satisfying the following condition:
Let $I:=\{i_1,\ldots,i_s\}$ be a subset of $\{1,\ldots,t\}$,
and let $\sigma=\Cone(e_{i_1},\ldots,e_{i_s},f_{s+1},\cdots,f_n)$ be an $I$-rigid cone contained in $H_{\epsilon}^{\geq 0}(\eta,\A^n)$.
Then $\sigma\cap \tau$ is generated by $\sigma\cap \tau\cap \eta$ and some cone contained in $\Cone(f_{s+1},\ldots,f_n)$.
\end{lem}
\begin{proof}
Without loss of generality,
we may assume $I=\{1,\ldots,s\}$.
Since $\sigma$ is $I$-rigid and $\sigma\subset H_\epsilon^{\geq 0}(\eta,\A^n)$,
we have
\[
\lvert \sigma \rvert
\subset
U_i^{\geq 0}
:=
\{(a_1,\ldots,a_n):\R_{\geq 0}^n:s\epsilon a_i\geq a_{s+1}+\cdots+a_n\}
\]
for every integer $1\leq i\leq s$.
Consider also
\[
U_i^{\leq 0}
:=
\{(a_1,\ldots,a_n):\R_{\geq 0}^n:s\epsilon a_i\leq a_{s+1}+\cdots+a_n\}
\]
and $U_i^0:=U_i^{\geq 0}\cap U_i^{\leq 0}$.
Without loss of generality,
we may assume
\[
\tau = \Cone(e_1,\ldots,e_u,g_{u+1},\ldots,g_n)
\]
for some integer $0\leq u\leq s$ and $g_{u+1},\ldots,g_n\in \N^n-\sigma\cap \eta$.
There exists $\epsilon>0$ such that
\begin{equation}
\label{dist.11.1}
g_{u+1},\ldots,g_n \in U_i^{\leq 0}-U_i^0
\end{equation}
for every $1\leq i\leq s$.

Assume $u<s$.
Then we have $e_1,\ldots,e_u\in U_i^0$ and $\lvert \sigma \rvert \subset U_i^{\geq 0}$ for every $1\leq i\leq s$.
Together with \eqref{dist.11.1},
we have $\sigma\cap \tau=\Cone(e_1,\ldots,e_u)$,
which is contained in $\eta$.

Assume $u=s$.
Consider an element $(a_1,\ldots,a_n)\in \lvert \sigma \cap \tau \rvert$ and its projections
\[
(b_1,\ldots,b_s,a_{s+1},\ldots,a_n),
(c_1,\ldots,c_s,a_{s+1},\ldots,a_n)
\]
to the faces $\lvert\Cone(f_{s+1},\ldots,f_n)\rvert$ and $\lvert \Cone(g_{s+1},\ldots,g_n)\rvert$ when we regard $(a_1,\ldots,a_n)$ as an element of $\lvert \sigma \rvert$ and $\lvert \tau \rvert$ respectively.
We have
\[
c_i\leq (a_{s+1}+\cdots+a_n)/(s\epsilon) \leq b_i
\]
for every $1\leq i\leq s$ since $\lvert \sigma \rvert \subset U_i^{\geq 0}$ and $f_i\in U_i^{\leq 0}$.
It follows that we have
\[
(b_1,\ldots,b_s,a_{s+1},\ldots,a_n)\in \lvert \sigma \cap \tau\rvert,
\]
Now,
let $\sigma'$ be the projection of $\sigma\cap \tau$ onto $\Cone(f_{s+1},\ldots,f_n)$.
Then we have $\sigma'\subset \sigma\cap \tau$ and hence $\sigma\cap \tau$ is generated by $\sigma\cap \tau\cap \eta=\Cone(e_1,\ldots,e_s)$ and $\sigma'$.
\end{proof}

Now, we generalize the definition of $H_\epsilon^{\geq 0}(\eta,\A^n)$ to general $\Sigma$.

\begin{df}
Let $\Sigma$ be a smooth fan,
and let $\eta$ be its cone.
For $\epsilon>0$,
we set
\[
H_\epsilon^{\geq 0}(\eta,\Sigma)
:=
\bigcup_{\sigma\in \Sigma_{\max}}H_\epsilon^{\geq 0}(\eta,\langle \sigma \rangle).
\]
Here,
$\langle \sigma\rangle$ denotes the fan in the lattice of $\Sigma$ with the single maximal cone $\sigma$.
\end{df}

In algebraic topology, an iterated barycentric subdivision makes the diameter of an $n$-simplex arbitrarily small.
The following is analogous to this result.

\begin{lem}
\label{dist.4}
Consider $\Sigma:=(\P^1-0)^r\times (\P^1)^{n-r}$ and $\eta:=\Cone(e_{r+1},\ldots,e_n)\in \Sigma$ with integers $0\leq r\leq n$.
Then for every $\epsilon>0$, there exist integers $d_2,\ldots,d_{n-r+1}\geq 0$ such that for every maximal cone
\[
\sigma
\in 
\sd_{\eta,2,d_2}^{n-1}
\sd_{\eta,3,d_3}^{n-2}
\cdots
\sd_{\eta,n-r+1,d_{n-r+1}}^{r}
(\Sigma)
\]
satisfying $\sigma\cap \eta \neq 0$,
we have $\lvert \sigma \rvert \subset H_\epsilon^{\geq 0}(\eta,\Sigma)$.
\end{lem}
\begin{proof}
We claim that for $2\leq j\leq n-r+2$,
there exist integers $d_{j},d_{j+1},\ldots,d_{n-r+1}$ such that for every maximal cone $\sigma$ of
\[
\Sigma_j
:=
\sd_{\eta,j,d_j}^{n-j+1}
\sd_{\eta,j+1,d_{j+1}}^{n-j}
\cdots
\sd_{\eta,n-r+1,d_{n-r+1}}^{r}
(\Sigma)
\]
satisfying $\dim( \sigma\cap \eta) \geq j-1$,
we have $\lvert \sigma \rvert \subset H_{\epsilon}^{\geq 0}(\eta,\Sigma)$.
The claim holds if $j=n-r+2$ since $\dim \eta=n-r$.

Assume that the claim holds for $j+1$.
Let $\sigma$ be a maximal cone of $\Sigma_j$ satisfying $\dim( \sigma\cap \eta)\geq j-1$,
and let $\tau$ be the unique maximal cone of $\Sigma_{j+1}$ containing $\sigma$.
If $\dim (\tau\cap \eta)\geq j$,
then we have  $\lvert \tau \rvert \subset H_{\epsilon}^{\geq 0}(\eta,\Sigma)$,
which implies  $\lvert \sigma \rvert \subset H_{\epsilon}^{\geq 0}(\eta,\Sigma)$.
Hence assume $\dim (\tau\cap \eta) \leq j-1$,
which implies $\dim (\tau\cap \eta)=j-1$ since $\dim (\sigma \cap \eta)\geq j-1$.

Apply Lemmas \ref{dist.2} and \ref{dist.9} to $\tau$ to obtain an integer $d_\tau\geq 0$ satisfying the following condition:
If $d_j\geq d_{\tau}$ and $\Sigma_j$ is formed by this $d_j$,
then for every maximal cone $\sigma'$ of $\Sigma_j$ contained in $\tau$ such that $\sigma'\cap \eta\neq 0$,
we have  $\lvert \sigma' \rvert \subset H_{\epsilon}^{\geq 0}(\eta,\Sigma)$.
Let $d_j$ be the maximum of $d_\tau$ for all $\tau\in \Sigma_{j+1}$ such that $\dim (\tau\cap \eta)=j-1$.
With this $d_j$, $\Sigma_j$ defined above satisfies the inductive hypothesis.
To conclude,
set $j=2$.
\end{proof}

The following is a consequence of Lemma \ref{dist.4}.

\begin{lem}
\label{dist.5}
Let $\Delta$ be a smooth subdivision of $\Sigma:=(\P^1-0)^r\times (\P^1)^{n-r}$ with integers $0\leq r\leq n$.
Assume $\eta:=\Cone(e_{r+1},\ldots,e_n)\in \Delta$.
Then there exist integers $d_2,\ldots,d_{n-r+1}\geq 0$ satisfying the following condition:
Let $\delta\in \Delta$ and
\[
\sigma\in \sd_{\eta,2,d_2}^{n-1}
\sd_{\eta,3,d_3}^{n-2}
\cdots
\sd_{\eta,n-r+1,d_{n-r+1}}^{r}
(\Sigma)
\]
be cones contained in the same maximal cone of $\Sigma$.
Then $\sigma\cap \delta$ is generated by $\sigma\cap \delta\cap \eta$ and a cone contained in a face $\sigma'$ of $\sigma$ such that $\sigma'\cap \eta=0$.
\end{lem}
\begin{proof}
For every $\epsilon>0$,
by Lemma \ref{dist.4},
there exist integers $d_2,\ldots,d_{n-r+1}\geq 0$ such that for every cone $\sigma\in \sd_{\eta,2,d_2}^{n-1}
\sd_{\eta,3,d_3}^{n-2}
\cdots
\sd_{\eta,n-r+1,d_{n-r+1}}^{r}
(\Sigma)$ satisfying $\sigma\cap \eta\neq 0$,
we have $\sigma \subset H_\epsilon^{\geq 0}(\eta,\Sigma)$.
Apply Lemma \ref{dist.11} to all the maximal cones of $\Delta$ and use Proposition \ref{dist.8} to conclude.
\end{proof}

Now, we are ready to prove the main result of this subsection.

\begin{prop}
\label{dist.10}
Let $\Delta$ be a smooth subdivision of $\Sigma:=(\P^1-0)^r\times (\P^1)^{n-r}$.
Assume $\eta:=\Cone(e_{r+1},\ldots,e_n)\in \Delta$.
Then there exists a sequence of star subdivisions
\[
\Delta_m \to \cdots \to \Delta_0:=
\sd_{\eta,2,d_2}^{n-1}
\sd_{\eta,3,d_3}^{n-2}
\cdots
\sd_{\eta,n-r+1,d_{n-r+1}}^r
(\Sigma)
\]
for some integers $d_2,\ldots,d_{n-r+1}\geq 0$ such that $\Delta_m$ is a subdivision of $\Delta$ and for each $i$,
$\Delta_{i+1}=\Delta_i^*(\delta_i)$ for some $2$-dimensional cone $\delta_i$ of $\Delta_i$ such that $e_{r+1}\ldots,e_n\notin \delta_i$.
\end{prop}
\begin{proof}
Choose integers $d_2,\ldots,d_{n-r+1}\geq 0$ as in Lemma \ref{dist.5}.
Consider the fans
\begin{gather*}
\Delta_0^\flat:=\{\delta\in \Delta_0:e_{r+1},\ldots,e_n\notin \delta\},
\text{ }
\Gamma
:=
\{\delta\cap \delta' : \delta\in \Delta,\delta'\in \Delta_0\},
\\
\Gamma^\flat:=\{\gamma\in \Gamma:e_{r+1},\ldots,e_n\notin \gamma\}.
\end{gather*}
By \cite[Theorem 2.4]{MR803344} (see also \cite[pp.\ 39-40]{TOda}),
there exists a sequence of star subdivisions relative to $2$-dimensional cones
\[
\Delta_m^\flat \to \cdots \to \Delta_0^\flat
\]
such that $\Delta_m^\flat$ is a subdivision of $\Gamma^\flat$.
Consider the corresponding sequence of star subdivisions relative to the same cones
\[
\Delta_m\to \cdots \to \Delta_0.
\]
Every cone of $\Delta_m$ (resp.\ $\Gamma$) is generated by a face of $\eta$ and a cone of $\Delta_m^\flat$ (resp.\ $\Gamma^\flat$) by construction (resp.\ the conclusion of Lemma \ref{dist.5}).
It follows that $\Delta_m$ is a subdivision of $\Gamma$ and hence a subdivision of $\Delta$ too.
\end{proof}

\begin{rmk}
\label{dist.14}
Let us explain why the usual barycentric subdivision is not enough for Proposition \ref{dist.10} even if $n-r=1$.
Consider $\Sigma_1:=\A^3$.
Let $\Sigma_{d+1}$ be the barycentric subdivision of $\Sigma_d$ relative to itself for every integer $d\geq 1$.
Consider a cone $\Cone(e_1,f_2,f_3)$ of $\Sigma_d$.
Then we have the cone $\sigma:=\Cone(e_1,e_1+f_2,e_1+f_2+f_3)$ of $\Sigma_{d+1}$ and then the cone $\alpha(\sigma):=\Cone(e_1,2e_1+f_2+f_3,3e_1+2f_2+f_3)$ of $\Sigma_{d+2}$.

Now, consider the cone $\tau:=\Cone(e_1,e_2,e_3)$ of $\A^3$.
Then for every integer $m\geq 1$,
we have
$\alpha^m(\tau)=\Cone(e_1,(x_m,y_m,z_m),(x_m',y_m',z_m'))$ with
\[
\begin{pmatrix}
x_m\\
x_m'
\end{pmatrix}
=
(M^{m-1}+\cdots+M^0)
\begin{pmatrix}
2\\
3
\end{pmatrix},
\text{ }
\begin{pmatrix}
y_m\\
y_m'
\end{pmatrix}
=
M^m
\begin{pmatrix}
1\\
0
\end{pmatrix},
\text{ }
\begin{pmatrix}
z_m\\
z_m'
\end{pmatrix}
=
M^m
\begin{pmatrix}
0\\
1
\end{pmatrix},
\]
where
\[
M:=
\begin{pmatrix}
1 & 1\\
2 & 1
\end{pmatrix}.
\]
Diagonalize $M$ to show that the ratios $[x_m:y_m:z_m]$ and $[x_m':y_m':z_m']$ do not converge to $[1:0:0]$.
Hence if $\epsilon>0$ is sufficiently small,
then we have $\alpha^m(\tau)\not \subset H_{\epsilon}^{\geq 0}(\eta,\A^3)$ for all $m\geq 0$,
where $\eta:=\Cone(e_1)$.

On the other hand, for every integer $d\geq 1$,
the rays of $\sd_{\eta,2,d}(\A^3)$ adjacent to $e_1$ are precisely 
\[
de_1+e_2,e_1+e_2+e_3,de_1+e_3.
\]
Hence the rays of $\sd_{\eta,2,d}^2(\A^3)$ adjacent to $e_1$ are precisely
\[
2de_1+e_2,(d+1)e_1+2e_2+e_3,(d+1)e_1+e_2+e_3,(d+1)e_1+e_2+2e_3,2de_1+e_3.
\]
It follows that if $(d+1)\epsilon \geq 3$ is satisfied, then every maximal cone of $\sd_{\eta,2,d}^2(\A^3)$ is contained in $H_{\epsilon}^{\geq 0}(\eta,\A^3)$. We illustrate a part of $\sd_{\eta,2,2}^2(\A^3)$ as follows to help the reader,
where the five rays adjacent to $e_1$ are marked:
\[
\begin{tikzpicture}[scale = 1]
\draw (0,4)--(0,0)--(2,2)--(0,4);
\draw (0,4)--(-2,2)--(0,0);
\draw (0,4) node[above] {$e_1$};
\draw (2,2) node[right] {$2e_1+e_3$};
\draw (-2,2) node[left] {$2e_1+e_2$};
\draw (0,0) node[below] {$e_1+e_2+e_3$};
\draw (1/2,7/2) node[above right] {$4e_1+e_3$};
\draw (-1/2,7/2) node[above left] {$4e_1+e_2$};
\filldraw[black] (1/2,7/2) circle (2pt);
\filldraw[black] (-1/2,7/2) circle (2pt);
\filldraw[black] (0,3) circle (2pt);
\filldraw[black] (2/3,2) circle (2pt);
\filldraw[black] (-2/3,2) circle (2pt);
\draw (2/3,2)--(0,4);
\draw (2/3,2)--(1/2,7/2);
\draw (2/3,2)--(1,3);
\draw (2/3,2)--(3/2,5/2);
\draw (2/3,2)--(2,2);
\draw (2/3,2)--(3/2,3/2);
\draw (2/3,2)--(1,1);
\draw (2/3,2)--(1/2,1/2);
\draw (2/3,2)--(0,0);
\draw (2/3,2)--(0,1);
\draw (2/3,2)--(0,2);
\draw (2/3,2)--(0,3);
\draw (2/3,2)--(0,4);
\draw (-2/3,2)--(0,4);
\draw (-2/3,2)--(-1/2,7/2);
\draw (-2/3,2)--(-1,3);
\draw (-2/3,2)--(-3/2,5/2);
\draw (-2/3,2)--(-2,2);
\draw (-2/3,2)--(-3/2,3/2);
\draw (-2/3,2)--(-1,1);
\draw (-2/3,2)--(-1/2,1/2);
\draw (-2/3,2)--(0,0);
\draw (-2/3,2)--(0,1);
\draw (-2/3,2)--(0,2);
\draw (-2/3,2)--(0,3);
\end{tikzpicture}
\]
\end{rmk}

\section{Standard subdivisions}
\label{subdivision}

In this section,
we introduce the notion of very standard subdivisions,
and we construct specific very $r$-standard subdivisions $\Gamma_{n,r}$ and $\Theta_{n,r,\bd}$ for integers $0\leq r\leq n$ and a finite sequence of positive integers $\bd$.
These will play technical roles in the later sections.

We refer to Definition \ref{intro.1} for $r$-standard subdivisions of $(\P^1)^n$,
where $0\leq r\leq n$ are integers.
Combining the conditions (i) and (ii) in Definition \ref{intro.1},
we have
\[
\Cone(e_1,\ldots,e_n)\in \Sigma
\]
for every $r$-standard subdivision $\Sigma$ of $(\P^1)^n$.
This condition, which is weaker than the original conditions (i) and (ii),
has the following geometric meaning:
The induced morphism of $\C$-realizations
\[
f\colon \Sigma_{\C}\to (\P_{\C}^1)^n
\]
satisfies
\[
f^{-1}(\A_{\C}^n)
\simeq
\A_{\C}^n.
\]
In the language of \cite{logSHF1}, this is precisely the condition that there exists an admissible blow-up $g\colon X\to \square_{\C}^n$ such that the underlying morphism of schemes $\ul{g}$ is identified to $f$ (this information is only for the readers of \cite{logSHF1} and not strictly needed in this paper).

\subsection{Definition and basic properties of very standard subdivisions}

The notion of very standard subdivisions will play a crucial role in Construction \ref{subdivision.27}.
In Lemma \ref{moving.5} and Proposition \ref{ordering.18},
we will show that the class of very standard subdivisions is closed under two different types of subdivisions.

\begin{df}
\label{ordering.4}
Let $0\leq r\leq n$ be integers.
A \emph{very $r$-standard subdivision of $(\P^1)$} is an $r$-standard subdivision $\Sigma$ of $(\P^1)^n$ satisfying the following condition too:
Let $\sigma$ be a cone of $\Sigma$ such that $e_1,\ldots,e_n\notin \sigma$.
If $I:=\{i_1,\ldots,i_s\}$ is a subset of $\{r+1,\ldots,n\}$ such that $\Cone( \sigma,e_{i_1}),\ldots,\Cone( \sigma,e_{i_s}) \in \Sigma$,
then $\Cone( \sigma,e_{i_1},\ldots,e_{i_s}) \in \Sigma$.
\end{df}

Let us first record two basic results on standard subdivisions.

\begin{prop}
\label{ordering.20}
Let $\Sigma$ be an $r$-standard subdivision of $(\P^1)^n$ with integers $0\leq r\leq n$,
and let $\Sigma'$ be the star subdivision of $\Sigma$ relative to a cone $\sigma$ such that $\{e_1,\ldots,e_r\}\notin \sigma$ and $\sigma\not\subset \Cone(e_{r+1},\ldots,e_n)$.
Then $\Sigma'$ is an $r$-standard subdivision of $(\P^1)^n$.
\end{prop}
\begin{proof}
The condition $\{e_1,\ldots,e_r\}\notin \sigma$ (resp.\ $\sigma\not\subset \Cone(e_{r+1},\ldots,e_n)$) is equivalent to the condition that $\Sigma'$ satisfies the condition (i) (resp.\ (ii)) in Definition \ref{intro.1}.
\end{proof}

\begin{prop}
\label{ordering.26}
Let $0\leq r\leq n$ be integers.
Then there is a one-to-one correspondence between the set of $r$-standard subdivisions of $(\P^1)^n$ and the set of smooth subdivisions of $(\P^1-0)^r\times (\P^1)^{n-r}$ containing $\Cone(e_{r+1},\ldots,e_n)$ as a cone.
\end{prop}
\begin{proof}
For an $r$-standard subdivision $\Sigma$ of $(\P^1)^n$,
its restriction $\res(\Sigma)$ to the support of $(\P^1-0)^r\times (\P^1)^{n-r}$ still contains $\Cone(e_{r+1},\ldots,e_n)$ as a cone.
If $\Sigma$ and $\Sigma'$ are $r$-standard subdivisions of $(\P^1)^n$ such that $\res(\Sigma)=\res(\Sigma')$,
then the condition (ii) in Definition \ref{intro.1} implies $\Sigma=\Sigma'$.

When a smooth subdivision $\Delta$ of $(\P^1-0)^{r}\times (\P^1)^{n-r}$ containing $\Cone(e_{r+1},\ldots,e_n)$ as a cone is given,
let $\Sigma$ be the set of cones of the form
$\Cone(\delta,e_{i_1},\ldots,e_{i_s})$
such that $\delta\in \Delta$, $I:=\{i_1,\ldots,i_s\}$ is a subset of $\{1,\ldots,r\}$,
and $\delta$ is contained in the lattice $\Z^{\{1,\ldots,r\}-I}\times \Z^{n-r}$.
To conclude that,
observe that $\Sigma$ is an $r$-standard subdivision of $(\P^1)^n$,
and we have $\res(\Sigma)=\Delta$.
\end{proof}

The following combinatorial result will be needed in Lemma \ref{moving.5}.

\begin{lem}
\label{moving.3}
Consider the fan $\A^n$ and its cones $\eta:=\Cone(e_1,\ldots,e_t)$ and $\tau:=\Cone(e_1,\ldots,e_m)$,
where $0\leq t\leq m\leq n$ are integers.
Let $\Sigma$ be the $\eta$-excluded barycentric subdivision of $\A^n$ relative to $\tau$,
and let $\sigma$ be a cone of $\Sigma$ such that $e_1,\ldots,e_t\notin \sigma$.
If $I:=\{i_1,\ldots,i_s\}$ is a subset of $\{1,\ldots,t\}$ such that
\[
\Cone( \sigma,e_{i_1}),\ldots,\Cone( \sigma,e_{i_s}) \in \Sigma,
\]
then we have $\Cone(\sigma,e_{i_1},\ldots,e_{i_s}) \in \Sigma$.
\end{lem}
\begin{proof}
We proceed by induction on $m$.
The claim is obvious for $m=0$.
Assume that the claim holds for $m-1$,
Consider $f:=e_1+\cdots+e_m$.
If $f\notin \sigma$,
then $\sigma$ is contained in
\[
\tau_i:=\Cone(e_1,\ldots,e_{i-1},e_{i+1},\ldots,e_n)
\]
for some integer $1\leq i\leq m$.
Hence by induction,
we have $\Cone(\sigma,e_{i_1},\ldots,e_{i_s}) \in \Sigma$.

If $f\in \sigma$,
then consider the facet $\sigma'$ of $\sigma$ such that $\sigma'$ and $f$ generate $\sigma$.
By induction,
we have $\Cone( \sigma',e_{i_1},\ldots,e_{i_s})\in\Sigma$.
Proposition \ref{subdivision.32} implies that there exists a $t$-admissible permutation $\alpha$ of $\{1,\ldots,m\}$ such that $\Cone(\sigma',e_{i_1},\ldots,e_{i_s})$ is a face of
\[
\sigma_\alpha=
\Cone(e_{\alpha(1)},\ldots,e_{\alpha(c)},e_{\alpha(1)}+\cdots+e_{\alpha(c+1)},\ldots,e_{\alpha(1)}+\cdots+e_{\alpha(m)},e_{m+1},\ldots,e_n),
\]
where $c$ is the smallest integer such that $\alpha(c+1)>t$.
Since $f=e_{\alpha(1)}+\cdots+e_{\alpha(m)}$,
we see that $\Cone(\sigma,e_{i_1},\ldots,e_{i_s})$ is a face of $\sigma_\alpha$ and hence a cone of $\Sigma$.
\end{proof}

Next, we show that certain excluded barycentric subdivisions remain very $r$-standard.

\begin{lem}
\label{moving.5}
Let $\Sigma$ be an $r$-standard subdivision of $(\P^1)^n$ with integers $0\leq r\leq n$,
and let $A$ be a subfan of $\Sigma$ satisfying the condition in \textup{Definition \ref{dist.12}}.
Assume that all cones of $A$ are contained in the support of $(\P^1-0)^r\times (\P^1)^n$ and not contained in $\eta$.
Consider the $\eta$-excluded barycentric subdivision $\Delta :=\Sigma_\eta^\bary(A)$ of $\Sigma$ relative to $A$,
where $\eta:=\Cone(e_{r+1},\ldots,e_n)$.
If $\Sigma$ is very $r$-standard,
then $\Delta$ is very $r$-standard too.
\end{lem}
\begin{proof}
By Proposition \ref{ordering.20},
$\Delta$ is an $r$-standard subdivision of $(\P^1)^n$.
It remains to check the condition in Definition \ref{ordering.4} for $\Delta$.

Let $\tau'$ be a cone of $\Delta$ such that $e_1,\ldots,e_n\notin \tau'$,
and let $\{i_1,\ldots,i_t\}$ be a subset of $\{1,\ldots,n\}$ such that $\Cone(\tau',e_{i_1}),\ldots,\Cone(\tau',e_{i_t})\in \Delta$.
Consider the cone $\tau$ of $\Sigma$ generated by $\tau'$.
If $\Cone(\tau,e_{i_j})\notin \Sigma$ for some integer $1\leq j\leq t$,
then let $\eta$ be a cone of $\Sigma$ such that $\dim \eta=\dim \Cone(\tau,e_{i_j})$ and
\[
\tau \subset \eta\subset \Cone(\tau,e_{i_j}).
\]
Then $\eta=\Cone(\tau,f)$ for some ray $f$ of $\Sigma$,
and we have
\[
\tau\subsetneq \Cone(\tau,f)\subsetneq \Cone(\tau,e_{i_j}).
\]
This implies that we have
\[
\tau'\subsetneq \Cone(\tau',f)\subsetneq \Cone(\tau',e_{i_j}).
\]
Since $\dim \Cone(\tau',f)=\dim \Cone(\tau',e_{i_j})$,
we have $\Cone(\tau',e_{i_j})\notin \Delta$,
which is a contradiction.
It follows that we have $\Cone(\tau,e_{i_j})\in \Sigma$ for all $1\leq j\leq t$.

Since $\Sigma$ is very $r$-standard,
we have $\sigma':=\Cone(\tau,e_{i_1},\ldots,e_{i_t}) \in \Sigma$.
The collection of all cones of $\Delta$ contained in $\sigma'$ forms a fan,
and we can apply Lemma \ref{moving.3} to this fan to conclude.
\end{proof}

We also record the following result here, which will be used in Lemma \ref{subdivision.24}.

\begin{prop}
\label{ordering.18}
Let $\Sigma$ be a very $r$-standard subdivision of $(\P^1)^n$ with integers $0\leq r\leq n$,
and let $\Sigma'$ be the star subdivision of $\Sigma$ relative to a $2$-dimensional cone $\sigma$ such that $e_i\notin \sigma$ for all $1\leq i\leq r$ and $\sigma\not\subset \Cone(e_{r+1},\ldots,e_n)$.
Then $\Sigma'$ is a very $r$-standard subdivision of $(\P^1)^n$.
\end{prop}
\begin{proof}
Since $\Sigma$ is an $r$-standard subdivision of $(\P^1)^n$ by Proposition \ref{ordering.20},
it suffices to check the condition in Definition \ref{ordering.4} for $\Sigma'$.

Let $\tau'$ be a cone of $\Sigma'$ such that $e_1,\ldots,e_n\notin \tau'$,
and let $I:=\{i_1,\ldots,i_s\}$ be a subset of $\{1,\ldots,n\}$ such that $\Cone( \tau',e_{i_1}),\ldots,\Cone( \tau',e_{i_s}) \in \Sigma$.

Consider the cone $\tau$ of $\Sigma$ generated by $\tau'$.
Let us argue as in the proof of Lemma \ref{moving.5}.
If $\Cone(\tau,e_{i_j})\notin \Sigma$ for some $1\leq j\leq s$,
then there exists a ray $f$ of $\Sigma$ such that
\[
\tau\subsetneq \Cone(\tau,f)\subsetneq \Cone(\tau,e_{i_j}).
\]
This implies that we have
\[
\tau'\subsetneq \Cone(\tau',f)\subsetneq \Cone(\tau',e_{i_j})
\]
and hence $\Cone(\tau',e_{i_j})\notin \Sigma'$,
which is a contradiction.
Hence we have $\Cone(\tau,e_{i_j})\in \Sigma$ for every integer $1\leq j\leq s$.

If $e_j\notin \tau$ for every integer $r+1\leq j\leq n$,
then we have $\Cone(\tau,e_{i_1},\ldots,e_{i_s})\in \Sigma$ since $\Sigma$ is a very $r$-standard subdivision.
Using the description of the star subdivision,
we have $\Cone(\tau',e_{i_1},\ldots,e_{i_t})\in \Sigma'$ too.

If $e_j\in \tau$ for some integer $r+1\leq j\leq n$,
then we have $e_j\in \sigma$.
Consider the face $\eta$ of $\tau$ such that $\Cone(e_j,\eta)=\tau$.
Then there exists a ray $f\in \eta$ such that $\sigma=\Cone(e_j,f)$,
and we have $\Cone(e_j+f,\eta)=\tau'$.
It follows that we have $\Cone(\eta,e_j)\in \Sigma$.
Together with $\Cone(\eta,e_{i_1}),\ldots,\Cone(\eta,e_{i_s})\in \Sigma$ and the assumption that $\Sigma$ is a very $r$-standard subdivision,
we have $\Cone(\eta,e_j,e_{i_1},\ldots,e_{i_s})\in \Sigma$, i.e., $\Cone(\tau,e_{i_1},\ldots,e_{i_s})\in \Sigma$.
Using the description of the star subdivision,
we have $\Cone(\tau',e_{i_1},\ldots,e_{i_s})\in \Sigma'$ too.
\end{proof}

\subsection{Construction of the fans \texorpdfstring{$\Gamma_{n,r}$}{Gamma} and \texorpdfstring{$\Theta_{n,r,\bd}$}{Theta}}

The fan $\Theta_{n,r,\bd}$ constructed in this section will serve as the fan appearing in the base step of the proof of Theorem \ref{intro.2},
see the proof of Lemma \ref{subdivision.35}.
We will show in Proposition \ref{ordering.13} that $\Theta_{n,r,\bd}$ is a very $r$-standard subdivision.

\begin{df}
For integers $0\leq r\leq n$,
let $\Gamma_{n,r}$ be the $\Cone(e_{r+1},\ldots,e_n)$-excluded barycentric subdivision of $(\P^1)^n$ relative to $(\P^1-0)^r\times (\P^1)^n$.
Observe that $\Gamma_{n,r}$ is an $r$-standard subdivision of $(\P^1)^n$.
\end{df}

\begin{exm}
We illustrate the fan $(\P^1)^3$ as follows:
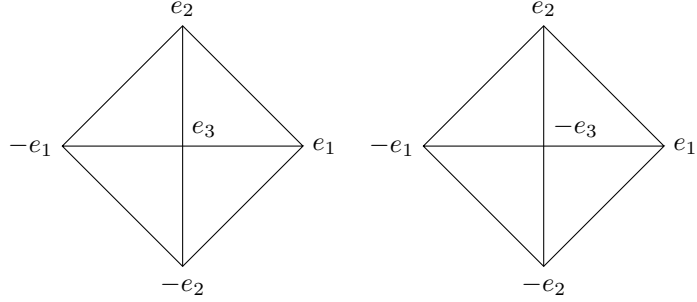
\[
\begin{tikzpicture}[scale = 0.8]
\draw (2,0)--(0,2)--(-2,0)--(0,-2)--(2,0)--(0,0)--(-2,0);
\draw (0,2)--(0,0)--(0,-2);
\draw (2,0) node[right] {$e_1$};
\draw (0,2) node[above] {$e_2$};
\draw (0,0) node[above right] {$e_3$};
\draw (-2,0) node[left] {$-e_1$};
\draw (0,-2) node[below] {$-e_2$};
\begin{scope}[shift={(6,0)}]
\draw (2,0)--(0,2)--(-2,0)--(0,-2)--(2,0)--(0,0)--(-2,0);
\draw (0,2)--(0,0)--(0,-2);
\draw (2,0) node[right] {$e_1$};
\draw (0,2) node[above] {$e_2$};
\draw (0,0) node[above right] {$-e_3$};
\draw (-2,0) node[left] {$-e_1$};
\draw (0,-2) node[below] {$-e_2$};
\end{scope}
\end{tikzpicture}
\]
Each cone with vertices $v_1$, $v_2$, and $v_3$ in the figure corresponds to the maximal cone $\Cone(v_1,v_2,v_3)$ of $(\P^1)^3$.

We illustrate the subdivision $\Gamma_{3,1}$ of $(\P^1)^3$ as follows:
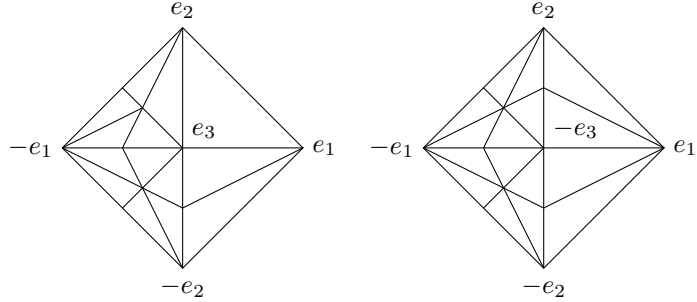
\[
\begin{tikzpicture}[scale = 0.8]
\draw (2,0)--(0,2)--(-2,0)--(0,-2)--(2,0)--(0,0)--(-2,0);
\draw (0,2)--(0,0)--(0,-2);
\draw (0,2)--(-1,0);
\draw (0,0)--(-1,1);
\draw (-2,0)--(-2/3,2/3);
\draw (0,-2)--(-1,0);
\draw (0,0)--(-1,-1);
\draw (-2,0)--(0,-1);
\draw (2,0)--(0,-1);
\draw (2,0) node[right] {$e_1$};
\draw (0,2) node[above] {$e_2$};
\draw (0,0) node[above right] {$e_3$};
\draw (-2,0) node[left] {$-e_1$};
\draw (0,-2) node[below] {$-e_2$};
\begin{scope}[shift={(6,0)}]
\draw (2,0)--(0,2)--(-2,0)--(0,-2)--(2,0)--(0,0)--(-2,0);
\draw (0,2)--(0,0)--(0,-2);
\draw (2,0)--(0,1);
\draw (2,0)--(0,-1);
\draw (0,0)--(-1,-1);
\draw (0,0)--(-1,1);
\draw (-2,0)--(0,1);
\draw (0,2)--(-1,0);
\draw (-2,0)--(0,-1);
\draw (0,-2)--(-1,0);
\draw (2,0) node[right] {$e_1$};
\draw (0,2) node[above] {$e_2$};
\draw (0,0) node[above right] {$-e_3$};
\draw (-2,0) node[left] {$-e_1$};
\draw (0,-2) node[below] {$-e_2$};
\end{scope}
\end{tikzpicture}
\]
\end{exm}

\begin{df}
\label{subdivision.36}
Let $0\leq r\leq n$, $s\geq 2$, and $d_2,\ldots,d_s\geq 1$ be integers.
Consider the cone $\eta:=\Cone(e_{r+1},\ldots,e_n)$ in the fan $(\P^1-0)^r\times (\P^1)^{n-r}$.
Let $\Theta_{n,r,d_2,\ldots,d_s}$ be the $r$-standard subdivision of $(\P^1)^n$ corresponding to the subdivision
\[
\sd_{\eta,\min(2,n-r+1),d_2}^{r+s-1}
\sd_{\eta,\min(3,n-r+1),d_3}^{r+s-2}
\cdots
\sd_{\eta,\min(s,n-r+1),d_s}^r((\P^1-0)^r\times (\P^1)^{n-r})
\]
of $(\P^1-0)^r\times (\P^1)^{n-r}$ obtained by Proposition \ref{ordering.26}.
If $\bd:=(d_2,\ldots,d_s)$,
then we set
\[
\Theta_{n,r,\bd}:=\Theta_{n,r,d_2,\ldots,d_s}.
\]
If $\bd$ is an empty sequence,
then we set
\[
\Theta_{n,r,\bd}:=(\P^1)^n.
\]
\end{df}

\begin{rmk}
With the above notation,
observe that $\Theta_{n,r,\bd}$ is a subdivision of $\Gamma_{n,r}$ if $\bd$ is nonempty.
Furthermore,
$\Theta_{n,r,\bd}$ is obtained from $(\P^1)^n$ by a sequence of $\eta$-excluded barycentric subdivisions,
which we will use several times later for induction arguments.
\end{rmk}

\begin{lem}
\label{moving.4}
For integers $0\leq r\leq n$,
$(\P^1)^n$ is a very $r$-standard subdivision of $(\P^1)^n$.
\end{lem}
\begin{proof}
Observe that $(\P^1)^n$ is an $r$-standard subdivision of $(\P^1)^n$.
If $\sigma$ is a cone of $\Sigma$ such that $e_1,\ldots,e_n\notin \sigma$,
then $\sigma=\Cone(-e_{j_1},\ldots,-e_{j_t})$ for some subset $J:=\{j_1,\ldots,j_t\}\subset \{1,\ldots,n\}$.
Let $I:=\{i_1,\ldots,i_s\}$ be the complement of $J$.
Then for an integer $1\leq i\leq n$, $\Cone(\sigma,e_i)\in (\P^1)^n$ if and only $i\in I$.
Furthermore, we have $\Cone(\sigma,e_{i_1},\ldots,e_{i_s})\in (\P^1)^n$.
It follows that $(\P^1)^n$ is a very $r$-standard subdivision.
\end{proof}

\begin{prop}
\label{ordering.13}
For integers $0\leq r\leq n$ and a finite sequence $\bd$ of positive integers,
$\Theta_{n,r,\bd}$ is a very $r$-standard subdivision of $(\P^1)^n$.
\end{prop}
\begin{proof}
This follows from Lemmas \ref{moving.4} and \ref{moving.5}.
\end{proof}

\begin{rmk}
For integers $0\leq r\leq n$ and a finite sequence $\bd$ of positive integers,
we expect that $\Theta_{n,r,\bd}$ can be obtained by a sequence of star subdivisions relative to suitable $2$-dimensional cones.
If this holds,
then Proposition \ref{ordering.18} implies Proposition \ref{ordering.13}.
\end{rmk}

\begin{exm}
The star subdivision $\Sigma$ of $(\P^1)^3$ relative to $\sigma:=\Cone(e_1,e_2,-e_3)$ is \emph{not} a very $0$-standard subdivision.
Indeed, we have $\Cone(-e_3,e_1),\Cone(-e_3,e_2)\in \Sigma$ but $\Cone(-e_3,e_1,e_2)\notin \Sigma$.
Hence when we deal with very $r$-standard subdivisions, we mainly work with star subdivisions relative to $2$-dimensional cones instead of $m$-dimensional cones with $m>2$.
\end{exm}

\section{Ordering maximal cones}
\label{ordering}

For a smooth complete fan $\Sigma$,
Fulton \cite[Theorem in p.\ 102]{Fulton:1436535} showed that $\CH_*(\Sigma_{\C})$ is a free abelian group and provided a basis under the assumption that the set of maximal cones $\Sigma_{\max}$ admits a suitable ordering.
The purpose of this section is to apply this computational result to $\Theta_{n,r,\bd}$ for integers $0\leq r\leq n$ and a finite sequence $\bd$ of positive integers.
For this,
we impose a suitable ordering on the maximal cones of $\Theta_{n,r,\bd}$.

\subsection{Admissible ordering}
In this subsection,
we will introduce the notion of an admissible ordering on the set of maximal cones $\Sigma_{\max}$ of an $r$-standard subdivision $\Sigma$ of $(\P^1)^n$.
We will show in Proposition \ref{ordering.22} that $(\Theta_{n,r,\bd})_{\max}$ admits an admissible ordering for all $0\leq r\leq n$ and $\bd$.

\begin{df}
Let $\Sigma$ be a smooth complete fan.
Suppose that we have an ordering $>$ on $\Sigma_{\max}$.
For $\sigma\in \Sigma_{\max}$,
we set
\[
\ess(\sigma)
:=
\bigcap_{\tau>\sigma, \dim\lvert \tau\cap \sigma \rvert = n-1} \tau,
\]
where $\ess$ stands for \emph{essence}.
If the indexing set is empty, then we set $\ess(\sigma):=\sigma$.
\end{df}

\begin{df}
\label{ordering.6}
Let $0\leq r\leq n$ be integers.
For an $r$-standard subdivision $\Sigma$ of $(\P^1)^n$,
let
\[
\Sigma^\circ
:=
\{\sigma\in \Sigma:\Cone(\sigma,e_i)\in \Sigma \text{ for some }r+1\leq i\leq n\}.
\]
If $r<n$,
then $\Sigma^\circ$ is nonempty and hence a subfan of $\Sigma$.

Let $\Sigma^\flat$ be the set $\Sigma-\Sigma^\circ$,
i.e.,
the set of cones of $\Sigma$ not in $\Sigma^\circ$.
Observe that $\Sigma^\flat$ does not need to be a fan.

See Definition \ref{ordering.28} and Proposition \ref{ordering.3} for the reason why we introduce $\Sigma^\circ$.
\end{df}

\begin{df}
\label{ordering.1}
Let $\Sigma$ be an $r$-standard subdivision of $(\P^1)^n$ with integers $0\leq r\leq n$.
For a ray $f$ of $\sigma\in \Sigma_{\max}$,
let $\widehat{\sigma}_f\in \Sigma_{\max}$ be the unique maximal cone of $\Sigma$ such that $f\not\in\widehat{\sigma}_f$ and $\dim\lvert \sigma\cap \widehat{\sigma}_f\rvert = n-1$.
An ordering on $\Sigma_{\max}$ is \emph{pre-admissible} if the following conditions are satisfied.
\begin{enumerate}
\item[(i)] $\ess(\sigma)\subset \tau$ implies 
$\sigma<\tau$ for all $\sigma,\tau\in \Sigma_{\max}$.
\item[(ii)] $\Cone( e_1,\ldots,e_n)$ is the least element of $\Sigma_{\max}$.
\item[(iii)] 
For $\sigma\in \Sigma_{\max}$ different from $\Cone( e_1,\ldots,e_n)$,
there exists $\tau\in \Sigma_{\max}$ such that $\tau<\sigma$ and $\dim \lvert \sigma\cap \tau \rvert=n-1$.
\item[(iv)] $e_i\in \sigma$ implies $\sigma < \widehat{\sigma}_{e_i}$ for $\sigma\in \Sigma_{\max}$ and $r+1\leq i\leq n$.
\item[(v)] If $e_i\not\in \sigma$ and $r+1\leq i\leq n$,
then the cone generated by $e_i$ and $\ess(\sigma)$ is not in $\Sigma$.
\end{enumerate}
A pre-admissible ordering on $\Sigma_{\max}$ is \emph{admissible} if the following conditions is also satisfied:
\begin{enumerate}
\item[(vi)] If $\sigma\in \Sigma_{\max}^\circ$ and $\tau\in \Sigma_{\max}-\Sigma_{\max}^\circ$,
then $\sigma<\tau$.
\end{enumerate}
\end{df}

\begin{exm}
We have an admissible ordering on $(\Gamma_{3,1})_{\max}$ as follows:
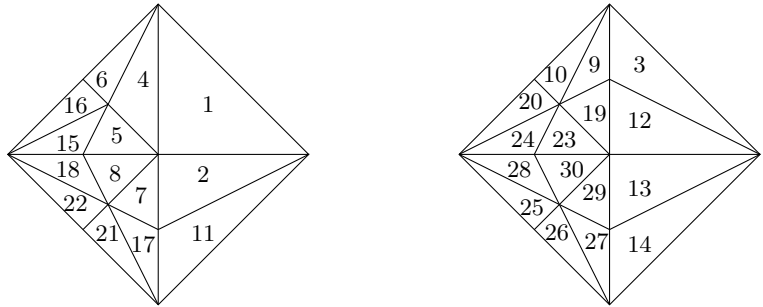
\[
\begin{tikzpicture}
\draw (2,0)--(0,2)--(-2,0)--(0,-2)--(2,0)--(0,0)--(-2,0);
\draw (0,2)--(0,0)--(0,-2);
\draw (0,2)--(-1,0);
\draw (0,0)--(-1,1);
\draw (-2,0)--(-2/3,2/3);
\draw (0,-2)--(-1,0);
\draw (0,0)--(-1,-1);
\draw (-2,0)--(0,-1);
\draw (2,0)--(0,-1);
\draw (2/3,2/3) node {\small 1};
\draw (-0.2,1) node {\small 4};
\draw (0.6,-0.27) node {\small 2};
\draw (0.6,-1.05) node {\small 11};
\draw (-0.55,0.25) node {\small 5};
\draw (-0.75,1) node {\small 6};
\draw (-0.23,-0.55) node {\small 7};
\draw (-0.57,-0.25) node {\small 8};
\draw (-1.2,0.15) node {\small 15};
\draw (-1.1,0.65) node {\small 16};
\draw (-0.2,-1.2) node {\small 17};
\draw (-1.2,-0.2) node {\small 18};
\draw (-0.67,-1.05) node {\small 21};
\draw (-1.1,-0.65) node {\small 22};
\begin{scope}[shift={(6,0)}]
\draw (2,0)--(0,2)--(-2,0)--(0,-2)--(2,0)--(0,0)--(-2,0);
\draw (0,2)--(0,0)--(0,-2);
\draw (2,0)--(0,1);
\draw (2,0)--(0,-1);
\draw (0,0)--(-1,-1);
\draw (0,0)--(-1,1);
\draw (-2,0)--(0,1);
\draw (0,2)--(-1,0);
\draw (-2,0)--(0,-1);
\draw (0,-2)--(-1,0);
\draw (0.4,1.2) node {\small 3};
\draw (0.4,0.45) node {\small 12};
\draw (0.4,-1.2) node {\small 14};
\draw (0.4,-0.45) node {\small 13};
\draw (-0.2,1.2) node {\small 9};
\draw (-0.72,1.06) node {\small 10};
\draw (-0.2,0.55) node {\small 19};
\draw (-1.05,0.7) node {\small 20};
\draw (-0.6,0.2) node {\small 23};
\draw (-1.15,0.2) node {\small 24};
\draw (-1.03,-0.7) node {\small 25};
\draw (-0.7,-1.04) node {\small 26};
\draw (-0.17,-1.15) node {\small 27};
\draw (-1.2,-0.2) node {\small 28};
\draw (-0.2,-0.5) node {\small 29};
\draw (-0.5,-0.2) node {\small 30};
\end{scope}
\end{tikzpicture}
\]
\end{exm}

Now, we begin the proof that $\Theta_{n,r,\bd}$ admits an admissible ordering. We first treat the base pre-admissible case.

\begin{lem}
\label{ordering.23}
Let $0\leq r\leq n$ be integers.
Regard $(\P^1)^n$ as an $r$-standard subdivision of $(\P^1)^n$.
Then $(\P^1)^n$ admits a pre-admissible ordering.
\end{lem}
\begin{proof}
We need to check the conditions (i)--(v) in Definition \ref{ordering.1}.
Any maximal cone of $(\P^1)^n$ is of the form
\[
\eta_x:=\Cone( x_1 e_1,\ldots,x_n e_n),
\]
where $x:=(x_1,\ldots,x_n)\in \{-1,1\}^n$.
For $y\neq x\in \{-1,1\}^n$,
we assign $\eta_x>\eta_y$ if the following condition is satisfied: If $i$ is the smallest integer such that $x_i\neq y_i$, then $x_i<y_i$.
The condition (ii) is clear.

If $x\in \{-1,1\}^n$ and $x\neq (1,\ldots,1)$,
then consider $y\in \{-1,1\}^n$ obtained from $x$ by changing one $-1$ into $1$.
We have $\eta_y<\eta_x$, which proves the condition (iii).

If $e_i\in \eta_x$ for some $1\leq i\leq n$,
then $x_i=1$,
and consider $z\in \{-1,1\}^n$ obtained from $x$ by changing $x_i$ into $-1$.
We have $\eta_x<\eta_z$ and $(\widehat{\eta_x})_{e_i}=\eta_z$,
which proves the condition (iv).

Observe that $\ess(\eta_x)$ is the cone generated by $-e_i$ for all $1\leq i\leq n$ such that $x_i=-1$.
If $y\in \{-1,1\}^n$ and $\ess(\eta_x)\subset \eta_y$,
then $x_i=-1$ implies $y_i=-1$.
Hence we have $\eta_x<\eta_y$, which proves the condition (i).

If $e_i\notin \eta_x$ for some $1\leq i\leq n$,
then $x_i=-1$.
Since $\Cone( e_i,-e_i) \notin (\P^1)^n$,
we have $\Cone( e_i,\ess(\eta_x)) \notin (\P^1)^n$.
This proves the condition (v).
\end{proof}

Next, we discuss an induction step.

\begin{lem}
\label{ordering.24}
Let $0\leq r\leq n$ be integers,
let $\Sigma$ be an $r$-standard subdivision of $(\P^1)^n$,
and let $A$ be a subfan of $\Sigma$ satisfying the condition in \textup{Definition \ref{dist.12}}.
Assume that $\Sigma$ admits an admissible ordering.
Then $\Delta:=\Sigma_\eta^\bary(A)$ admits a pre-admissible ordering,
where $\eta:=\Cone(e_{r+1},\ldots,e_n)$.
\end{lem}
\begin{proof}
Step 1. \emph{Structure of $\Delta$.}
Let
\[
\sigma:=\Cone(f_1,\ldots,f_n)
\]
be a maximal cone of $\Sigma$ different from $\Cone(e_1,\ldots,e_n)$.
Using the condition in Definition \ref{dist.12},
without loss of generality,
we may assume that $\Cone(f_1,\ldots,f_m)$ is in $A$ and contains every other cone of $A$ contained in $\sigma$ for some integer $0\leq m\leq n$.
We may also assume
\[
f_1,\ldots,f_t\in \{e_{r+1},\ldots,e_n\},
\text{ }
f_{t+1},\ldots,f_m\notin \{e_{r+1},\ldots,e_n\}.
\]
By the condition (iv), we have $\sigma<
\widehat{\sigma}_{f_1},\ldots,\widehat{\sigma}_{f_t}$.
Hence we may assume
\begin{equation}
\label{ordering.22.1}
\widehat{\sigma}_{f_{s+1}},\ldots, \widehat{\sigma}_{f_m} < \sigma < 
\widehat{\sigma}_{f_1},\ldots,\widehat{\sigma}_{f_s}
\end{equation}
for some integer $t\leq s\leq m$.
We may also assume
\begin{equation}
\widehat{\sigma}_{f_{m+1}},\ldots, \widehat{\sigma}_{f_l} < \sigma < 
\widehat{\sigma}_{f_{l+1}},\ldots,\widehat{\sigma}_{f_n}
\end{equation}
for some integer $m\leq l\leq n$.
Here, if $l=m$ (resp.\ $l=n$), then interpret the first (resp.\ second) inequality as the empty condition.
Then we have
\begin{equation}
\label{ordering.22.3}
\ess(\sigma)=
\widehat{\sigma}_{f_1} \cap \cdots \cap \widehat{\sigma}_{f_s}
\cap
\widehat{\sigma}_{f_{l+1}}
\cap
\cdots
\cap
\widehat{\sigma}_{f_n}
=
\Cone( f_{s+1},\ldots,f_l).
\end{equation}
The condition (iii) implies $s<l$.

By Proposition \ref{subdivision.32},
the unique maximal cone of $\Delta$ contained in $\sigma$ is of the form
\[
\sigma_\alpha
:=
\Cone( f_{\alpha(1)},\ldots,f_{\alpha(c)},g_{c+1},\ldots,g_m,f_{m+1},\ldots,f_n)
\]
with
\[
g_i:=f_{\alpha(1)}+\cdots +f_{\alpha(i)}
\]
for $c+1\leq i\leq m$,
where $\alpha$ is a $t$-admissible permutation of $\{1,\ldots,m\}$,
and $c$ is the smallest integer satisfying $\alpha(c+1)>t$.
Recall from Definition \ref{subdivision.31} that we have $\alpha(1)<\cdots<\alpha(c)\leq t$.
We also have $c\leq t$.

We will keep using the above notation.
Also, keep in mind that we have the inequalities
\[
0\leq c\leq t\leq s\leq m \leq l\leq n,
\text{ }
s<l.
\]

Step 2. \emph{Assigning an ordering}.
Let $\tau'$ be a cone of $\Delta$ such that the maximal cone $\tau$ of $\Sigma$ containing $\tau'$ is different from $\sigma$.
If $\sigma>\tau$ (resp.\ $\sigma<\tau$),
then we assign $\sigma_\alpha>_\Delta\tau'$ (resp.\ $\sigma_\alpha<_{\Delta}\tau'$).
We use the notation $>_\Delta$ for the new ordering to distinguish with the original ordering $>$.

It remains to assign an ordering for two maximal cones of $\Delta$ contained in $\sigma$.
If $\beta$ is another $t$-admissible permutation of $\{1,\ldots,m\}$,
then we assign $\sigma_\alpha<_{\Delta}\sigma_\beta$ if the following condition is satisfied:
If $j$ is the largest number such that $\alpha(j)\neq \beta(j)$, then we have $\alpha(j)>\beta(j)$.

Step 3. \emph{Computation of $\ess(\sigma_\alpha)$}.
For $1\leq i\leq c$,
we have
$\sigma<\widehat{\sigma}_{f_{\alpha(i)}}$ by \eqref{ordering.22.1}.
This implies
$\sigma_\alpha<_{\Delta}(\widehat{\sigma_\alpha})_{f_{\alpha(i)}}$.

For $c+1\leq i\leq m-1$,
let $\alpha_i$ is the permutation of $\{1,\ldots,m\}$ obtained by switching $\alpha(i)$ and $\alpha(i+1)$ in $\alpha$.
We have $(\widehat{\sigma_\alpha})_{g_i}=\sigma_{\alpha_i}$ since
\[
(\widehat{\sigma_\alpha})_{g_i}\cap \sigma_{\alpha_i}
=
\Cone(f_{\alpha(1)},\ldots,f_{\alpha(c)},g_{c+1},\ldots,g_{i-1},g_{i+1},\ldots,g_m,f_{m+1},\ldots,f_n)
\]
has dimension $n-1$.
We also have $(\widehat{\sigma_\alpha})_{g_m}\subset \widehat{\sigma}_{f_{\alpha(m)}}$ since
\[
(\widehat{\sigma_\alpha})_{g_m}
\cap
\widehat{\sigma}_{f_{\alpha(m)}}
=
\Cone(f_{\alpha(1)},\ldots,f_{\alpha(c)},g_{c+1},\ldots,g_{m-1},f_{m+1},\ldots,f_m)
\]
has dimension $n-1$.
Hence $\sigma_\alpha<_{\Delta}(\widehat{\sigma_\alpha})_{g_m}$ if and only if $\alpha(m)\leq s$ by \eqref{ordering.22.1}.

Consider the subset $I_\alpha\subset \{c+1,c+2,\ldots,m-1\}$ such that $i\in I_\alpha$ if and only if $\alpha(i)>\alpha(i+1)$.
Then for $c+1\leq i\leq m-1$,
we have $\sigma_{\alpha_i}<_{\Delta}\sigma_\alpha$ if and only if $i\in I_\alpha$.
If $I_\alpha=\{i_1,\ldots,i_w\}$,
then combine what we have discussed above to have
\begin{equation}
\label{ordering.22.5}
\ess(\sigma_\alpha)
=
\left\{
\begin{array}{ll}
\Cone( g_{i_1},g_{i_2},\cdots, g_{i_w}, f_{m+1},\ldots,f_l) & \text{if $\alpha(m)\leq s$},
\\
\Cone( g_{i_1},g_{i_2},\cdots, g_{i_w},g_m,f_{m+1},\ldots,f_l) & \text{otherwise}.
\end{array}
\right.
\end{equation}

If $\alpha(m)\leq s$,
then we have $\alpha(i_w+1)<\cdots<\alpha(m)\leq s$.
Hence we have $\{s+1,\ldots,n\}\subset \{\alpha(1),\ldots,\alpha(i_w)\}$.
If $\alpha(m)\geq s$,
then we have $g_n\in \ess(\sigma_\alpha)$.
It follows that in both cases,
the smallest cone of $\Sigma$ containing $\ess(\sigma_\alpha)$ contains $\ess(\sigma)$.

Step 4. \emph{Verification of the condition }(ii).
By the condition (ii) for $\Sigma$,
we have $\Cone( e_1,\ldots,e_n) < \sigma$.
This implies $\Cone( e_1,\ldots,e_n ) <_{\Delta} \sigma_\alpha$.

Step 5. \emph{Verification of the condition }(iii).
Assume $\alpha=\id$.
If $l>m$,
then we have $\sigma_\alpha>_{\Delta}(\widehat{\sigma_\alpha})_{f_l}$ since $\sigma>\widehat{\sigma}_{f_l}$.
If $l=m$,
then $s<m$,
and hence we have $\sigma_\alpha>_{\Delta}(\widehat{\sigma_\alpha})_{g_m}$ as observed in Step 3.

If $\alpha\neq \id$,
then we have $I_\alpha\neq \emptyset$,
and we have $\sigma_{\alpha_i}<_{\Delta}\sigma_\alpha$ for $i\in I_\alpha$ as observed in Step 3.
Hence in any case,
$\sigma_\alpha$ satisfies the condition (iii).

Step 6. \emph{Verification of the condition }(iv).
Observe that we have
\[
\{f_1,\ldots,f_m\}\cap \{e_{r+1},\ldots,e_n\}=\{f_{\alpha(1)},\ldots,f_{\alpha(c)}\}.
\]
For $1\leq i\leq c$,
we have $\sigma_\alpha<_{\Delta}(\widehat{\sigma_\alpha})_{f_{\alpha(i)}}$ as observed in Step 3.
If $f_i\in \{e_{r+1},\ldots,e_n\}$ for some $i\geq m+1$,
then we have $\sigma_\alpha<_{\Delta}(\widehat{\sigma_\alpha})_{f(i)}$ since $\sigma<\widehat{\sigma}_{f(i)}$.
Hence $\sigma_\alpha$ satisfies the condition (iv).

Step 7. \emph{Verification of the condition} (i).
Consider the maximal cone $\sigma_\beta$ of $\Delta$ contained in $\sigma$ with a $t$-admissible permutation $\beta\neq \alpha$ of $\{1,\ldots,m\}$.
Assume $\ess(\sigma_\alpha)\subset \sigma_\beta$.
We need to show $\sigma_\alpha<_{\Delta}\sigma_\beta$.
Note that the rays of $\sigma_\beta$ are $f_{\beta(1)},\ldots,f_{\beta(v)},f_{m+1},\ldots,f_n$ and
\begin{equation}
\label{ordering.22.4}
f_{\beta(1)}+\cdots+f_{\beta(v+1)},\ldots,f_{\beta(1)}+\cdots+f_{\beta(m)},
\end{equation}
where $v$ is the smallest integer satisfying $\beta(v+1)>t$, which also satisfies $v\leq t$.
The assumption $\ess(\sigma_\alpha)\subset \sigma_\beta$ implies that $g_{i_j}=f_{\alpha(1)}+\cdots + f_{\alpha(i_j)}$ is a ray of $\sigma_\beta$ for $1\leq j\leq w$ and $g_{i_j}\notin \{e_{r+1},\ldots,e_n\}$.
Hence $g_{i_j}$ is equal to one of \eqref{ordering.22.4} that is the sum of $i_j$ elements,
so we have
\[
f_{\alpha(1)}+\cdots+f_{\alpha(i_j)}
=
f_{\beta(1)}+\cdots+f_{\beta(i_j)}
\]
for $1\leq j\leq w$.
We also have
\[
f_{\alpha(1)}+\cdots+f_{\alpha(m)}
=
f_{\beta(1)}+\cdots+f_{\beta(m)}.
\]
It follows that we have
\begin{equation}
\label{ordering.22.2}
\{\alpha(i_{j-1}+1),\ldots,\alpha(i_j)\}
=
\{\beta(i_{j-1}+1),\ldots,\beta(i_j)\}
\end{equation}
for $1\leq j\leq w+1$,
where we set $i_0:=v$ and $i_{w+1}:=m$.
We also have
\[
\{\alpha(1),\ldots,\alpha(v)\}
=
\{\beta(1),\ldots,\beta(v)\}
\]
and hence $\alpha(i)=\beta(i)$ for $1\leq i\leq v$.
Let $k$ be the largest number such that $\alpha(k)\neq \beta(k)$.
Then there exists $1\leq j\leq w+1$ such that $i_{j-1}+1\leq k\leq i_j$.
By the definition of $I_\alpha$,
we have $\alpha(i_{j-1}+1)<\cdots <\alpha(i_j)$.
Together with \eqref{ordering.22.2},
we have $\alpha(k)>\beta(k)$.
Hence we have $\sigma_\alpha<_{\Delta} \sigma_\beta$.

Let $\tau$ be a maximal cone of $\Sigma$ different from $\sigma$,
and let $\tau'$ be a maximal cone of $\Delta$ contained in $\tau$.
Assume $\ess(\sigma_\alpha)\subset \tau'$.
We need to show $\sigma_\alpha <\tau'$.
We have $\ess(\sigma_\alpha)\subset \tau$.
Since $g_m$ is not contained in any facet of $\sigma$, we have $g_m\notin \tau$,
which implies $\alpha(m)\leq s$.
We also have $g_{i_w}\in \tau$,
which implies $f_{\alpha(i)}\in \tau$ for $1\leq i\leq i_w$.
Since we have $\alpha(i_w+1)<\cdots <\alpha(m)\leq s$,
we have
\[
\{s+1,\ldots,m\} \subset \{\alpha(1),\ldots,\alpha(i_w)\}.
\]
Hence we have $\Cone( f_{s+1},\ldots,f_m) \subset \tau$.
Furthermore,
the inclusion $\ess(\sigma_\alpha)\subset \tau$ implies $\Cone(f_{m+1},\ldots,f_l)\subset \tau$ by \eqref{ordering.22.5}.
Since $\ess(\sigma)=\Cone( f_{s+1},\ldots,f_l)$ by \eqref{ordering.22.3},
the condition (i) for $\Sigma$ implies $\sigma<\tau$.
It follows that we have $\sigma_\alpha<_{\Delta}\tau'$.
Hence $\sigma_\alpha$ satisfies the condition (i).

Step 8. \emph{Verification of the condition }(v).
Assume that $r+1\leq i\leq n$, $e_i\notin \sigma_\alpha$, and the cone generated by $e_i$ and $\ess(\sigma_\alpha)$ is in $\Delta$.
Note that we have $e_i\notin \{f_{m+1},\ldots,f_n\}$ since $\sigma_\alpha$ contains $f_{m+1},\ldots,f_n$.

Assume $e_i=f_j$ for some $1\leq j\leq t$.
If $\alpha=\id$,
then $c=t$ and $f_j\in \sigma_\alpha$,
which is a contradiction.
If $\alpha \neq \id$,
then $I_\alpha \neq 0$,
and we have $\Cone( f_j,g_{i_1}) \in \Delta$.
This implies $j\in \{\alpha(1),\ldots,\alpha(i_1)\}$ since otherwise $f_j+g_{i_1}$ is a ray of $\Delta$.
The inequality
\[
\alpha(1)<\cdots < \alpha(t)\leq s<\alpha(t+1)<\cdots <\alpha(i_1)
\]
implies $j\in \{\alpha(1),\ldots,\alpha(t)\}$.
Hence we have $f_j\in \sigma_\alpha$,
which is a contradiction.

Assume $e_i\neq f_j$ for all $1\leq j\leq t$.
Let $\tau$ be the maximal cone of $\Sigma$ containing $e_i$ and $\ess(\sigma_\alpha)$.
As observed in Step 3,
the smallest cone of $\Sigma$ containing $\ess(\sigma_\alpha)$ contains $\ess(\sigma)$.
Hence $\tau$ contains $e_i$ and $\ess(\sigma)$.
However, the cone generated by $e_i$ and $\ess(\sigma)$ is not a cone of $\Sigma$ by the condition (v) for $\Sigma$,
which is a contradiction.

This shows that $\sigma_\alpha$ satisfies the condition (v).
\end{proof}

The existence of a pre-admissible ordering gauntness an admissible ordering as follows.

\begin{lem}
\label{ordering.25}
Let $\Sigma$ be an $r$-standard subdivision of $(\P^1)^n$,
where $0\leq r\leq n$ are integers.
If $\Sigma_{\max}$ admits a pre-admissible ordering,
then it admits an admissible ordering.
\end{lem}
\begin{proof}
Assume that an ordering $<$ on $\Sigma_{\max}$ satisfies the conditions (i)--(v) in Definition \ref{ordering.1}.
Now, for $\sigma',\tau'\in \Sigma_{\max}$,
we assign $\sigma'<_\mathrm{new}\tau'$ for the following three cases:
\begin{itemize}
\item
$\sigma',\tau'\in \Sigma_{\max}^\circ$ and $\sigma'<\tau'$.
\item
$\sigma'\in \Sigma_{\max}^\circ$ and $\tau'\notin \Sigma_{\max}^\circ$.
\item
$\sigma',\tau'\notin \Sigma_{\max}^\circ$ and $\sigma'<\tau'$.
\end{itemize}
The condition (vi) is obvious for the new ordering.

Assume $\sigma',\tau'\in \Sigma_{\max}$ and $\dim \lvert \sigma'\cap \tau'\rvert=n-1$.
If $\sigma',\tau'\in \Sigma_{\max}^\circ$ or $\sigma',\tau'\notin \Sigma_{\max}^\circ$,
then $\sigma'<\tau'$ if and only if $\sigma'<_\mathrm{new} \tau'$.
If $\sigma'\in \Sigma_{\max}^\circ$ and $\tau'\notin \Sigma_{\max}^\circ$, then $\sigma'<_\mathrm{new}\tau'$, and $\sigma'=\widehat{\tau'}_{e_i}$ for some integer $r+1\leq i\leq n$.
Hence the condition (iv) implies $\sigma'<\tau'$.
In particular, $\sigma'<\tau'$ is equivalent to $\sigma'<_\mathrm{new} \tau'$ whenever $\dim \lvert \sigma'\cap \tau'\rvert=n-1$.
This implies that $\ess(\sigma')$ does not change for $<_\mathrm{new}$ and the conditions (i)--(v) for $<_\mathrm{new}$ is satisfied.
\end{proof}

Now, we combine the above three results to conclude as follows.

\begin{prop}
\label{ordering.22}
Let $0\leq r\leq n$ be integers,
and let $\bd$ be a finite sequence of positive integers.
Then $(\Theta_{n,r,\bd})_{\max}$ admits an admissible ordering.
\end{prop}
\begin{proof}
We know that $((\P^1)^n)_{\max}$ admits a pre-admissible ordering by Lemma \ref{ordering.23}.
Using Lemma \ref{ordering.24} repeatedly,
we see that $(\Theta_{n,r,\bd})_{\max}$ admits a pre-admissible ordering.
Lemma \ref{ordering.25} finishes the proof.
\end{proof}

\subsection{Application to Chow groups}

For a fan $\Sigma$ and its cone $\sigma$,
consider the projection
\begin{equation}
\label{ordering.5.2}
\overline{(-)}\colon N(\Sigma)\to N(\Sigma)/N_\sigma.
\end{equation}
See Definition \ref{intro.3} for the fan $V(\sigma)$.
We have a one-to-one correspondence of sets
\begin{equation}
\label{ordering.5.1}
\ol{(-)}
\colon
\{\tau\in \Sigma: \sigma\prec \tau\}
\xrightarrow{\cong}
V(\sigma).
\end{equation}
Recall that $\sigma \prec \tau$ means $\sigma$ is a face of $\tau$.
Do not confuse $\sigma\prec \tau$ with $\sigma<\tau$.
Let $\Sigma_{\C}$ be the toric variety over $\C$ associated with $\Sigma$. Let $O(\sigma)$ be the $0$ fan in the lattice $N(V(\sigma))$ so that $O(\sigma)_{\C}$ is the maximal torus of the toric variety $V(\sigma)$.
For integers $d$,
let $\CH_d(\Sigma)$ be the Chow group $\CH_d(\Sigma_{\C})$,
and let $H_d^\BM(\Sigma)$ be the Borel-Moore homology $H_d^\BM(\Sigma_{\C})$.
Recall from \cite[p.\ 53]{Fulton:1436535} that we can regard $V(\sigma)_{\C}$ as a closed subscheme of $\Sigma_{\C}$.
If $\sigma$ has dimension $d$,
let $[V(\sigma)]\in \CH_d(\Sigma)$ be the element $[V(\sigma)_{\C}]\in \CH_d(\Sigma_{\C})$.

\begin{df}
\label{ordering.28}
For an $r$-standard subdivision $\Sigma$ of $(\P^1)^n$ with integers $0\leq r\leq n$,
we set
\[
\CH_*^\flat(\Sigma)
:=
\CH_*(\Sigma_\C-\Sigma_\C^\circ).
\]
\end{df}

\begin{prop}
\label{ordering.19}
Let $\Sigma$ be a smooth complete fan.
If $\Sigma_{\max}$ admits an ordering satisfying the condition \textup{(i)} in \textup{Definition \ref{ordering.1}},
then $\CH_*(\Sigma)$ is the free abelian group generated by $[V(\ess(\sigma))]$ for $\sigma\in \Sigma_{\max}$.
\end{prop}
\begin{proof}
See \cite[Theorem in p.\ 102]{MR1415592}.
\end{proof}

\begin{prop}
\label{ordering.17}
Let $\Sigma$ be an $r$-standard subdivision of $(\P^1)^n$ with integers $0\leq r\leq n$.
Assume that $\Sigma_{\max}$ admits an admissible ordering.
Then we have the following properties.
\begin{enumerate}
\item[\textup{(1)}] $\CH_*(\Sigma)\cong H_{2*}^\BM(\Sigma)$, and $H_i^\BM(\Sigma)=0$ for odd $i$.
\item[\textup{(2)}] $\CH_*(\Sigma^\circ)\cong H_{2*}^\BM(\Sigma^{ \circ})$, and $H_i^\BM(\Sigma^{\circ})=0$ for odd $i$.
\item[\textup{(3)}] $\CH_*^\flat(\Sigma)\cong H_{2*}^\BM (\Sigma_\C-\Sigma_\C^\circ)$, and $H_i^\BM(\Sigma_\C-\Sigma_\C^\circ)=0$ for odd $i$.
\item[\textup{(4)}] $\CH_*(\Sigma^\circ)$ is the free abelian group generated by $[V(\ess(\sigma))]$ for $\sigma\in \Sigma_{\max}^\circ$.
\item[\textup{(5)}] $\CH_*^\flat(\Sigma)$ is the free abelian group generated by $[V(\ess(\sigma))]$ for $\sigma\in \Sigma_{\max}-\Sigma_{\max}^\circ$.
\end{enumerate}
\end{prop}
\begin{proof}
Let $\sigma_1<\cdots<\sigma_m$ be an admissible ordering on $\Sigma_{\max}$.
For $1\leq i\leq m$,
consider
\[
Y_i:=\bigcup_{\ess(\sigma_i)\subset \gamma\subset \sigma_i} O(\gamma)_{\C}
\]
and $Z_i:=Y_i\cup Y_{i+1}\cup \cdots \cup Y_m$ in \cite[p.\ 103]{Fulton:1436535}.
Note that $Z_1=\Sigma_\C$ by \cite[Lemma (1) in p.\ 103]{Fulton:1436535}.
The argument in \cite[p.\ 103]{Fulton:1436535} shows that $\CH_*(\Sigma)\cong H_{2*}^\BM(\Sigma_{\C})$, $H_*^\BM(Z_i)$ vanishes in odd degrees, $H_*^\BM(Z_i)\to H_*^\BM(Z_{i+1})$ is injective for every $i$,
and $\CH_*(Z_i)$ is the free abelian group generated by the classes $[V(\sigma_i)]$ for $j\geq i$.
Hence $H_*^\BM(Z_i)\to H_*^\BM(\Sigma_{\C})$ is injective.
Together with the exact sequence
\[
H_{2*-1}^\BM(\Sigma_{\C}) \to H_{2*-1}^\BM(\Sigma_{\C}-Z_i)\to H_{2*}^\BM(Z_i)\to H_{2*}^\BM(\Sigma_{\C}),
\]
we have $H_{2d-1}^\BM(\Sigma_{\C}-Z_i)=0$.
Hence we have the commutative diagram with exact rows
\[
\begin{tikzcd}
&
\CH_*(Z_i)\ar[d,"\cong"']\ar[r]&
\CH_*(\Sigma)\ar[d,"\cong"]\ar[r]&
\CH_*(\Sigma_{\C}-Z_i)\ar[d]\ar[r]&
0
\\
0\ar[r]&
H_{2*}^\BM(Z_i)\ar[r]&
H_{2*}^\BM(\Sigma)\ar[r]&
H_{2*}^\BM(\Sigma_\C-Z_i)\ar[r]&
0.
\end{tikzcd}
\]
The five lemma implies that we have $\CH_*(\Sigma_{\C}-Z_i)\cong H_{2*}^\BM(\Sigma_{\C}-Z_i)$.
By the condition (vi) in Definition \ref{ordering.1},
there exists an integer $1\leq q\leq m$ such that $i\leq q$ if and only if $\sigma_q\in \Sigma^\circ$.
Since $\CH_*(Z_i)$ (resp.\ $\CH_*(\Sigma)$) is the free abelian group generated by $[V(\sigma_j)]$ for $j\geq i$ (resp.\ all $j$),
$\CH_*(\Sigma_{\C}-Z_i)$ is the free abelian group generated by $[V(\sigma_j)]$ for $j< i$.
To conclude,
observe that we have $Z_q=\Sigma_\C-\Sigma_\C^\circ$ and $\Sigma_{\C}-Z_q=\Sigma^\circ_\C$.
\end{proof}

\section{Resolution of toric Chow homology}
\label{resolution}

The purpose of this section is to show Proposition \ref{ordering.15}, which shows that a natural chain complex $\CH_p^\flat(\Theta_{r+\bullet,r,\bd})$ is quasi-isomorphic to $0$ for $0\leq p\leq r-1$.
This is a base step of the inductive argument in the proof of Theorem \ref{intro.2}, see Lemma \ref{subdivision.4} below.
To prove Proposition \ref{ordering.15},
we will first construct a free resolution $Z_{p.\bullet}^\flat(\Sigma)$ of $\CH_p^\flat(\Sigma)$ for every $r$-standard subdivision $\Sigma$ of $(\P^1)^n$ in Proposition \ref{ordering.16}.
Then we will show that each chain complex $Z_{p.\bullet}^\flat(\Theta_{n,r,\bd})$ is quasi-isomorphic to $0$, which requires cubical identities for $\Z_{p.\bullet}^\flat(\Theta_{n,r,\bd})$ in \S \ref{identities}.
After that, a usual argument using the spectral sequence for the double complex $Z_{p,\bullet}^\flat(\Theta_{r+\bullet,r,\bd})$ will finish the proof.

\subsection{A spectral sequence converging to Chow groups of toric varieties}

Let $\Sigma$ be an $n$-dimensional fan with an integer $n\geq 0$.
Recall from \cite[\S 5]{MR3264256} that we have a spectral sequence
\begin{equation}
\label{moving.1.1}
E_{p,q}^1=C_{p,q}(\Sigma):=\bigoplus_{\sigma \in \Sigma(n-p)} H_{p+q}^\BM( O(\sigma))
\Rightarrow
H_{p+q}^\BM(\Sigma),
\end{equation}
where $\Sigma(n-p)$ denotes the set of $(n-p)$-dimensional cones of $\Sigma$.
Note that $V(\sigma)$ and $O(\sigma)$ have dimension $p$ if $\sigma$ has dimension $n-p$.
By \cite[Corollary 6.3]{MR2255969},
this degenerates at the second page, and there is no extension problem.
For simplicity of notation,
we also set $C_{p,q}(\Sigma)_\sigma:=H_{p+q}^\BM(\cO(\sigma))$.

Following \cite[\S 6]{MR2255969},
we can describe $Z_{p,q}(\Sigma):=C_{p+q,p}(\Sigma)$ as follows.
We have
\begin{equation}
\label{moving.1.2}
Z_{p,q}(\Sigma)\cong\bigoplus_{\sigma \in \Sigma(n-p-q)} \wedge^q M(\sigma).
\end{equation}
We write an element of $Z_{p,q}(\Sigma)$ as $\sum_{\sigma\in \Sigma(n-p-q)} \alpha_\sigma [V(\sigma)]$,
where $\alpha_\sigma\in \wedge^q M(\sigma)$,
and $[V(\sigma)]$ is a symbol for the direct summand $\wedge^q M(\sigma)$.
As noted in \cite[Equation (9)]{MR2255969},
the differential $d\colon Z_{p,q}(\Sigma)\to Z_{p,q-1}(\Sigma)$ is given by
\[
d(\alpha[V(\sigma)])
:=
\sum_{\sigma \prec_1 \tau}
 u_{\tau,\sigma} \iprod \alpha
[V(\tau)]
\]
for $\sigma \in \Sigma(n-p)$ and $\alpha\in \wedge^q M(\sigma)$,
where $u_{\tau,\sigma}\in \tau$ is a lattice point whose image generates $N_\tau/N_\sigma$, $u_{\tau,\sigma}\iprod \colon \wedge^q M(\sigma) \to \wedge^{q-1} M(\tau)$ is the interior product,
and the notation $\sigma \prec_1 \tau$ means that $\sigma$ is a facet of $\tau$.
Here, the interior product is defined using the formula
\[
\langle \beta,u_{\tau,\sigma}\iprod \alpha
\rangle
=
\langle \beta'\wedge u_{\tau,\sigma} ,\alpha
\rangle
\]
for all $\beta\in \wedge^{q-1}(N(\Sigma)/N_\tau)$,
where $\langle -,-\rangle$ is the duality pairing, $u_{\tau,\sigma}$ is regarded as an element of $N(\Sigma)/N_\sigma$ in the right-hand side, and $\beta'$ is any lifting of $\beta$ to $\wedge^{q-1}(N(\Sigma)/N_\sigma)$.

\begin{prop}
\label{moving.1}
Let $\Sigma$ be a smooth $n$-dimensional fan such that the cycle class map $\CH_*(\Sigma)\to H_{2*}^\BM(\Sigma)$ is an isomorphism and $H_*^\BM(\Sigma)$ vanishes in odd degrees.
Then for every integer $p\geq 0$, there is a natural resolution of $\CH_p(\Sigma)$:
\begin{equation}
\cdots \xrightarrow{d} Z_{p,1}(\Sigma) \xrightarrow{d} Z_{p,0}(\Sigma) \to \CH_p(\Sigma)\to 0.
\end{equation}
\end{prop}
\begin{proof}
Consider the spectral sequence $E_{p,q}^1$ in \eqref{moving.1.1},
which converges to $H_{p+q}^\BM(\Sigma)$ and degenerates at the second page.
Then there is an isomorphism
\(
E_{p,p}^2
\cong
\CH_p(\Sigma)
\)
by \cite[Proposition 2.1]{MR1415592}.
Together with the assumption on $H_*^{\BM}(\Sigma)$,
we have the vanishing $E_{p,q}^2=0$ whenever $p\neq q$.
From this,
we obtain the desired resolution.
\end{proof}

\begin{rmk}
For every fan $\Sigma$ and integer $p\geq 0$,
the sequence
\[
Z_{p,1}(\Sigma) \xrightarrow{d} Z_{p,0}(\Sigma) \to \CH_p(\Sigma)\to 0
\]
is exact by \cite[Proposition 2.1]{MR1415592}.
\end{rmk}
\begin{df}
Let $0\leq r\leq n$ and $p,q\geq 0$ be integers.
For an $r$-standard subdivision $\Sigma$ of $(\P^1)^n$,
we set
\[
Z_{p,q}^\flat(\Sigma)
:=
\ker(Z_{p,q}(\Sigma)\to Z_{p,q}(\Sigma^\circ)).
\]
\end{df}

\begin{prop}
\label{ordering.16}
Let $\Sigma$ be an $r$-standard subdivision of $(\P^1)^n$.
Assume that $\Sigma_{\max}$ admits an admissible ordering.
Then there is a natural resolution of $\CH_p^\flat(\Sigma)$:
\[
\cdots \xrightarrow{d} Z_{p,1}^\flat(\Sigma) \xrightarrow{d} Z_{p,0}^\flat(\Sigma)
\to \CH_p^\flat(\Sigma)\to 0.
\]
\end{prop}
\begin{proof}
Borel-Moore homology and Chow groups are contravariant with respect to open immersions.
Hence we have a commutative diagram
\[
\begin{tikzcd}
\cdots \ar[r,"d"]&
Z_{p,1}(\Sigma)\ar[d]\ar[r,"d"]&
Z_{p,0}(\Sigma)\ar[d]\ar[r]&
\CH_p(\Sigma)\ar[d]\ar[r]&
0
\\
\cdots \ar[r,"d"]&
Z_{p,1}(\Sigma^\circ)\ar[r,"d"]&
Z_{p,0}(\Sigma^\circ)\ar[r]&
\CH_p(\Sigma^\circ)\ar[r]&
0
\end{tikzcd}
\]
whose rows are exact by Propositions \ref{ordering.17}(1),(2) and \ref{moving.1}.
Furthermore, the morphism $\CH_p(\Sigma)\to \CH_p(\Sigma^\circ)$ is surjective since the induced morphism of $\C$-realizations $\Sigma_{\C}^\circ \to \Sigma_{\C}$ is an open immersion of schemes,
and the morphism $Z_{p,i}(\Sigma)\to Z_{p,i}(\Sigma^\circ)$ is surjective for every integer $i\geq 0$ using \eqref{moving.1.2}.
Take kernels along the columns to conclude.
\end{proof}

\begin{const}
\label{subdivision.2}
Let $1\leq i\leq n$ be integers,
and let $\Sigma$ be a smooth complete fan in $\Z^n$.
We have the $(n-1)$-dimensional fan
\[
D_{i,0}(\Sigma):=
V(e_i)
\]
under the assumption that $e_i$ is a ray of $\Sigma$.
We have the $(n-1)$-dimensional fan
\[
D_{i,1}(\Sigma):=\{\sigma\in \Sigma:\sigma \subset \Z^{i-1}\times \{0\} \times \Z^{n-i}\}
\]
in $\Z^{i-1}\times \{0\} \times \Z^{n-i}$.
We say that $\Sigma$ is \emph{$i$-admissible} if $D_{i,1}(\Sigma)$ is complete.
Observe that $\Sigma$ is $i$-admissible for every $1\leq i\leq n$ if $\Sigma$ is a subdivision of $(\P^1)^n$.

We have the induced closed immersions of toric varieties
\[
\delta_{i,0}\colon D_{i,0}(\Sigma)_{\C}\to \Sigma_\C,
\text{ }
\delta_{i,1}\colon D_{i,1}(\Sigma)_\C\to \Sigma_\C,
\]
where $\delta_{i,0}$ is due to \cite[p.\ 53]{Fulton:1436535}.
These induce the maps
\[
\delta_{i,0}^*\colon \CH_*(\Sigma)\to \CH_{*-1}(D_{i,0}(\Sigma)),
\text{ }
\delta_{i,1}^*\colon \CH_*(\Sigma)\to \CH_{*-1}(D_{i,1}(\Sigma)).
\]

For example, $\delta_{i,0}$ (resp.\ $\delta_{i,1}$) exhibits $D_{i,0}((\P^1)^n)_{\C}$ (resp.\ $D_{i,1}((\P^1)^n)_{\C}$) as the closed subscheme $(\P_{\C}^1)^{i-1} \times \{0\} \times (\P_{\C}^1)^{n-i}$ (resp.\ $(\P_{\C}^1)^{i-1} \times \{1\} \times (\P_{\C}^1)^{n-i}$) of $(\P_\C^1)^n$.
This explains the motivation for the choice of the subscript $\epsilon\in \{0,1\}$.
\end{const}

\begin{exm}
\label{subdivision.37}
Let $1\leq i\leq n$ be integers,
let $\Sigma$ be an $i$-admissible smooth complete fan in $\Z^n$,
let $\eta$ be a cone of $\Sigma$,
and let $A$ be a subfan of $\Sigma$ satisfying the condition in Definition \ref{dist.12}.
Then we have
\[
D_{i,1}(\Sigma_\eta^\bary(A))
\cong
D_{i,1}(\Sigma)_{\eta_i}^\bary(A_i),
\]
where $\eta_i:=\eta\cap (\Z^{i-1}\times 0 \times \Z^{n-i})$,
and $A_i:=\{\tau\in A:\tau\subset \Z^{i-1}\times 0 \times \Z^{n-i}\}$.
Note that $A_i$ satisfies the condition in Definition \ref{dist.12} too.
\end{exm}

The following result explains the reason why we need $\CH_*^\flat(\Sigma)$:

\begin{prop}
\label{ordering.3}
Let $\Sigma$ be an $r$-standard subdivision of $(\P^1)^n$.
Assume that $\Sigma_{\max}$ admits an admissible ordering.
Then for every integer $p\geq 0$,
there is a natural isomorphism
\[
\CH_p^\flat(\Sigma)
\cong
\bigcap_{i=r+1}^{n} \ker(\CH_p(\Sigma)\xrightarrow{\delta_{i,0}^*} \CH_{p-1}(D_{i,0}(\Sigma))).
\]
\end{prop}
\begin{proof}
We have the one-to-one correspondence
$\ol{(-)}
\colon
\{\tau\in \Sigma: e_i\in \tau\}
\xrightarrow{\cong}
D_{i,0}(\Sigma)$ 
for every integer $r+1\leq i\leq n$,
see \eqref{ordering.5.1}.
We impose an ordering on $\delta_{i,0}^* \Sigma$ as follows:
For $\sigma,\tau \in \Sigma_{\max}$,
we assign $\ol{\sigma}<\ol{\tau}$ if $\sigma<\tau$.

If a cone $\sigma\in \Sigma_{\max}$ contains $e_i$,
then we can express $\sigma$ as $\Cone( e_i,f_1,\ldots,f_{n-1})$.
We may assume
$\widehat{\sigma}_{f_{t+1}},\ldots,\widehat{\sigma}_{f_{n-1}}<\sigma<\widehat{\sigma}_{f_1},\ldots,\widehat{\sigma}_{f_t}$ for some $t$ without loss of generality,
see Definition \ref{ordering.1} for $\widehat{\sigma}_{f_j}$.

The condition (iv) in Definition \ref{ordering.1} implies $\sigma<\sigma_{\widehat{e_i}}$.
Hence we have
\[
\ess(\sigma)
= 
\sigma
\cap
\widehat{\sigma}_{e_i}
\cap
\widehat{\sigma}_{f_1}
\cap
\cdots
\cap
\widehat{\sigma}_{f_t}
=
\Cone( f_{t+1},\ldots,f_{n-1}).
\]
Since $\ol{\sigma}=\Cone( f_1,\ldots,f_{n-1})$ and $\widehat{\ol{\sigma}}_{f_{t+1}},\ldots,\widehat{\ol{\sigma}}_{f_{n-1}}<\sigma<\widehat{\ol{\sigma}}_{f_1},\ldots,\widehat{\ol{\sigma}}_{f_t}$,
we also have
\[
\ess(\ol{\sigma})
=
\Cone( f_{t+1},\ldots,f_{n-1}).
\]
Hence the condition (i) in Definition \ref{ordering.1} for $\Sigma$ implies the condition (i) in Definition \ref{ordering.1} for $D_{i,0}(\Sigma)$.
Furthermore,
we have
\[
\delta_{i,0}^*[V(\ess(\sigma))]
=
[V(\ess(\ol{\sigma}))]
\]
for $\sigma\in \Sigma_{\max}$ such that $e_i\in \sigma$.
On the other hand,
the condition (v) in Definition \ref{ordering.1} implies
\[
\delta_{i,0}^*[V(\ess(\sigma))]
=
0
\]
for $\sigma\in \Sigma_{\max}$ such that $e_i\notin \sigma$.

Together with Proposition \ref{ordering.19},
we see that $\delta_{i,0}^*\colon \CH_*(\Sigma)\to \CH_{*-1}(D_{i,0}(\Sigma))$ is identified with the projection
\[
\bigoplus_{\sigma\in \Sigma_{\max}} \Z
\to
\bigoplus_{e_i\in \sigma\in \Sigma_{\max}} \Z.
\]
It follows that we have an isomorphism
\[
\bigcap_{i=r+1}^{n} \ker(\CH_*(\Sigma)\xrightarrow{\delta_{i,0}^*} \CH_{*-1}(D_{i,0}(\Sigma)))
\cong
\bigoplus_{e_{i+1},\ldots,e_n\notin \sigma \in \Sigma_{\max}} \Z.
\]
This can be identified with $\CH_*^\flat(\Sigma)$ using Proposition \ref{ordering.17}(5).
\end{proof}

\subsection{Cubical identities for \texorpdfstring{$Z_{p,q}^\flat(\Theta_{n,r,\bd})$}{Zpq}}
\label{identities}

Our next purpose is to construct some maps $\delta_{i,1}^*$, $\rho_i^*$, and $\nu_i^*$ for $Z_{p,q}^\flat(\Theta_{n,r,\bd})$,
which are crucial ingredients for the proof of Proposition \ref{ordering.14} below.
The reason of our detour through $Z_{p,q}^\flat(\Theta_{n,r,\bd})$ is that $\CH_p^\flat(\Theta_{n,r,\bd})$ does not enjoy $\rho_i^*$ and $\nu_i^*$.

\begin{const}
\label{ordering.7}
Let $\Sigma$ be an $r$-standard subdivision of $(\P^1)^n$ with integers $0\leq r\leq n$.
For $\sigma\subset \Z^{i-1}\times 0 \times \Z^{n-i}$ with an integer $r+1\leq i\leq n$,
let $\delta_{i,1}^*\sigma$ be the corresponding cone in $D_{i,1}(\Sigma)$.
The $i$th inclusion $\delta_{i,1}\colon N(D_{i,1}(\Sigma))\cong \Z^{n-1}\to N(\Sigma)\cong \Z^n$ sends $\delta_{i,1}^*\sigma$ onto $\sigma$.
Hence we have the induced map $\delta_{i,1}^*\colon M(\sigma)\to M(\delta_{i,1}^*\sigma)$.
We have the map
\[
\delta_{i,1}^*
\colon
Z_{p,q}^\flat(\Sigma)
\to
Z_{p-1,q}^\flat(D_{i,1}(\Sigma))
\]
given by
\[
\delta_{i,1}^*(\alpha[V(\sigma)])
:=
\left\{
\begin{array}{ll}
\delta_{i,1}^*\alpha [V(\delta_{i,1}^*\sigma)] & \text{if }\sigma \subset \Z^{i-1}\times 0 \times \Z^{n-i},
\\
0 & \text{otherwise},
\end{array}
\right.
\]
for $\sigma\in \Sigma^\flat:=\Sigma-\Sigma^\circ$ and $\alpha\in \wedge^q M(\sigma)$.
Note that $\delta_{i,1}^*$ is compatible with the differentials $d$ since
we have
\[
\delta_{i,1}^*(u_{\tau,\sigma}\iprod \alpha [V(\tau)])
=
\left\{
\begin{array}{ll}
u_{\tau,\sigma} \iprod \delta_{i,1}^* \alpha [V(\tau)]& \text{if }\tau \subset \Z^{i-1}\times 0 \times \Z^{n-i},
\\
0 & \text{otherwise}
\end{array}
\right.
\]
for $\sigma\prec_1 \tau$ and $\alpha\in \wedge^q M(\sigma)$.
\end{const}

\begin{const}
Let $0\leq r\leq n$ be integers, and let $\bd$ be a finite sequence of positive integers.
If $r<n$,
then we have
\[
D_{i,1}(\Theta_{n,r,\bd})\cong \Theta_{n-1,r,\bd}
\]
using Example \ref{subdivision.37}.
For integers $p$ and $q$,
we set
\[
\delta^*:=\sum_{i=r+1}^n (-1)^{i-r}\delta_{i,1}^*\colon Z_{p,q}^\flat (\Theta_{n,r,\bd})
\to
Z_{p-1,q}^\flat(\Theta_{n-1,r,\bd}).
\]
If $r=n$, then we simply set $\delta^*=0$.
We have the induced chain complex
\[
\cdots
\xrightarrow{\delta^*}
Z_{p+1,q}^\flat(\Theta_{r+1,r,\bd})
\xrightarrow{\delta^*}
Z_{p,q}^\flat(\Theta_{r,r,\bd}) \to 0.
\]
\end{const}

\begin{lem}
\label{ordering.12}
Let $0\leq r\leq n$ be integers,
and let  $a=(a_1,\ldots,a_n)$ be a ray of a subdivision $\Sigma$ of $\Gamma_{n,r}$ such that $a_1,\ldots,a_r\leq 0$ and $a\neq e_1,\ldots,e_n$.
If $a_s=0$ and $r+1\leq s\leq n$,
then $\Cone(a,e_s)\notin \Sigma$.
\end{lem}
\begin{proof}
Without loss of generality,
we may assume $s=n$.
Since $a\neq e_1,\ldots,e_n$,
there exists an integer $r+1\leq t\leq n$ such that $a_t<0$.
We may also assume $t=n-1$.
Then by considering the hyperplane $\{(x_1,\ldots,x_n)\in \Z^n:x_{n-1}+x_n=0\}$,
we see that there exists an $(n-1)$-dimensional cone $\tau$ of $\Gamma_{n,r}$ such that $\tau\cap \Cone(a,e_s)$ is equal to the ray $(a_1,\ldots,a_{n-2},a_{n-1},-a_{n-1})$.
This implies $\Cone(a,e_s)\notin \Sigma$.
\end{proof}

\begin{const}
\label{ordering.8}
Let $\Sigma$ be an $r$-standard subdivision of $(\P^1)^n$,
where $0\leq r\leq n$ be integers.
Assume that $\Sigma$ is a subdivision of $\Gamma_{n,r}$,
e.g., $\Theta_{n,r,\bd}$ for a nonempty finite sequence $\bd$ of positive integers.
For an integer $r+1\leq i\leq n$,
consider the $i$th projection $\rho_i\colon N(\Sigma)\cong \Z^n\to N(D_{i,1}(\Sigma))\cong \Z^{n-1}$.
For $\sigma\in D_{i,1}(\Sigma)^\flat$ such that $\dim \sigma >0$,
let $\rho_i^* \sigma$ is the unique cone of $\Sigma$ contained in $\Z^{i-1}\times 0\times \Z^{n-i}$ such that $\rho_i$ sends $\rho_i^*\sigma$ onto $\sigma$.
Then we have the map $\rho_i^*\colon M(\sigma)\to M(\rho_i^*\sigma)$ induced by the dual map $\rho_i^*\colon M(D_{i,1}(\Sigma))\to M(\Sigma)$.
We claim $\rho_i^*\sigma \in \Sigma^\flat$.
Since $\sigma\in D_{i,1}(\Sigma)^\flat$,
we only need to show that there exists no cone of $\Sigma$ containing $\rho_i^*\sigma$ and $e_i$.
This follows from Lemma \ref{ordering.12} since the assumption $\sigma \in D_{i,1}(\Sigma)$ implies that every ray of $\sigma$ satisfies the assumption in  Lemma \ref{ordering.12}.

For all integers $0\leq p\leq n-1$ (the condition $\dim \sigma>0$ corresponds to $p\leq n-1$) and $q\geq 0$,
we have the map
\[
\rho_i^*
\colon
Z_{p,q}^\flat (D_{i,1}(\Sigma))
\to
Z_{p+1,q}^\flat (\Sigma)
\]
given by
\[
\rho_i^*(\alpha[V(\sigma)]):=\rho_i^*\alpha [V(\rho_i^*\sigma)]
\]
for $\sigma\in D_{i,1}(\Sigma)^\flat$ and $\alpha\in \wedge^q M(\sigma)$.
\end{const}

\begin{lem}
\label{ordering.9}
Let $0\leq r\leq n$ be integers,
and let $\bd$ be a nonempty finite sequence of positive integers.
Consider the maps $\delta_{i,1}^*$ and $\rho_i^*$ for $Z_{p,q}^\flat(\Theta_{n,r,\bd})$ obtained by \textup{Constructions \ref{ordering.7}} and \textup{\ref{ordering.8}}.
If $0\leq p\leq n-1$,
we have
\[
\delta_{i,1}^*\rho_j^*
=
\left\{
\begin{array}{ll}
\rho_j^*\delta_{i-1,1}^* & \text{if $j<i$},
\\
\id & \text{if $j=i$},
\\
\rho_{j-1}^*\delta_{i,1}^* & \text{if $j>i$}.
\end{array}
\right.
\]
\end{lem}
\begin{proof}
A cone $\sigma\in \Theta_{n,r,\bd}^\flat$ can be written as $\Cone( f_1,\ldots,f_m)$,
and we also write $f_q$ as $(f_{q,1},\ldots,f_{q,n})\in \Z^n$.
We have $\rho_i^* \sigma = \Cone( f_1',\ldots,f_m')$, where
\[
f_q':=(f_{q,1},\ldots,f_{q,i-1},0,f_{q,i},\ldots,f_{q,n})
\]
for $1\leq q\leq m$.
We have $\delta_{i,1}^*\sigma = \Cone( f_1'',\ldots,f_m'')$,
where
\[
f_q'':=(f_{q,1},\ldots,f_{q,i-1},f_{q,i+1},\ldots,f_{q,n+1}).
\]
Using these two descriptions,
we can show
\[
\delta_{i,1}^*\rho_j^*{[V(\sigma)]}
=
\left\{
\begin{array}{ll}
\rho_j^*\delta_{i-1,1}^* {[V(\sigma)]}& \text{if $j<i$},
\\
{[V(\sigma)]}& \text{if $j=i$},
\\
\rho_{j-1}^*\delta_{i,1}^* {[V(\sigma)]}& \text{if $j>i$}.
\end{array}
\right.
\]
The maps $\delta_{i,1}$ and $\rho_i$ for $N(\Theta_{n,r,\bd})$ obtained by Constructions \ref{ordering.7} and \ref{ordering.8} satisfy
\[
\rho_j\delta_{i,1}=
\left\{
\begin{array}{ll}
\delta_{i-1,1}\rho_j & \text{if $j<i$},
\\
\id & \text{if $j=i$},
\\
\delta_{i,1}\rho_{j-1} & \text{if $j>i$},
\end{array}
\right.
\]
which yields a similar identity for $\delta_{i,1}^*\rho_j^*\alpha$, where $\alpha\in \wedge^q M(\sigma)$.
Combine what we have discussed above to obtain the desired identity.
\end{proof}

\begin{const}
\label{ordering.10}
Let $0\leq r\leq n$ be integers,
and let $\bd$ be a nonempty finite sequence of positive integers.
Consider the map $\nu_i\colon N(\Theta_{n+1,r,\bd})\cong \Z^{n+1} \to N(\Theta_{n,r,\bd})\cong \Z^n$ given by
\[
(x_1,\ldots,x_{n+1})\in \Z^{n+1} \mapsto (x_1,\ldots,x_{i-1},x_i+x_{i+1},x_{i+2},\ldots,x_{n+1}).
\]
For $\sigma \in \Theta_{n,r,\bd}^\flat$ such that $\sigma\subset \Z^{i-1}\times 0 \times \Z^{n-i}$,
we have $\rho_i^* \sigma \subset \Z^{i-1}\times 0\times 0 \times \Z^{n-i}$.
Hence we have $\nu_i(\rho_i^*\sigma)\subset \sigma$,
so we have the induced map $\nu_i^* \colon M(\sigma)\to M(\rho_i^*\sigma)$.
For all integers $0\leq p\leq n-1$ and $q\geq 0$,
we have the map
\[
\nu_i^*
\colon
Z_{p,q}^\flat(\Theta_{n,r,\bd})
\to
Z_{p+1,q}^\flat(\Theta_{n+1,r,\bd})
\]
given by
\[
\nu_i^* (\alpha[V(\sigma)])
:=
\left\{
\begin{array}{ll}
\nu_i^*\alpha[V(\rho_i^*\sigma)] & \text{if $\sigma\subset \Z^{i-1}\times 0 \times \Z^{n-i}$},
\\
\rho_i^*(\alpha[V(\sigma)])+\rho_{i+1}^*(\alpha[V(\sigma)])& \text{otherwise}
\end{array}
\right.
\]
for $\sigma \in \Theta_{n,r,\bd}^\flat$ and $\alpha\in \wedge^q M(\sigma)$.
\end{const}

\begin{lem}
\label{ordering.11}
Let $0\leq r\leq n$ be integers,
and let $\bd$ be a nonempty finite sequence of positive integers.
Consider the maps $\delta_{i,1}^*$ and $\nu_i^*$ for $Z_{p,q}^\flat(\Theta_{n,r,\bd})$ in \textup{Constructions \ref{ordering.7}} and \textup{\ref{ordering.10}}.
We have
\[
\delta_{i,1}^*\nu_j^*
=
\left\{
\begin{array}{ll}
\nu_j^*\delta_{i-1,1}^* & \text{if $j<i-1$},
\\
\id & \text{if $j=i-1,i$},
\\
\nu_{j-1}^*\delta_{i,1}^* & \text{if $j>i$}.
\end{array}
\right.
\]
\end{lem}
\begin{proof}
Case 1:
Assume $\sigma\subset \Z^{j-1}\times 0 \times \Z^{n-j}$.
Then we can write $\sigma$ as $\Cone( f_1,\ldots,f_m)$ with
\[
f_q:=(f_{q,1},\ldots,f_{q,j-1},0,f_{q,j+1},\ldots,f_{q,n})
\]
for $1\leq q\leq m$.
We have $\rho_j^*\sigma =\Cone( f_1',\ldots,f_m')$, where
\[
f_q':=(f_{q,1},\ldots,f_{q,j-1},0,0,f_{q,j+1},\ldots,f_{q,n}).
\]
We can also write $\delta_{j,1}^* \rho_j^*\sigma$ (resp.\ $\delta_{j+1,1}^*\rho_j^*\sigma$) as $\Cone( f_1'',\ldots,f_m'')$,
where $f_q''$ is obtained from $f_q'$ by removing $0$ in the $j$th (resp.\ $(j+1)$th) component. Then we have $f_q''=f_q$.
Hence we have $\delta_{j,1}^* \rho_j^*\sigma=\delta_{j+1,1}^* \rho_j^*\sigma=\sigma$.
On the other hand,
we have
$\delta_{i,1}^*\rho_j^*\sigma=\rho_j^* \delta_{i-1,1}^*\sigma$ and $\delta_{i-1,1}^*\sigma \subset \Z^{j-1}\times 0 \times \Z^{n-j}$ if $j<i-1$,
and we have $\delta_{i,1}^*\rho_j^*\sigma=\rho_{j-1}^* \delta_{i,1}^*\sigma$ and $\delta_{i,1}^*\sigma \subset \Z^{j-2}\times 0 \times \Z^{n-j+1}$ if $j>i$,
where we need Lemma \ref{ordering.9} for the two identities.
Combine what we have shown above to have
\[
\delta_{i,1}^*\nu_j^*[V(\sigma)]
=
\left\{
\begin{array}{ll}
\nu_j^*\delta_{i-1,1}^* {[V(\sigma)]}& \text{if $j<i-1$},
\\
{[V(\sigma)]} & \text{if $j=i-1,i$},
\\
\nu_{j-1}^*\delta_{i,1}^* {[V(\sigma)]}& \text{if $j>i$}.
\end{array}
\right.
\]
The maps $\delta_{i,1}$ and $\nu_i$ for $N(\Theta_{n,r,\bd})$ in Constructions \ref{ordering.7} and \ref{ordering.10} satisfy
\[
\nu_j\delta_{i,1}=
\left\{
\begin{array}{ll}
\delta_{i-1,1}\nu_j & \text{if $j<i-1$},
\\
\id & \text{if $j=i-1,i$},
\\
\delta_{i,1}\nu_{j-1} & \text{if $j>i$},
\end{array}
\right.
\]
which yields a similar identity for $\delta_{i,1}^*\nu_j^*\alpha$, where $\alpha\in \wedge^q M(\sigma)$.
Combine what we have discussed above to obtain the desired identity.

Case 2:
Assume $\sigma\not\subset \Z^{j-1}\times 0 \times \Z^{n-j}$.
By Lemma \ref{ordering.9}, we have
\begin{align*}
\delta_{i,1}^*\nu_j^*(\alpha [V(\sigma)])
= &
\delta_{i,1}^*\rho_j^*(\alpha [V(\sigma)])+\delta_{i,1}^*\rho_{j+1}^*(\alpha [V(\sigma)])
\\
= &
\left\{
\begin{array}{ll}
\rho_j^*\delta_{i-1,1}^*(\alpha [V(\sigma)])+\rho_{j+1}^*\delta_{i-1,1}^*(\alpha [V(\sigma)])&
\text{if $j<i-1$},
\\
\rho_j^*\delta_{j,1}^* (\alpha[V(\sigma)])+\alpha[V(\sigma)] & \text{if $j=i-1$},
\\
\alpha[V(\sigma)]+\rho_{j+1}^*\delta_{j,1}^* (\alpha[V(\sigma)]) & \text{if $j=i$},
\\
\rho_{j-1}^*\delta_{i,1}^*(\alpha [V(\sigma)])+\rho_j^*\delta_{i,1}^*(\alpha [V(\sigma)])&
\text{if $j>i$}.
\end{array}
\right.
\end{align*}
To conclude, use the following three claims: If $j<i-1$, then we have $\delta_{i-1,1}^* \sigma \not\subset \Z^{j-1}\times 0 \times \Z^{n-j}$.
We have $\delta_{j,1}^*\sigma=0$.
If $j>i$, then we have $\delta_{i,1}^*\sigma \not\subset  \Z^{j-2}\times 0 \times \Z^{n-j+1}$.
\end{proof}

With $\delta_{i,1}^*$, $\rho_i^*$, and $\rho_i^*$ for $Z_{p,q}^\flat(\Theta_{n,r,\bd})$ in hand, the general theory of cubical abelian groups imply the following result.

\begin{prop}
\label{ordering.14}
Let $r\geq 0$ be an integer,
and let $\bd$ be a nonempty finite sequence of positive integers.
Then the chain complex
\begin{equation}
\label{ordering.14.1}
\cdots
\xrightarrow{\delta^*}
Z_{p+1,q}^\flat(\Theta_{r+1,r,\bd})
\xrightarrow{\delta^*}
Z_{p,q}^\flat(\Theta_{r,r,\bd})
\end{equation}
is quasi-isomorphic to $0$ for every integer $0\leq p\leq r-1$.
\end{prop}
\begin{proof}
Using the identities for $\delta_{i,1}^*$ and $\nu_i^*$ in \cref{ordering.11},
we can argue as in \cite[Proof of Proposition A.11]{MR3259031} to show that the above chain complex is quasi-isomorphic to the chain complex $N_\bullet$ given by
\[
N_m:=\bigcap_{i=1}^{m-1} \ker(\delta_{r+i,1}^*\colon Z_{p+m,q}^\flat (\Theta_{r+m,r,\bd})
\to
Z_{p+m-1,q}^\flat (\Theta_{r+m-1,r,\bd}))
\]
with the differential $\delta_{r+m,1}^*\colon N_m\to N_{m-1}$.
For $x\in N_m$ such that $\delta_{r+m,1}^*x=0$,
consider $\rho_{r+m+1}^*x \in Z_{p+m+1,q}^\flat(\Theta_{r+m +1,r,\bd})$.
By Lemma \ref{ordering.9},
we have $\delta_{r+i,1}^*\rho_{r+m+1}^*x=\rho_{r+m}^*\delta_{r+i,1}^*x=0$ for $1\leq i\leq m$ and $\delta_{r+m+1,1}^*\rho_{r+m+1}^*x=x$.
It follows that $N_\bullet$ is quasi-isomorphic to $0$.
\end{proof}

\begin{rmk}
\label{ordering.27}
Consider the maps $\nu_i^*$ for $Z_{p,q}^\flat(\Theta_{n,r,\bd})$ obtained by Construction \ref{ordering.10}.
As in Lemma \ref{ordering.11},
one can show
\begin{gather}
\label{ordering.27.1}
\nu_i^*\nu_j^*
=
\nu_{j+1}^*\nu_i^*
\text{ if }
j>i,
\\
\label{ordering.27.2}
\nu_i^*\rho_j^*
=
\left\{
\begin{array}{ll}
\rho_{j}^*\nu_{i-1}^* & \text{if }j<i,
\\
\rho_{i+1}^*\rho_i^* & \text{if }j=i,
\\
\rho_{j-1}^*\nu_i^* & \text{if }j>i.
\end{array}
\right.
\end{gather}
We omit the proof since we do not need these identities later.
The referee suggested the following conceptual proof of Proposition \ref{ordering.14} using these identities.

By Lemma \ref{ordering.11} and \eqref{ordering.27.1},
the groups
\[
C_m:=Z_{p+m+1,q}^\flat(\Theta_{r+m+1,r,\bd}),
\text{ }
m\geq {-1}
\]
form an augmented simplicial abelian group $C_\bullet$.
The chain complex associated with $C_\bullet$,
which is \eqref{ordering.14.1} (up to a degree shift),
is quasi-isomorphic to the normalized chain complex associated with $C_\bullet$,
which is $N_{\bullet}$ (up to a degree shift).
Furthermore,
by Lemma \ref{ordering.9} and \eqref{ordering.27.2},
the maps
\[
\rho_{r+m+1}^*\colon C_{m-1}\to C_m,
\text{ }
m\geq 0
\]
form an extra degeneracy of $C_\bullet$.
Hence we conclude the proof.
This is why $\rho_{r+m+1}^*$ alone is enough in the proof among all the $\rho_i^*$'s in the same degree.
\end{rmk}

\begin{prop}
\label{ordering.15}
Let $r\geq 0$ be an integer,
and let $\bd$ be a nonempty finite sequence of positive integers.
Then the chain complex
\[
\cdots
\xrightarrow{\delta^*}
\CH_{p+1}^\flat(\Theta_{r+1,r,\bd})
\xrightarrow{\delta^*}
\CH_p^\flat(\Theta_{r,r,\bd})
\]
is quasi-isomorphic to $0$ for every integer $0\leq p\leq r-1$.
\end{prop}
\begin{proof}
Consider the double complex $A_{i,j}:=Z_{p+i,j}^\flat(\Theta_{r+i,r,\bd})$,
which can be described as
\[
\begin{tikzcd}[column sep=small, row sep=small]
&
\vdots\ar[d,"d"]&
\vdots\ar[d,"d"]
\\
\cdots\ar[r,"\delta^*"]&
Z_{p+1,1}^\flat(\Theta_{r+1,r,\bd})\ar[d,"d"]\ar[r,"\delta^*"]&
Z_{p,1}^\flat(\Theta_{r,r,\bd})\ar[d,"d"]
\\
\cdots \ar[r,"\delta^*"]&
Z_{p+1,0}^\flat(\Theta_{r+1,r,\bd})\ar[r,"\delta^*"]&
Z_{p,0}^\flat(\Theta_{r,r,\bd}).
\end{tikzcd}
\]
Proposition \ref{ordering.16} (resp.\ \ref{ordering.14}) implies that the chain complex $A_{i,*}$ with fixed $i$ (resp.\ $A_{*,j}$ with fixed $j$) is quasi-isomorphic to $\CH_{p+i}^\flat(\Theta_{r+i,r,\bd})$ (resp.\ $0$).
Compare the two spectral sequences associated with the double complex $A_{*,*}$ to finish the proof.
\end{proof}

\section{Admissible subdivisions}
\label{admissible}

In this section,
we are interested in a fan $\Sigma$ in $\Z^n$ equipped with maps $\delta_{i,0}^*$ for every $i\in I_0$ and $\delta_{i,1}^*$ for every $i\in I_1$ for some index sets $I_0$ and $I_1$ that resemble $\delta_{i,0}^*$ and $\delta_{i,1}^*$ in the previous sections.
For the precise condition of $\Sigma$, see Definition \ref{admissible.1}.
If a class $x\in \CH^*(\Sigma)$ satisfies $\delta_{i,0}^* x=0$ for every $i\in I_0$ and $\delta_{i,1}^*x=0$ for every $i\in I$,
then we will show in Lemmas \ref{subdivision.21} and \ref{subdivision.22} the existence of $y\in \CH^*(\Sigma')$ for some nice subdivision $\Sigma'$ of $\Sigma\times \P^1$ such that $\delta_{i,0}^*y=0$ for $i\in I_0\amalg \{n+1\}$, $\delta_{i,1}^*y=0$ for $i\in I_1$, and $\delta_{n+1,1}^*y=x$.

This result is a crucial ingredient of the proof of Lemma \ref{subdivision.11} below,
which is the induction step in the proof of Theorem \ref{intro.2}.

\subsection{Definition and basic properties of admissible subdivisions}

We begin with the following technical definition.

\begin{df}
\label{admissible.1}
Let $B$ be a finite set with subsets $I_0$, $I_1$, and $I_2$ such that $I_0\cup I_1$ is disjoint with $I_2$.
We say that a fan $\Sigma$ in $\Z^B$ is \emph{$(I_0,I_1,I_2)$-admissible} if the following conditions are satisfied:
\begin{enumerate}
\item[(i)] $\Sigma$ is complete and smooth.
\item[(ii)] If $I_0=\{a_1,\ldots,a_s\}$,
then $\Cone( e_{a_1},\ldots,e_{a_s})\in \Sigma$.
\item[(iii)] The fan $D_{i,1}(\Sigma)$ is complete for $i\in I_1$.
\item[(iv)]
Let $i\in I_2$.
Then $e_i$ is a unique ray of $\Sigma$ whose $i$th coordinate is $>0$.
\end{enumerate}
An \emph{$(I_0,I_1,I_2)$-admissible subdivision of $\Sigma$} is a subdivision $\Sigma'$ of $\Sigma$ such that $\Sigma'$ is an $(I_0,I_1,I_2)$-admissible fan in $\Z^B$.
\end{df}

\begin{exm}
Let $0\leq r\leq n$ be integers.
An $r$-standard subdivision of $(\P^1)^n$ is $(I_0,I_1,I_2)$-admissible with $I_0=I_1=\{r+1,\ldots,n\}$ and $I_2=\{1,\ldots,r\}$.

Let $\Sigma$ be a very $r$-standard subdivision of $(\P^1)^n$,
and let $a:=(a_1,\ldots,a_n)$ be a ray of $\Sigma$ such that $a\neq e_1,\ldots,e_n$ and $a_1,\ldots,a_r\leq 0$.
Then $V(a)$ does not need to be an $r$-standard subdivision of $(\P^1)^{n-1}$.
Rather, we will show in Construction \ref{subdivision.27} that $V(a)$ is an $(I_0,I_1,I_2)$-subdivision for some suitable $I_0$, $I_1$, and $I_2$.
This is the reason why we need $(I_0,I_1,I_2)$-admissible subdivisions, which are more flexible than $r$-standard subdivisions.
\end{exm}

Next, we discuss two basic results on refinements of admissible fans.

\begin{lem}
\label{admissible.2}
Let $\Sigma_1$ and $\Sigma_2$ be $(I_0,I_1,I_2)$-admissible fans in $\Z^B$,
where $B$ is a finite set, $I_0,I_1,I_2\subset B$, and $(I_0\cup I_1)\cap I_2=\emptyset$.
Then there exists a fan $\Delta$ in $\Z^B$ that is an $(I_0,I_1,I_2)$-admissible subdivision of both $\Sigma_1$ and $\Sigma_2$.
\end{lem}
\begin{proof}
Consider the fan $\Sigma_{12}:=\{\sigma_1\cap \sigma_2:\sigma_1\in \Sigma_1,\sigma_2\in \Sigma_2\}$, which satisfies the conditions (ii)--(iv) in Definition \ref{admissible.1}.
Let $\Sigma_{12}'$ be the largest subfan of $\Sigma_{12}$ not containing $e_i$ for all $i\in I_2$.
By toric resolution of singularities \cite[Theorem 11.1.9]{CLStoric},
$\Sigma_{12}'$ admits a sequence of star subdivisions
\[
\Delta_m'\to \cdots \to \Delta_0':=\Sigma_{12}'
\]
such that $\Delta_m'$ is smooth and for every smooth cone $\sigma$ of $\Sigma_{12}'$,
$\sigma$ is a cone of $\Delta_m'$.

Let $\Delta_m\to \cdots \to \Delta_0:=\Sigma_{12}$ be the corresponding sequence of star subdivisions.
Then $\Delta_m$ satisfies the conditions (i)--(iv) in \cref{admissible.1}.
\end{proof}

\begin{lem}
\label{subdivision.25}
Let $\Sigma$ be an $(I_0,I_1,I_2)$-admissible fan in $\Z^B$,
where $B$ is a finite set, $I_0,I_1,I_2\subset B$, and $(I_0\cup I_1)\cap I_2=\emptyset$,
and let $\Sigma'\to \Sigma$ be an $(I_0,I_1,I_2)$-admissible subdivision.
Then there exists a sequence of $(I_0,I_1,I_2)$-subdivisions
\[
\Sigma_m\to \cdots \to \Sigma_0:=\Sigma
\]
and an $(I_0,I_1,I_2)$-admissible subdivision $\Sigma_m\to \Sigma'$ such that for $1\leq j\leq m$, $\Sigma_j$ is the star subdivision of $\Sigma_{j-1}$ relative to a $2$-dimensional cone $\sigma_{j-1}\in \Sigma_{j-1}$ such that $e_i\notin \sigma_{j-1}$ for $i\in I_2$ and $\sigma_{j-1}\not\subset \Cone( e_{a_1},\ldots,e_{a_u})$ if $I_0=\{a_1,\ldots,a_u\}$.
\end{lem}
\begin{proof}
Consider the largest subfans $\Delta$ and $\Delta'$ of $\Sigma$ and $\Sigma'$ not containing $e_i$ for all $i\in I_2$.
By \cite[pp.\ 39-40]{TOda},
there exists a sequence of subdivisions $\Delta_m\to \cdots \to \Delta_0:=\Delta$ and a subdivision $\Delta_m\to \Delta'$ such that for $1\leq j\leq m$, $\Delta_j$ is the star subdivision of $\Delta_{j-1}$ relative to a $2$-dimensional cone $\sigma_{j-1}\in \Delta_{j-1}$ such that $\sigma_{j-1}\not \subset \Cone( e_{a_1},\ldots,e_{a_u})$.
The corresponding star subdivisions $\Sigma_m\to \cdots \to \Sigma_0$ satisfy the desired properties.
\end{proof}

We consider divisors of admissible fans as follows.

\begin{df}
For a fan $\Sigma$,
let $Z^1(\Sigma)$ be the free abelian group generated by $[V(f)]$ for the rays $f$ of $\Sigma$.

Let $\Sigma$ be an $(I_0,I_1,I_2)$-admissible fans in $\Z^B$,
where $B$ is a finite set, $I_0,I_1,I_2\subset B$, and $(I_0\cup I_1)\cap I_2=\emptyset$.
Let $Z_{I_0}^1(\Sigma)$ be the free abelian subgroup of $Z^1(\Sigma)$ generated by $[V(f)]$ for all rays of $\Sigma$ not in $\{e_i: i\in I_0\}$.
\end{df}

The following lemma justifies the definition of $Z_{I_0}^1(\Sigma)$.

\begin{lem}
\label{subdivision.19}
Let $\Sigma$ be an $(I_0,I_1,I_2)$-admissible fans in $\Z^B$,
where $B$ is a finite set, $I_0,I_1,I_2\subset B$, and $(I_0\cup I_1)\cap I_2=\emptyset$.
Then the map $\Sym^p Z_{I_0}^1(\Sigma)\to \CH^p(\Sigma)$ induced by $Z_{I_0}^1(\Sigma)\to \CH^1(\Sigma)$ is surjective for every integer $p\geq 0$.
\end{lem}
\begin{proof}
The claim is clear if $p=0$,
so assume $p\geq 1$.
Let $f_1,\ldots,f_m$ be the rays of $\Sigma$.
By \cite[Theorem 10.8]{MR495499},
there is an isomorphism of graded rings
\begin{equation}
\label{subdivision.19.1}
\CH^*(\Sigma):=\CH^*(\Sigma_\C)
\cong
\Z[x_1,\ldots,x_m]/(I+J),
\end{equation}
where $x_j:=[V(f_j)]$,
$I$ is the ideal generated by the monomials $x_{j_1}\cdots x_{j_k}$ such that $\Cone(f_{j_1},\ldots,f_{j_k})\notin \Sigma$,
and $J$ is the ideal generated by the elements
\(
\sum_{j=1}^m \langle f_j,e_i\rangle x_j
\)
for all $i\in B$.
Using \eqref{subdivision.19.1},
we see that the induced map $\Sym^p \CH^1(\Sigma)\to \CH^p(\Sigma)$ is surjective.
Hence it suffices to show that $Z_{I_0}^1(\Sigma)\to \CH^1(\Sigma)$ is surjective.
If $f_j\neq e_i$ for all $i\in I_0$,
then $x_j=[V(f_j)]$ is in its image.
Hence it suffices to show that $[V(e_i)]$ is in its image for $i\in I_0$.
We have
\[
[V(e_i)]
=
-\sum_{f\in \Sigma(1),f\neq e_i} \langle f,e_i\rangle [V(f)].
\]
To conclude,
observe that $\langle e_k,e_i\rangle =0$ whenever $k\in B$ and $k\neq i$.
\end{proof}

\subsection{Cubical identities for \texorpdfstring{$Z_{I_0}^1(\Sigma)$}{ZI01}}

For a scheme $X$,
let $\PsDiv(X)$ denote the set of pseudo-divisors \cite[Definition 2.2.1]{Fulton} on $X$.
Recall that a pseudo-divisor on $X$ is a triple $(L,Z,s)$, where $L$ is a line bundle on $X$,
$Z$ is a closed subset of $X$,
and $s$ is a nowhere vanishing section of $L$ on $X-Z$.
Note that for every morphism of schemes $Y\to X$,
we have the natural pullback map $f^*\colon \PsDiv(X)\to \PsDiv(Y)$.
This is the reason why we use $\PsDiv$ below for technical convenience since we do not have the pullback map for Cartier divisors in general.

\begin{const}
\label{subdivision.14}
Let $\Sigma$ be a fan.
Then we have a natural injective map
\[
Z^1(\Sigma)
\to
\PsDiv(\Sigma_{\C})
\]
sending $[V(f)]$ to the pseudo-divisor determined by the Cartier divisor $[V(f)_{\C}]$ for every ray $f$ of $\Sigma$.
\end{const}

Let us discuss when the pullback maps for $Z^1$ and $\PsDiv$ are compatible.
First, we treat the case of star subdivisions.

\begin{lem}
\label{subdivision.15}
Let $g\colon \Sigma'\to \Sigma$ be a star subdivision of smooth fans relative to a cone $\sigma\in \Sigma$.
Then there exists a unique map $g^*\colon Z^1(\Sigma)\to Z^1(\Sigma')$ such that the square
\[
\begin{tikzcd}
Z^1(\Sigma)\ar[d,"g^*"']\ar[r]&
\PsDiv(\Sigma_{\C})\ar[d,"g^*"]
\\
Z^1(\Sigma')\ar[r]&
\PsDiv(\Sigma_{\C}')
\end{tikzcd}
\]
commutes.
\end{lem}
\begin{proof}
The uniqueness is clear since the horizontal maps are injective.
For the existence,
we can work locally on $\Sigma$.
Hence we may assume $\Sigma=\A^n$ and $\sigma=\Cone(e_1,\ldots,e_r)$ with $1\leq r\leq n$.
In this case,
$\Sigma_{\C}'$ has a Zariski cover consisting of 
\[
U_i:=\Spec(\C[x_i,x_1/x_i,\ldots,x_r/x_i,x_i,x_{r+1},\ldots,x_n])
\]
for $1\leq i\leq r$.
Consider the divisor $Z_i:=V(e_j)_{\C}$ on $\Sigma_{\C}$ for $1\leq j\leq n$.
Then the divisor $U_i\cap g^*(Z_j)$ on $U_i$ is the sum of smooth divisors $(x_i)+(x_{x_j/x_i})$ if $i\in \{1,\ldots,r\}-\{j\}$ and $(x_j)$ if $i\notin \{1,\ldots,r\}-\{j\}$.
Use this description to show
\[
g^*(Z_j)
=
\left\{
\begin{array}{ll}
V(e_j)_{\C}+V(f)_{\C}&
\text{if $1\leq j\leq r$,}
\\
V(e_j)_{\C}&
\text{otherwise},
\end{array}
\right.
\]
where $f:=e_1+\cdots+e_r$ is the center of $\sigma$.
Hence $g^*(Z_j)$ is in the image of $Z^1(\Sigma')\to \Div(\Sigma_{\C}')$,
which establishes the existence of $g^*\colon Z^1(\Sigma)\to Z^1(\Sigma')$.
\end{proof}

\begin{rmk}
Let $\Sigma$ be an $(I_0,I_1,I_2)$-admissible fans in $\Z^B$,
where $B$ is a finite set, $I_0,I_1,I_2\subset B$, and $(I_0\cup I_1)\cap I_2=\emptyset$.
Then for $i\in I_0$ (resp.\ $i\in I_1$), $D_{i,0}(\Sigma)$ (resp.\ $D_{i,1}(\Sigma)$) is an $(I_0-\{i\},I_1-\{i\},I_2)$-admissible fan in $\Z^{B-\{i\}}$,
see Construction \ref{subdivision.2} for $D_{i,0}(\Sigma)$ and $D_{i,1}(\Sigma)$.

If $\{t\}$ is an indeterminate variable,
then $\Sigma \times \P^{\{t\}}$ is an $(I_0\cup \{t\},I_1\cup \{t\},I_2)$-admissible fan in $\Z^{B\cup \{t\}}$, where $\P^{\{t\}}:=\P^1$
is a projective line,
and the letter $t$ is used as the coordinate variable.
\end{rmk}

Next, we discuss the case of $\delta_{i,\epsilon}^*$.

\begin{lem}
\label{subdivision.16}
Let $\Sigma$ be an $(I_0,I_1,I_2)$-admissible fans in $\Z^B$,
where $B$ is a finite set, $I_0,I_1,I_2\subset B$, and $(I_0\cup I_1)\cap I_2=\emptyset$.
Then for $(i,\epsilon)\in I_0\times \{0\}\amalg I_1\times \{1\}$,
there exists a unique map 
\(
\delta_{i,\epsilon}^*
\colon
Z_{I_\epsilon}^1(\Sigma)
\to
Z_{I_\epsilon-\{i\}}^1(D_{i,\epsilon}(\Sigma))
\)
such that the square
\[
\begin{tikzcd}
Z_{I_0}^1(\Sigma)\ar[d,"\delta_{i,\epsilon}^*"']\ar[r]&
\PsDiv(\Sigma_{\C})\ar[d,"\delta_{i,\epsilon}^*"]
\\
Z_{I_0-\{i\}}^1(D_{i,\epsilon}(\Sigma))\ar[r]&
\PsDiv(D_{i,\epsilon}(\Sigma)_{\C})
\end{tikzcd}
\]
commutes.
Furthermore,
we have
\[
\delta_{i,0}^*[V(f)]
=
\left\{
\begin{array}{ll}
[V(\ol{f})] & \text{if $\Cone(f,e_i)\in \Sigma$},
\\
0 & \text{otherwise}
\end{array}
\right.
\]
for $i\in I_0$ (see \eqref{ordering.5.2} for $\overline{(-)}$),
and we have
\[
\delta_{i,1}^*[V(f)]
=
\left\{
\begin{array}{ll}
[V(f)] & \text{if $f\in \im(\Z^{B-\{i\}}\to \Z^B)$},
\\
0 & \text{otherwise}
\end{array}
\right.
\]
for $i\in I_1$,
where $f\in \Sigma(1)-\{e_i:i\in I_0\}$.
\end{lem}
\begin{proof}
The uniqueness is clear since the horizontal maps are injective.
For the existence,
it suffices to note that the above formulas for $\delta_{i,0}^*$ and $\delta_{i,1}^*$ also hold as pseudo-divisors.
\end{proof}

We also need the compatibility of the pullback maps for projections as follows.

\begin{lem}
\label{subdivision.17}
Let $\Sigma$ be an $(I_0,I_1,I_2)$-admissible fans in $\Z^B$,
where $B$ is a finite set, $I_0,I_1,I_2\subset B$, and $(I_0\cup I_1)\cap I_2=\emptyset$.
Then for an indeterminate variable $t\notin B$,
there exists a unique map
\(
\pi_t^*
\colon
Z_{I_0}^1(\Sigma)
\to
Z_{I_0\cup \{t\}}^1(\Sigma\times \P^{\{t\}})
\)
such that the square
\[
\begin{tikzcd}
Z_{I_0}^1(\Sigma)\ar[d,"\pi_t^*"']\ar[r]&
\PsDiv(\Sigma_{\C})\ar[d,"\pi_t^*"]
\\
Z_{I_0\cup \{t\}}^1(\Sigma\times \P^{\{t\}})\ar[r]&
\PsDiv((\Sigma\times \P^{\{t\}})_{\C})
\end{tikzcd}
\]
commutes,
where $\pi_t\colon \Sigma\times \P^{\{t\}}\to \Sigma$ is the projection.
Furthermore,
we have
\(
\pi_t^*[V(f)]=[V(f)]
\)
for $f\in \Sigma(1)-\{e_i:i\in I_0\}$,
where $f$ in the right-hand side is regarded as the ray whose coordinate in the variable $t$ is $0$.
\end{lem}
\begin{proof}
The uniqueness is clear since the horizontal maps are injective.
For the existence,
it suffices to note that the above formula for $\pi_t^*$ also holds as pseudo-divisors.
\end{proof}

Now, let us use the above results to transform cubical identities for $\PsDiv$ to $Z_{I_0}^1$.

\begin{lem}
\label{subdivision.18}
Let $\Sigma$ be an $(I_0,I_1,I_2)$-admissible fans in $\Z^B$,
where $B$ is a finite set, $I_0,I_1,I_2\subset B$, and $(I_0\cup I_1)\cap I_2=\emptyset$.
\begin{enumerate}
\item[\textup{(1)}]
Let $g\colon \Sigma'\to \Sigma$ be a subdivision of $(I_0,I_1,I_2)$-admissible fans in $\Z^B$.
Then for $(i,\epsilon)\in I_0\times \{0\} \amalg I_1\times \{1\}$, we have
\[
\delta_{i,\epsilon}^*g^*=g^*\delta_{i,\epsilon}^*\colon Z_{I_0}^1(\Sigma)\to Z_{I_0-\{i\}}^1(D_{i,\epsilon}(\Sigma')).
\]
\item[\textup{(2)}]
For $\epsilon=0,1$ and $i,i'\in I_\epsilon$ such that $i\neq i'$,
we have
\[
\delta_{i,\epsilon}^*\delta_{i',\epsilon}^*=\delta_{i',\epsilon}^*\delta_{i,\epsilon}^*
\colon
Z_{I_0}^1(\Sigma)
\to
Z_{I_0-\{i,i'\}}^1(D_{i,\epsilon}(D_{i',\epsilon}(\Sigma))).
\]
\item[\textup{(3)}]
For $i\in I_0$ and $i'\in I_1$ such that $i\neq i'$,
we have
\[
\delta_{i,1}^*\delta_{i',0}^*=\delta_{i',0}^*\delta_{i,1}^*
\colon
Z_{I_0}^1(\Sigma)
\to
Z_{I_0-\{i,i'\}}^1(D_{i,1}(D_{i',0}(\Sigma))).
\]
\item[\textup{(4)}]
For $t\notin B$ and $\epsilon=0,1$,
we have
\[
\delta_{t,\epsilon}^*\pi_t^*=\id
\colon
Z_{I_0}^1(\Sigma)
\to
Z_{I_0}^1(\Sigma).
\]
\item[\textup{(5)}]
For $t\notin B$ and $(i,\epsilon)\in I_0\times \{0\} \amalg I_1\times \{1\}$,
we have
\[
\delta_{i,\epsilon}^*\pi_t^*=\pi_t^*\delta_{i,\epsilon}^*
\colon
Z_{I_0}^1(\Sigma)
\to
Z_{I_0-\{i\}}^{1}(D_{i,\epsilon}(\Sigma)\times \P^{\{t\}}).
\]
\end{enumerate}
\end{lem}
\begin{proof}
Using Lemmas \ref{subdivision.15}, \ref{subdivision.16}, and \ref{subdivision.17}, it suffices to check the corresponding claims for pseudo-divisors,
which is clear since the pullback maps for pseudo-divisors are compatible with compositions.
\end{proof}

\subsection{Construction of a class behaving like a homotopy}

Let us use cubical identities for $Z_{I_0}^1$ established in the previous subsection.
We first lift an element in a Chow group to an element in $\Sym^p{Z_{I_0}^1}$ as follows.

\begin{lem}
\label{subdivision.21}
Let $\Sigma$ be an $(I_0,I_1,I_2)$-admissible fan in $\Z^B$,
where $B$ is a finite set, and $I_0$, $I_1$, and $I_2$ are disjoint subsets of $B$.
If an element $x\in \CH^p(\Sigma)$ satisfies $\delta_{i,0}^* x=0$ for every $i\in I_0$ and $\delta_{i,1}^*x=0$ for every $i\in I_1$,
then there exists an $(I_0,I_1,I_2)$-admissible subdivision $f\colon \Sigma'\to \Sigma$ with an element $\widetilde{x}'\in \Sym^p Z^1_{I_0}(\Sigma')$ satisfying the following conditions:
\begin{enumerate}
\item[\textup{(i)}] The image of $\widetilde{x}'$ in $\CH^p(\Sigma')$ is equal to $f^*x$.
\item[\textup{(ii)}]
$\delta_{i,0}^*\widetilde{x}' =0$ for every $i\in I_0$ and $\delta_{i,1}^*\widetilde{x}'=0$ for every $i\in I_1$.
\end{enumerate}
\end{lem}
\begin{proof}
For a subset $I$ of $I_0\amalg I_1$, consider the following condition:
\begin{enumerate}
\item[$(\mathrm{ii}_I)$]
$\delta_{i,0}^*\widetilde{x}'=0$ for every $i\in I_0\cap I$ and $\delta_{i,1}^*\widetilde{x}'=0$ for every $i\in I_1\cap I$.
\end{enumerate}
We proceed by induction on $I$ to show that there exists such $\Sigma'$ with $\widetilde{x}'$ satisfying (i) and $(\mathrm{ii}_I)$.

By Lemma \ref{subdivision.19},
there exists $\widetilde{x}\in \Sym^p Z_{I_0}^1(\Sigma)$ whose image in $\CH^p(\Sigma)$ is $x$.
Observe that $\widetilde{x}$ satisfies (i) and $(\mathrm{ii}_{\emptyset})$.

Assume that there exists $f\colon \Sigma'\to \Sigma$ and $\widetilde{x}'\in \Sym^p Z_{I_0}^1(\Sigma')$ satisfying (i) and ($\mathrm{ii}_I$),
where $I\subset I_0\amalg I_1$.
For $\epsilon=0,1$ and $t\in (I_0\amalg I_1)-I$,
by Lemma \ref{admissible.2},
there exists an $(I_0,I_1,I_2)$-admissible fan with subdivisions $u\colon \Sigma''\to \Sigma'$ and $v\colon \Sigma''\to D_{t,\epsilon}(\Sigma')\times \P^{\{t\}}$.
Consider $\widetilde{x}'':=u^*\widetilde{x}'-v^*\pi_t^* \delta_{t,\epsilon}^*\widetilde{x}'\in \Sym^p Z_{I_0}^1(\Sigma'')$.
Using Lemma \ref{subdivision.18} and the assumption that $\widetilde{x}'$ satisfies ($\mathrm{ii}_I$),
we have $\delta_{i,0}^*\widetilde{x}''=0$ for $i\in I_0\cap I$ and $\delta_{i,1}^*\widetilde{x}''=0$ for $i\in I_1\cap I$.
Also,
Lemma \ref{subdivision.18} implies $\delta_{t,\epsilon}^*\widetilde{x}''=0$.
Hence $\widetilde{x}''$ satisfies ($\mathrm{ii}_{I\cup \{t\}}$).
To complete the induction argument,
observe that the condition (i) for $\widetilde{x}'$ implies the condition (i) for $\widetilde{x}''$ since $\delta_{t,\epsilon}^*x=0$ by assumption.
\end{proof}

Next,
we construct an element that behaves like a homotopy for $\Sym^p Z_{I_0}^1$.

\begin{lem}
\label{subdivision.22}
Let $\Sigma$ be an $(I_0,I_1,I_2)$-admissible fans in $\Z^B$,
where $B$ is a finite set, $I_0,I_1,I_2\subset B$, and $(I_0\cup I_1)\cap I_2=\emptyset$.
If an element $\widetilde{x}\in \Sym^p Z_{I_0}^{1}(\Sigma)$ satisfies $\delta_{i,0}^* \widetilde{x}=0$ for $(i,\epsilon)\in I_0\times \{0\}\amalg I_1\times \{1\}$,
then there exists an $(I_0\cup \{t\},I_1\cup \{t\},I_2)$-admissible subdivision $\Sigma'$ of $\Sigma \times \P^{\{t\}}$ with indeterminate $t\notin B$ and an element $\widetilde{y}\in \Sym^p Z_{I_0\cup \{t\}}^{1}(\Sigma')$ such that $\delta_{i,0}^* \widetilde{y}=0$ for $i\in I_0\amalg \{t\}$, $\delta_{i,1}^* \widetilde{y}=0$ for $i\in I_1$, and $\delta_{t,1}^* \widetilde{y}=\widetilde{x}$.
\end{lem}
\begin{proof}
Let $f_1,\ldots,f_m$ be the rays of $\Sigma \times \P^{\{t\}}$ different from $e_i$ for every $i\in I_0\cup I_2\cup \{t\}$.
We set
\[
\Sigma'
:=
(\cdots ((\Sigma \times \P^{\{t\}})^*(\sigma_m))^*(\sigma_{m-1})\cdots )^*(\sigma_1),
\]
where $\sigma_i:=\Cone(f_i,e_t)$.
Note that $\Sigma'$ depends on the choice of the order of $f_1,\ldots,f_m$ in general.
Observe that $\Sigma'$ is an $(I_0\cup \{t\},I_1\cup \{t\},I_2)$-admissible fan.
The key properties of $\Sigma'$ are as follows:
\begin{enumerate}
\item[(i)]
The set of cones of $\Sigma'$ contained in $\Z^B\times 0$ forms the fan $\Sigma\times 0$.
\item[(ii)]
For every cone $\sigma$ of $\Sigma$ such that $e_i\notin \sigma$ for every $i\in I_0\cup I_2$,
we have $\Cone(\sigma,e_t)\notin \Sigma'$.
\end{enumerate}

We can write
\[
\widetilde{x}=\sum_{j=1}^m [V(f_{1j})]\cdots [V(f_{pj})],
\]
where each $f_{ij}$ is a ray of $\Sigma$.
Consider
\[
\widetilde{y}=\sum_{j=1}^m [V(g_{1j})]\cdots [V(g_{pj})]\in \Sym^p Z_{I_0}^1(\Sigma'),
\]
where each $g_{ij}$ is obtained from $f_{ij}$ by adding $0$ to the $t$th coordinate.
Together with the properties (i) and (ii),
we have $\widetilde{y}\in \Sym^p_{Z_{I_0\cup \{t\}}^{ 1}}(\Sigma')$ and $\delta_{t,0}^* \widetilde{y}=0$.
The assumption $\delta_{i,\epsilon}^*\widetilde{x}=0$ for $(i,\epsilon)\in I_0\times \{0\}\amalg I_1\times \{1\}$ implies $\delta_{i,\epsilon}^*\widetilde{y}=0$ for $(i,\epsilon)\in I_0\times \{0\}\amalg I_1\times \{1\}$.
\end{proof}

\begin{exm}
To help the reader understand the proof of Lemma \ref{subdivision.22},
let us illustrate the fan $\Sigma'$ for an example of $\Sigma$ as follows.
Consider the fan $\Sigma$ obtained from $(\P^1)^2$ by star subdivisions relative to the cones $\Cone(-e_1,e_2)$ and $\Cone(-e_1,-e_2)$.
Observe that $\Sigma$ is a $\{\{2\},\{2\},\{1\}\}$-admissible fan in $\Z^2$.
We illustrate $\Sigma$ as follows:
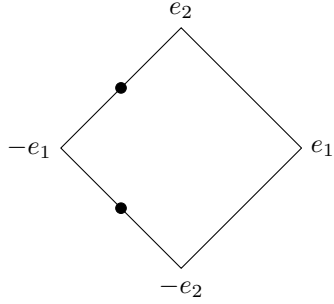
\[
\begin{tikzpicture}[scale = 0.8]
\draw (2,0)--(0,2)--(-2,0)--(0,-2)--(2,0);
\draw (2,0) node[right] {$e_1$};
\draw (0,2) node[above] {$e_2$};
\draw (-2,0) node[left] {$-e_1$};
\draw (0,-2) node[below] {$-e_2$};
\filldraw (-1,1) circle (2.5pt);
\filldraw (-1,-1) circle (2.5pt);
\end{tikzpicture}
\]

We set $f_1:=(-1,1)$, $f_2:=(-1,0)$, $f_3:=(-1,-1)$, and $f_4:=(0,-1)$.
Then we illustrate $\Sigma'$ as follows:
\[
\begin{tikzpicture}[scale = 0.8]
\draw (2,0)--(0,2)--(-2,0)--(0,-2)--(2,0)--(0,0)--(-2,0);
\draw (0,2)--(0,0)--(0,-2);
\draw (2,0) node[right] {$e_1$};
\draw (0,2) node[above] {$e_2$};
\draw (0,0) node[above right] {$e_t$};
\draw (-2,0) node[left] {$-e_1$};
\draw (0,-2) node[below] {$-e_2$};
\draw (-2,0)--(0,-1)--(2,0);
\draw (-1,-1)--(0,-1);
\draw (0,-1)--(-1,0)--(0,2);
\draw (-1,1)--(-1,0);
\begin{scope}[shift={(6,0)}]
\draw (2,0)--(0,2)--(-2,0)--(0,-2)--(2,0)--(0,0)--(-2,0);
\draw (0,2)--(0,0)--(0,-2);
\draw (0,0)--(-1,-1);
\draw (0,0)--(-1,1);
\draw (2,0) node[right] {$e_1$};
\draw (0,2) node[above] {$e_2$};
\draw (0,0) node[above right] {$-e_{t}$};
\draw (-2,0) node[left] {$-e_1$};
\draw (0,-2) node[below] {$-e_2$};
\end{scope}
\end{tikzpicture}
\]
\end{exm}

\section{Proof of Theorem \ref{intro.2}}
\label{induction}

We will finish the proof of Theorem \ref{intro.2} at the end of this section.
In \S \ref{induction2}, we will consider star subdivisions
\[
\Sigma_m\to \cdots \to \Sigma_0=\Theta_{n,r,\bd}
\]
obtained by Propositions \ref{dist.10} and \ref{ordering.26} such that $\Sigma_m$ refines a given fan $\Sigma$.
We will interpret Theorem \ref{intro.2} as a claim for an individual standard subdivision $\Sigma$.
Then Proposition \ref{ordering.15} is the base case of Theorem \ref{intro.2},
and hence it will suffice to deduce the claim for $\Delta_{i+1}$ from the claim for $\Delta_i$.

In \S \ref{blow}, we will discuss several blow-up squares and their Chow groups.
In \S \ref{Valpha}, we will complete the proof of this induction step.
Lemmas \ref{subdivision.21} and \ref{subdivision.22} are crucial ingredients here.

\subsection{Induction argument}
\label{induction2}

Let us first explain how we can interpret Theorem \ref{intro.2} for an individual standard subdivision $\Sigma$.

\begin{df}
\label{subdivision.3}
Let $0\leq r\leq n$ and $p>0$ be integers.
For $x\in \CH_\sta^p(n,r)$ such that $\delta_{i,\epsilon}^*x=0$ for $r+1\leq i\leq n$ and $\epsilon=0,1$,
we will consider the following conditions:
\begin{itemize}
\item[($*$)]
There exists $y\in \CH_\sta^p(n+1,r)$ such that $\delta_{i,0}^*y=0$ for $r+1\leq i\leq n+1$, $\delta_{i,1}^*y=0$ for $r+1\leq i\leq n$,
and $\delta_{n+1,1}^*y=x$.
\end{itemize}
We say that an $r$-standard subdivision $\Sigma$ of $(\P^1)^n$ satisfies $(*)$ if such an element $x$ satisfies $(*)$ whenever $x$ is in the image of $\CH^p(\Sigma)\to \CH_\sta^p(n,r)$.
\end{df}

\begin{prop}
\label{subdivision.38}
Under the above notation,
the map $\CH^p(\Sigma)\to \CH_\sta^p(n,r)$ is injective.
\end{prop}
\begin{proof}
By Proposition \ref{subdivision.26} below,
it suffices to show that $\Sta_{n,r}$ is cofiltered.

Every Hom set in $\Sta_{n,r}$ is $\emptyset$ or $*$.
Hence it suffices to show that $\Sta_{n,r}$ is connected.
For this,
let $\Delta,\Delta'\in \Sta_{n,r}$.
The fan
\[
\Delta'':=\{\delta\cap \delta': \delta\in \Delta,\delta'\in \Delta'\}
\]
satisfies the conditions (i) and (ii) in Definition \ref{intro.1}, but $\Delta''$ does not need to be smooth.
The toric resolution of singularities \cite[Theorem 11.1.9]{CLStoric} yields a subdivision $\Delta'''\to \Delta''$ such that $\Delta'''$ is an $r$-standard subdivision of $(\P^1)^n$.
\end{proof}

\begin{prop}
\label{subdivision.26}
Let $\Sigma'\to \Sigma$ be a subdivision of fans.
Then the induced map $\CH^p(\Sigma)\to \CH^p(\Sigma')$ is injective for every integer $p\geq 0$.
\end{prop}
\begin{proof}
By \cite[Theorem 2.4]{MR803344} (see also \cite[pp.\ 39-40]{TOda}),
there exists a subdivision of fans $\Sigma''\to \Sigma'$ such that $\Sigma''\to \Sigma$ is the composition of a sequence of star subdivisions.
Hence we reduce to the case where $\Sigma'\to \Sigma$ is a star subdivision.
Then the morphism of $\C$-realizations $\Sigma_{\C}'\to \Sigma_\C$ is the blow-up along a smooth center,
so the blow-up formula for Chow groups \cite[Theorem 3.3(b), Proposition 6.7(e)]{Fulton} finishes the proof.
\end{proof}

The statement of the following result is a part of the general theory of cubical abelian groups, and we write down the proof in detail.

\begin{lem}
\label{subdivision.33}
Let $0\leq r\leq n$ and $p>0$ be integers,
and let $x$ be an element of $\CH_\sta^p(n,r)$ such that $\delta_{i,\epsilon}^*x=0$ for $r+1\leq i\leq n$ and $\epsilon=0,1$.
If there exists $y\in \CH_\sta^p(n+1,r)$ such that $\delta_{i,0}^*y=0$ for $r+1\leq i\leq n+1$ and $\delta^*y=x$,
then $x$ satisfies $(*)$.
\end{lem}
\begin{proof}
For every integer $r\leq t\leq n$,
we claim that exists $y_t\in \CH_\sta^p(n+1,r)$ such that
$\delta_{i,0}^*y=0$ for $r+1\leq i\leq n+1$,
$\delta_{i,1}^*y=0$ for $r+1\leq i\leq t$,
and $\delta_{n+1,1}^*y=x$.
We proceed by induction on $t$.
The claim is the assumption if $t=r$.
Assume that the claim holds for $t-1$.
Let us argue as in \cite[Proposition A.11]{MR3259031}.
Consider $y_t:=y_{t-1}-\mu_t^*\delta_{t,1}^*y_{t-1}$,
see Definition \ref{intro.5} for $\mu_t^*$.
Using \eqref{intro.5.1}, \eqref{intro.5.2}, and \eqref{intro.5.3},
we have
\[
\delta_{i,1}^*y_t=
\left\{
\begin{array}{ll}
0 & \text{if $i\leq t$},
\\
\delta_{t+1,1}^*y_{t-1}-\delta_{t,1}^*y_{t-1}
&
\text{if $i=t+1$},
\\
\delta_{i,1}^*y_{t-1}-\mu_t^*\delta_{t,1}^*\delta_{i,1}^*y_{t-1}
&
\text{if $i>t+1$}.
\end{array}
\right.
\]
We also have $\mu_t^*\delta_{t,1}^*\delta_{t,1}^*y_{t-1}=\mu_t^*\delta_{t,1}^*\delta_{t+1,1}^*y_{t-1}$.
Hence we have
\[
\delta^* y_t
=
\delta^* y_{t-1} + \mu_t^* \delta_{t,1}^* \delta^* y_{t-1},
\]
which implies $\delta^*y_t=x$ since $\delta_{t,1}^*\delta^*y_{t-1}=\delta_{t,1}^*x=0$.
This completes the induction argument.
\end{proof}

The following result serves as the base step of Theorem \ref{intro.2}.

\begin{lem}
\label{subdivision.4}
For integers $0\leq r\leq n$ and nonempty finite sequence of positive integers $\bd$,
$\Theta_{n,r,\bd}$ satisfies $(*)$.
\end{lem}
\begin{proof}
If $x\in \CH^p(\Theta_{n,r,\bd})$ with $p>0$ satisfies $\delta_{i,\epsilon}^*x=0$ for $r+1\leq i\leq n$ and $\epsilon=0,1$,
then we can regard $x$ as an element of $\CH_{n-p}^\flat(\Theta_{n,r,\bd})$ by Proposition \ref{ordering.3}.
By Proposition \ref{ordering.15},
there exists $y\in \CH_{n+1-p}^\flat(\Theta_{n+1,r,\bd})$ such that $\delta^* y=x$.
We can regard $y$ as an element of $\CH^p(\Theta_{n+1,r,\bd})$.
Lemma \ref{subdivision.33} finishes the proof.
\end{proof}

\begin{lem}
\label{subdivision.11}
Let $\Sigma$ be a very $r$-standard subdivision of $\Gamma_{n,r}$, where $0\leq r\leq n$ are integers,
and let $\sigma$ be a $2$-dimensional cone of $\Sigma$ such that $e_1,\ldots,e_n \notin \sigma$.
Assume also that the star subdivision $\Sigma'$ of $\Sigma$ relative to $\sigma$ is a very $r$-standard subdivision of $(\P^1)^n$,
i.e.,
$\sigma$ is contained in $(-\N)^r\times \Z^{n-r}$.
If $\Sigma$ satisfies $(*)$,
then $\Sigma'$ satisfies $(*)$.
\end{lem}

We finish its proof at the end of this section.

\begin{lem}
\label{subdivision.35}
Assume \textup{Lemma \ref{subdivision.11}}.
Then \textup{Theorem \ref{intro.2}} holds.
\end{lem}
\begin{proof}
Let $r\in \N$.
Recall the complex
\[
\cdots
\xrightarrow{\delta^*}
\CH_\sta^{p,\flat}(r+1,r)\xrightarrow{\delta^*}
\CH_\sta^{p,\flat}(r,r).
\]
in Definition \ref{intro.5}.
If $p=0$,
then the complex becomes
\[
\cdots \xrightarrow{0} \Z \xrightarrow{1} \Z \xrightarrow{0} \Z,
\]
which is quasi-isomorphic to $\Z\cong \CH^0(\C)$.

Assume $p>0$.
We need to show that every $x\in \CH_\sta^{p,\flat}(n,r)$ satisfies $(*)$.
We may assume that $x$ is in the image of $\CH^p(\Sigma)\to \CH_\sta^p(n,r)$ for some $r$-standard subdivision $\Sigma$ of $(\P^1)^n$.
By Propositions \ref{dist.10} and \ref{ordering.26},
there exists a sequence of star subdivisions 
\[
\Sigma_m\to \cdots \to \Sigma_0=\Theta_{n,r,\bd}
\]
for some nonempty finite sequence $\bd$ of positive integers such that $\Sigma_m$ is a subdivision of $\Sigma$ and for each $i$,
$\Sigma_{i+1}=\Sigma_i^*(\sigma_i)$ for some $2$-dimensional cone $\sigma_i$ satisfying $e_1,\ldots,e_n\notin \sigma_i$.
Since $\Sigma_m$ is a subdivision of $\Sigma$,
$x$ is in the image of $\CH^p(\Sigma_m)\to \CH_\sta^p(n,r)$.
Hence it suffices to show that $\Sigma_m$ satisfies $(*)$.
This is a consequence of Lemmas \ref{subdivision.4} and \ref{subdivision.11}.
\end{proof}

\S \ref{blow} and \ref{Valpha} are devoted to the proof of Lemma \ref{subdivision.11}.

\subsection{Chow groups of blow-up squares}
\label{blow}

Let $X$ be a smooth separated scheme over a field $k$, and let $Z$ be a codimension $2$ smooth closed subscheme of $X$.
Consider the induced cartesian square of schemes
\[
\begin{tikzcd}
\Bl_Z X\times_X Z\ar[d,"p'"']\ar[r,"i'"]&
\Bl_Z X\ar[d,"p"]
\\
Z\ar[r,"i"]&
X.
\end{tikzcd}
\]
By \cite[Theorem 3.3(b), Proposition 6.7(e)]{Fulton},
for every integer $d\geq 0$,
we have the blow-up formula:
\begin{equation}
\label{subdivision.8.1}
p^*+i_*'p'^*
\colon
\CH_d(X)
\oplus
\CH_{d-1}(Z)
\xrightarrow{\cong}
\CH_d(\Bl_Z X).
\end{equation}

\begin{lem}
\label{subdivision.8}
Under the above notation,
let $Y$ be a codimension $1$ smooth closed subscheme of $X$ that has strict normal crossing with $Z$ in the sense of \cite[Definition 7.2.1]{BPO} and does not contain $Z$ as a closed subscheme.
We set $W:=Y\times_X Z$.
Then the square
\[
\begin{tikzcd}
\CH_d(X)\oplus \CH_{d-1}(Z)\ar[r,"\cong"]\ar[d,"{(f^*, g^*)}"']&
\CH_d(\Bl_Z X)\ar[d,"f'"]
\\
\CH_d(Y)\oplus \CH_{d-1}(W)\ar[r,"\cong"]&
\CH_d(\Bl_W Y)
\end{tikzcd}
\]
commutes for every $d\in \N$,
where the horizontal isomorphisms are obtained by \eqref{subdivision.8.1},
and $f\colon Y\to X$, $g\colon W\to Z$, and $f'\colon \Bl_W Y\to \Bl_Z X$ are the induced morphisms.
\end{lem}
\begin{proof}
Consider the induced square of schemes
\[
\begin{tikzcd}
\Bl_W Y\times_Y W\ar[d,"g'"']\ar[r,"j'"]&
\Bl_W Y\ar[d,"f'"]
\\
\Bl_Z X\times_X Z\ar[r,"i'"]&
\Bl_Z X,
\end{tikzcd}
\]
which is cartesian since $Y$ has strict normal crossing with $Z$.
We only need to show $f'^*i_*'=j_*'g'^*$,
but this is a consequence of the base change formula (a special case of the excess intersection formula \cite[Proposition 6.6(3)]{Fulton}).
\end{proof}

Next,
we will explain several basic examples of subschemes of toric varieties that have strict normal crossing with some divisors.

\begin{lem}
\label{subdivision.9}
Let $\Sigma$ be a smooth fan in $\Z^n$ with a ray $f$.
If $\sigma$ is a $2$-dimensional cone of $\Sigma$ such that $f\notin \sigma$,
then the divisor $V(f)_{\C}$ on $\Sigma_{\C}$ has strict normal crossing with $V(\sigma)_{\C}$.
\end{lem}
\begin{proof}
If $\Cone( \sigma,f)\notin \Sigma$,
then the claim is clear since $V(f)_{\C}\cap V(\sigma)_{\C}=\emptyset$.
Hence assume $\Cone( \sigma,f)\in \Sigma$.
We can work locally on $\Sigma$,
so we may assume $\Sigma=\A^n$, $\sigma=\A^2$, $f=e_n$, and $n\geq 3$,
and then we have $\Sigma_{\C}\cong \Spec(\C[x_1,\ldots,x_n])$.
To conclude,
observe that $V(f)_{\C}$ (resp.\ $V(\sigma)_{\C}$) is the closed subscheme defined by the ideal $(x_n)$ (resp.\ $(x_1,x_2)$).
\end{proof}

\begin{rmk}
We will use Lemma \ref{subdivision.9} for the situation when $\Sigma$ is an $r$-standard subdivision of $(\P^1)^n$ and $f=e_i$ with an integer $r+1\leq i\leq n$.
\end{rmk}

\begin{lem}
\label{subdivision.10}
Let $\Sigma$ be a smooth subdivision of $(\P^1)^n$,
where $n\in \N$.
If $1\leq i\leq n$ and $\sigma$ is a $2$-dimensional cone of $\Sigma$ contained in $\Z^{i-1}\times 0 \times \Z^{n-i}$ such that $\Sigma$ is $i$-admissible in the sense of \textup{Construction \ref{subdivision.2}},
then the divisor $D_{i,1}(\Sigma)_{\C}$ on $\Sigma_{\C}$ is strict normal crossing with $V(\sigma)_{\C}$.
\end{lem}
\begin{proof}
The morphism of fans $D_{i,1}(\Sigma)\to \Sigma$ factors through the subfan of $\Sigma$ consisting of $\tau$ such that $\dim \tau\cap (\Z^{i-1}\times 0 \times \Z^{n-i})=n-1$.
Hence for every open subscheme $U$ of $\Sigma_{\C}$ corresponding to such $\tau$,
we only need to show that $D_{i,1}(\Sigma)_{\C}\cap U$ is strict normal crossing with $V(\sigma)_{\C}\cap U$.

We have $\tau=\Cone( f_1,\ldots,f_n)$ for some $f_1,\ldots,f_{n-1}\in \Z^{i-1}\times 0 \times \Z^{n-i}$ and $f_n\in \Z^n$ such that $\Cone( f_1,f_2) = \sigma$.
Then we have $U\cong \Spec(\C[x_1,\ldots,x_n])$,
$D_{i,1}(\Sigma)_{\C}\cap U\cong \Spec(\C[x_1,\ldots,x_n]/(x_n-1))$, and $V(\sigma)_{\C}\cap U\cong \Spec(\C[x_1,\ldots,x_n]/(x_1,x_2))$.
We conclude using these descriptions.
\end{proof}

\begin{lem}
\label{subdivision.20}
Let $\Sigma$ be a smooth fan,
let $\sigma$ and $\tau$ be its two cones such that $\sigma\cap \tau=0$,
and let $\Sigma'$ be the star subdivision of $\Sigma$ relative to $\tau$.
Consider the cone $\sigma'$ of $\Sigma$ corresponding to $\sigma$.
Then the induced square of schemes
\[
\begin{tikzcd}
V(\sigma')_{\C}\ar[d,"f'"']\ar[r,"i'"]&
\Sigma_{\C}'\ar[d,"f"]
\\
V(\sigma)_{\C}\ar[r,"i"]&
\Sigma_{\C}
\end{tikzcd}
\]
is cartesian.
\end{lem}
\begin{proof}
We set $X:=\Sigma_{\C}$, $Z:=V(\sigma)_{\C}$, and $W:=V(\tau)_{\C}$.
Then we have $\Sigma_{\C}'\cong \Bl_W X$ and $V(\sigma')_{\C}\cong \Bl_{Z\cap W}Z$.
Since $\sigma\cap \tau=0$, $Z$ and $W$ are transversal in $X$.
Hence we have $\Bl_W X\times_X Z \cong \Bl_{Z\cap W}Z$,
which shows that the above square is cartesian.
\end{proof}

\begin{lem}
\label{subdivision.28}
Let $\Sigma$ be a fan, and let $\sigma$ be its $2$-dimensional cone with rays $\alpha$ and $\beta$.
Then the induced square of $\C$-realizations
\[
\begin{tikzcd}
V(\sigma)_{\C}\ar[d]\ar[r]&
V(\alpha)_{\C}\ar[d]
\\
V(\beta)_{\C}\ar[r]&
\Sigma_{\C}
\end{tikzcd}
\]
is cartesian.
\end{lem}
\begin{proof}
We can work locally on $\Sigma$,
so we may assume $\Sigma=\A^n$, $\alpha=e_1$, and $\beta=e_2$ for some integer $n\geq 2$.
We have the isomorphisms
\begin{gather*}
\Sigma_{\C}\cong \Spec(\C[x_1,\ldots,x_n]),
\text{ }V(\alpha)_{\C}\cong \Spec(\C[x_1,\ldots,x_n]/(x_1)),
\\
V(\beta)_{\C}\cong \Spec(\C[x_1,\ldots,x_n]/(x_2)),
\text{ }V(\sigma)_{\C}\cong \Spec(\C[x_1,\ldots,x_n]/(x_1,x_2)).
\end{gather*}
We conclude from this.
\end{proof}

\subsection{Study of \texorpdfstring{$V(\alpha)$}{V(alpha)}}
\label{Valpha}

Let us first explain how the notion of very standard subdivisions is connected to the notion of admissible subdivisions via $V(\alpha)$ for rays $\alpha$ of a given very standard subdivision.

\begin{const}
\label{subdivision.27}
Let $\Sigma$ be a very $r$-standard subdivision of $\Gamma_{n,r}$,
where $0\leq r\leq n$ are integers,
and let $\alpha:=(\alpha_1,\ldots,\alpha_n)$ be a ray of $\Sigma$ such that $\alpha\neq e_1,\ldots,e_n$ and $\alpha_1,\ldots,\alpha_r\leq 0$.
Consider the sets
\begin{gather*}
I_0:=\{r+1\leq i\leq n:\Cone( \alpha,e_i) \in \Sigma\},
\\
I_1:=\{r+1\leq i\leq n:\alpha_i=0\},
\text{ }
I_2:=\{1\leq i\leq r: \alpha_i=0\}
\end{gather*}
and the quotient map $\ol{(-)}\colon N(\Sigma)\to N(\alpha)$ in \eqref{ordering.5.2}.
Observe that $I_0\cap I_1=\emptyset$ by Lemma \ref{ordering.12}.
We also have $I_0\cap I_2=\emptyset$ and $I_1\cap I_2=\emptyset$.
Since $\Sigma$ is a very $r$-standard subdivision,
there exists no ray $\beta:=(\beta_1,\ldots,\beta_n)$ of $\Sigma$ such that $\beta\neq e_i$ for all $1\leq i\leq n$ and $\beta_i>0$ for some $1\leq i\leq r$.
This implies that we have
\[
I_2=\{1\leq i\leq r:\Cone(\alpha,e_i)\in \Sigma\}.
\]

The lattice $N(\alpha)$ is the quotient of $N(\Sigma)=\Z^n$ under the relation $\alpha_1e_1+\cdots+\alpha_ne_n=0$ or equivalently $\sum_{i\in \{1,\ldots,n\}-(I_1\cup I_2)} \alpha_i e_i=0$.
Since $\Sigma$ is an $r$-standard subdivision,
there exists an integer $r+1\leq u\leq n$
such that $\alpha_u<0$,
which implies $\Cone(\alpha,e_u) \notin \Sigma$ and hence $u\notin I_0$.
We also have $u\notin I_1,I_2$.
Consider the set $B:=\{1,\ldots,n\}-\{u\}$.
Then $\{\ol{e_i}:i\in B\}$ is a basis of $N(\alpha)$,
and we have $I_0,I_1,I_2\subset B$.

Since $\Sigma$ is complete and smooth,
$V(\alpha)$ satisfies the condition (i) in Definition \ref{admissible.1}.
Since $\Sigma$ is a very $r$-standard subdivision and $\Cone(\alpha,e_i)\in \Sigma$ for all $i\in I_0$,
we have $\Cone(\alpha,e_{a_1},\ldots,e_{a_s})\in \Sigma$ if $I_0=\{a_1,\ldots,a_s\}$.
Hence $V(\alpha)$ satisfies the condition (ii) in Definition \ref{admissible.1}.

For $i\in I_1$,
we have $D_{i,1}(V(\alpha))=V(\alpha')$ if $\alpha'$ is the ray of $D_{i,1}(\Sigma)$ corresponding to $\alpha$.
Since $\Sigma$ satisfies the condition (iii) in Definition \ref{admissible.1},
$D_{i,1}(\Sigma)$ is complete.
Hence $V(\alpha')$ is complete too,
so $V(\alpha)$ satisfies the condition (iii) in Definition \ref{admissible.1}.

Let $i\in I_2$.
Assume that $\ol{f}$ is a ray of $V(\alpha)$ whose $i$th coordinate is $>0$.
There exists a ray $f$ of $\Sigma$ whose image in $V(\alpha)$ is $\ol{f}$,
and the $i$th coordinate of $f$ is $>0$.
The condition (iv) in Definition \ref{admissible.1} for $\Sigma$ implies $f=e_i$.
Hence $V(a)$ satisfies the condition (iv) in Definition \ref{admissible.1}.

We have shown that $V(\alpha)$ is an $(I_0,I_1,I_2)$-admissible fan in $\Z^B$.
\end{const}

Suppose that a very standard subdivision $\Sigma$ and its ray $\alpha$ are given.
If we apply Lemmas \ref{subdivision.21} or \ref{subdivision.22} to $V(\alpha)$,
then we get an admissible subdivision $\Delta$ of $V(\alpha)$ or $V(\alpha)\times \P^1$.
The next result allows us to find a subdivision $\Sigma'$ of $\Sigma$ or $\Sigma\times \P^1$ such that $\Sigma_{\C}'$ contains $\Delta_{\C}$ as a closed subscheme.

\begin{lem}
\label{subdivision.24}
Let $\Sigma$ be a very $r$-standard subdivision of $(\P^1)^n$,
where $0\leq r\leq n$ are integers,
and let $\alpha$ be a ray of $\Sigma$ such that $\alpha \neq e_1,\ldots,e_n$.
Consider $I_0$, $I_1$, and $I_2$ for $V(\alpha)$ in \textup{Construction \ref{subdivision.27}},
and let $\Delta$ be an $(I_0,I_1,I_2)$-admissible subdivision of $V(\alpha)$.
Then there exists a very $r$-standard subdivision $\Sigma'$ of $\Sigma$ such that $V(\alpha')$ is a subdivision of $\Delta$ and the induced square of $\C$-realizations
\[
\begin{tikzcd}
V(\alpha')_{\C}\ar[d]\ar[r]&
\Sigma_{\C}'\ar[d]
\\
V(\alpha)_{\C}\ar[r]&
\Sigma_{\C}
\end{tikzcd}
\]
is cartesian,
where $\alpha'$ is the ray of $\Sigma'$ corresponding to $\alpha$.
\end{lem}
\begin{proof}
After subdividing $\Delta$ further,
by Lemma \ref{subdivision.25},
we may assume that we have a sequence
\[
\Delta=:\Delta_m\to \cdots \to \Delta_0:=V(\alpha)
\]
such that for each $j$, $\Delta_j$ is the star subdivision of $\Delta_{j-1}$ relative to a $2$-dimensional cone $\tau_{j-1}\in \Delta_{j-1}$ such that $e_i\notin \tau_{j-1}$ for $i\in I_2$ and $\tau_{j-1}\not\subset \Cone( e_{a_1},\ldots,e_{a_u})$ if $I_0=\{a_1,\ldots,a_u\}$.

Let us inductively construct a very $r$-standard subdivision $\Sigma_j$ of $\Sigma$ such that $V(\alpha_j)\cong \Delta_j$ and the induced square of $\C$-realizations
\[
\begin{tikzcd}
V(\alpha_j)_{\C}\ar[d]\ar[r]&
(\Sigma_j)_{\C}\ar[d]
\\
V(\alpha)_{\C}\ar[r]&
\Sigma_{\C}
\end{tikzcd}
\]
is cartesian,
where $\alpha_j$ is the ray of $\Sigma_j$ corresponding to $\alpha$.
We set $\Sigma_0:=\Sigma$.
When $\Sigma_{j-1}$ is given,
consider the unique $2$-dimensional cone $\sigma_{j-1}$ of $\Sigma_{j-1}$ such that $\Cone( \alpha,\sigma_{j-1}) \in \Sigma_{j-1}$ and $\ol{\sigma_{j-1}}=\tau_{j-1}$,
see \eqref{ordering.5.2} for $\overline{(-)}$.
If $e_i\in \sigma_{j-1}$ for some $1\leq i\leq r$,
then we have $e_i\in \tau_{j-1}$,
which is a contradiction.
Hence $e_i\notin \sigma_{j-1}$ for all $1\leq i\leq r$.
If $\sigma_{j-1}\subset \Cone( e_{r+1},\ldots,e_n)$,
then $\tau_{j-1}\subset \Cone( e_{a_1},\ldots,e_{a_s})$,
which is a contradiction.
Hence $\sigma_{j-1}\not \subset \Cone( e_{r+1},\ldots,e_n)$.

Let $\Sigma_j$ be the star subdivision of $\Sigma_{j-1}$ relative to $\sigma_{j-1}$,
which is a very $r$-standard subdivision by Proposition \ref{ordering.18}.
Lemma \ref{subdivision.20} implies that the induced square of $\C$-realizations
\[
\begin{tikzcd}
V(\alpha_j)_{\C}\ar[d]\ar[r]&
(\Sigma_j)_{\C}\ar[d]
\\
V(\alpha_{j-1})_{\C}\ar[r]&
(\Sigma_{j-1})_{\C}
\end{tikzcd}
\]
is cartesian.
We conclude by setting $\Sigma':=\Sigma_m$.
\end{proof}

\begin{proof}[Proof of Lemma \ref{subdivision.11}]
Let $\alpha$ be the center of $\sigma$,
which is a ray of $\Sigma'$.
Consider $\Delta:=V(\sigma)$ and $\Delta':=V(\alpha)$.
Assume that an element $x'\in \CH^*(\Sigma')$ satisfies
\[
\delta_{i,\epsilon}^*x'=0
\text{ for $r+1\leq i\leq n$ and $\epsilon=0,1$}
\]
in $\CH_\sta^*(n-1,r)$ and hence in $\CH^*(D_{i,\epsilon}(\Sigma))$ by Proposition \ref{subdivision.38}.
Consider $(x,z)\in \CH^*(\Sigma)\oplus \CH^*(\Delta)$ corresponding to $x'\in \CH^*(\Sigma')$ via the blow-up formula \eqref{subdivision.8.1},
and consider $z':=p^*z$,
where $p\colon \Delta'\to \Delta$ is the induced morphism.
Observe that $\Delta'$ is an $(I_0,I_1,I_2)$-admissible fan,
where $I_0$, $I_1$, and $I_2$ are in Construction \ref{subdivision.27}.

Using Lemmas \ref{subdivision.8}, \ref{subdivision.9}, and \ref{subdivision.10},
we have
\begin{gather}
\label{subdivision.11.1}
\text{$\delta_{i,0}^*z'=0$ for $i\in I_0$ and $\delta_{i,1}^*z'=0$ for $i\in I_1$,}
\\
\delta_{i,\epsilon}^*x=0\text{ for $r+1\leq i\leq n$ and $\epsilon=0,1$.}
\end{gather}
By assumption, $x$ satisfies $(*)$.
Hence we reduce to the case where $x=0$.
Then we have $x'=\iota_*' z'$,
where $\iota'\colon \Delta'\to \Sigma'$ is the induced morphism.

Lemmas \ref{subdivision.21} and \ref{subdivision.22} imply that there exists an $(I_0\amalg \{n+1\},I_1\amalg \{n+1\},I_2)$-admissible subdivision $\Delta_1'$ of $\Delta'\times \P^{\{n+1\}}=\Delta'\times \P^1$ with $z_1'\in \CH^*(\Delta')$ such that
\[
\text{$\delta_{i,0}^*z_1'=0$ for $i\in I_0\amalg \{n+1\}$ and $\delta_{i,1}^*z_1'=0$ for $i\in I_1$}
\]
and $\delta_{n+1,1}^*z_1'$ is equal to the image $z''$ of $z'$ in $\CH^*(D_{n+1,1}(\Delta_1'))$.
By Lemma \ref{subdivision.24},
we may assume that there exists a very $r$-standard subdivision $\Sigma_1'$ of $\Sigma'\times \P^1$ such that $\Delta_1'=V(\alpha_1')$ and the induced square of $\C$-realizations
\[
\begin{tikzcd}
(\Delta_1')_{\C}\ar[d]\ar[r]&
(\Sigma_1')_{\C}\ar[d]
\\
(\Delta'\times \P^1)_{\C}\ar[r]&
(\Sigma'\times \P^1)_{\C}
\end{tikzcd}
\]
is cartesian,
where $\alpha_1'$ is the ray of $\Sigma_1'$ whose first $n$-coordinates correspond to $\alpha'$ and $(n+1)$th coordinate is $0$.
We also have the induced cartesian squares
\[
\begin{tikzcd}
D_{n+1,1}(\Delta_1')_{\C}\ar[r]\ar[d]&
D_{n+1,1}(\Sigma_1')_{\C}\ar[d]
\\
(\Delta_1')_{\C}\ar[r]&
(\Sigma_1')_{\C},
\end{tikzcd}
\text{ }
\begin{tikzcd}
(\Delta'\times \P^1)_{\C}\ar[d]\ar[r]&
(\Sigma'\times \P^1)_{\C}\ar[d]
\\
\Delta_{\C}'\ar[r]&
\Sigma_{\C}',
\end{tikzcd}
\]
where the morphism $(\Sigma'\times \P^1)_{\C}\to \Sigma_{\C}'$ is the projection.
Hence the induced square of the $\C$-realizations
\[
\begin{tikzcd}
D_{n+1,1}(\Delta_1')_\C\ar[d]\ar[r,"\iota_{n+1,1}'"]&
D_{n+1,1}(\Sigma_1')_\C\ar[d]
\\
\Delta_{\C}'\ar[r,"\iota'"]&
\Sigma_{\C}'
\end{tikzcd}
\]
is cartesian.
Together with the base change formula \cite[Proposition 6.6(3)]{Fulton},
we see that $\iota_{n+1,1*}' z''$ is equal to the image of $\iota_*' z'$ in $\CH^*(D_{n+1,1}(\Sigma_1'))$.

Consider the induced cartesian square of $\C$-realizations,
\[
\begin{tikzcd}
D_{i,\epsilon}(\Delta_1')_{\C}\ar[r,"\iota_{i,\epsilon}'"]\ar[d,"\delta_{i,\epsilon}"']&
D_{i,\epsilon}(\Sigma_1')_{\C}\ar[d,"\delta_{i,\epsilon}"]
\\
(\Delta_1')_{\C}\ar[r,"\iota_1'"]&
(\Sigma_1')_{\C},
\end{tikzcd}
\]
where $D_{i,0}(\Delta_1'):=\emptyset$ if $i\in \{r+1,\ldots,n\}-I_0$ and $D_{i,1}(\Delta_1'):=\emptyset$ if $i\in \{r+1,\ldots,n\}-I_1$.
For $\epsilon=0$,
we need \ref{subdivision.28} for the claim that this square is cartesian.
By the base change formula \cite[Proposition 6.6(3)]{Fulton},
we have
\begin{align*}
\delta_{i,0}^*\iota_{1*}'z_1' &=
\left\{
\begin{array}{ll}
\iota_{i,0*}'\delta_{i,0}^*z_1'
&
\text{if }i\in I_0\amalg \{n+1\},
\\
0 & \text{otherwise},
\end{array}
\right.
\\
\delta_{i,1}^*\iota_{1*}'z_1' &=
\left\{
\begin{array}{ll}
\iota_{i,1*}'\delta_{i,1}^*z_1'
&
\text{if }i\in I_1\amalg \{n+1\},
\\
0 & \text{otherwise}.
\end{array}
\right.
\end{align*}
Using the previous computation on $\delta_{i,\epsilon}^*z_1'$,
we have
\[
\delta_{i,0}^* \iota_{1*}'z_1'=0
\text{ for $r+1\leq i\leq n+1$ and }
\delta_{i,1}^* \iota_{1*}'z_1'=0
\text{ for $r+1\leq i\leq n$},
\]
and $\delta_{n+1,1}^*\iota_{1*}'z_1'$ is equal to $\iota_{n+1,1*}' z''$.
We have shown that $\iota_{n+1,1*}' z''$ is equal to the image of $\iota_*' z'$ in $\CH^*(D_{n+1,1}(\Sigma_1'))$.
This implies that $\iota_*' z'=x'$ satisfies $(*)$.
\end{proof}

We complete the proof of Theorem \ref{intro.2} together with Lemma \ref{subdivision.35}.

\subsection*{Acknowledgement}

This research was conducted in the framework of the DFG-funded research training group GRK 2240: \emph{Algebro-Geometric Methods in Algebra, Arithmetic and Topology}.
We are grateful to the referee for the careful reading and constructive suggestions, which led to many improvements of the manuscript.

\bibliography{bib}
\bibliographystyle{siam}

\end{document}